\documentclass[11pt, oneside,dvipsnames]{amsart}

\usepackage{tikz-cd}
\usepackage[framemethod=tikz]{mdframed}
\usepackage{enumerate}
\usepackage{setspace}
\usepackage[english]{babel}
\usepackage{mathtools}
\usepackage[margin=1in]{geometry}
\usepackage{amssymb}
\usepackage[alphabetic]{amsrefs}
\usepackage{amsmath}
\usepackage{graphicx}
\usepackage{extarrows}
\usepackage{microtype}
\usepackage{graphicx}
\graphicspath{ {./pics/} }
\usepackage[colorlinks=true, allcolors=teal]{hyperref}
\usepackage{amsthm}
\usepackage{indentfirst}
\usepackage{import}
\usepackage{xifthen}
\usepackage{pdfpages}
\usepackage{transparent}
\usepackage{changepage}
\usepackage{wrapfig}
\hypersetup{
	colorlinks=true,
	linkcolor=Blue,
	filecolor=red,      
	urlcolor=red,
	citecolor=teal,
}

\newcommand{\of}{\circ}
\renewcommand{\to}{\rightarrow}

\newcommand{\Z}{\mathbb{Z}}

\newcommand{\C}{\mathbb{C}}
\newcommand{\R}{\mathbb{R}}
\newcommand{\del}{\partial}
\newcommand{\D}{\mathbb{D}}

\newcommand{\delbar}{\overline{\partial}}

\newcommand{\Hom}{\text{Hom}}

\newcommand{\sub}{\subset}

\newcommand{\inv}{^{-1}}

\newcommand{\Maps}{\text{Maps}}
\newcommand{\interval}{[0,1]}
\newcommand{\wt}{\widetilde}

\renewcommand{\to}{\rightarrow}

\newcommand{\inn}{{in}}
\newcommand{\out}{{out}}
\newcommand{\outp}{{out+}}
\newcommand{\outm}{{out-}}
\newcommand{\ol}{\overline{}}
\newcommand{\M}{\mathcal{M}}
\newcommand{\J}{\mathcal{J}}
\DeclareMathOperator{\Ind}{Ind}
\DeclareMathOperator{\vir}{vir}
\DeclareMathOperator{\aut}{aut}
\DeclareMathOperator{\im}{im}

\newcommand{\bC}{\mathbb{C}}
\newcommand{\bD}{\mathbb{D}}

\newcommand{\bH}{\mathbb{H}}

\newcommand{\bL}{\mathbb{L}}
\newcommand{\bM}{\mathbb{M}}
\newcommand{\bN}{\mathbb{N}}

\newcommand{\bP}{\mathbb{P}}

\newcommand{\bR}{\mathbb{R}}

\newcommand{\bZ}{\mathbb{Z}}
\newcommand{\cA}{\mathcal{A}}
\newcommand{\cB}{\mathcal{B}}

\newcommand{\cE}{\mathcal{E}}
\newcommand{\cF}{\mathcal{F}}

\newcommand{\cI}{\mathcal{I}}
\newcommand{\cJ}{\mathcal{J}}

\newcommand{\cL}{\mathcal{L}}
\newcommand{\cM}{\mathcal{M}}
\newcommand{\cN}{\mathcal{N}}
\newcommand{\cO}{\mathcal{O}}
\newcommand{\cP}{\mathcal{P}}

\newcommand{\cR}{\mathcal{R}}

\newcommand{\cU}{\mathcal{U}}

\newcommand{\cW}{\mathcal{W}}

\newcommand{\sm}{\ensuremath{\setminus}}
\newcommand{\s}[2]{\sum\limits_{#1}{#2} }

\newcommand{\fu}{\mathfrak{u}}

\renewcommand{\th}{^{th}}

\renewcommand{\to}{\rightarrow}

\newcommand{\dd}[1]{\frac{\partial}{\partial #1}}

\DeclareMathOperator{\newt}{Newt}

\newcommand{\abc}{{(a,b,c)}}

\newenvironment*{prooflemma*}{\paragraph{Proof:}}{}
\newtheorem{Theorem}{Theorem}[section]
\newtheorem{thm*}{Main Theorem}
\newtheorem*{prop*}{Proposition}

\newtheorem{lemma}[Theorem]{Lemma}
\newtheorem{prop}[Theorem]{Proposition}

\theoremstyle{definition}
\newtheorem{assum}[Theorem]{Assumption}
\newtheorem{slogan}[Theorem]{Slogan}
\newenvironment{fslogan}
  {\begin{mdframed}\begin{slogan}}
  {\end{slogan}\end{mdframed}}
\newenvironment{fthm}
  {\begin{mdframed}\begin{thm*}}
  {\end{thm*}\end{mdframed}}
\newtheorem{conv}[Theorem]{Convention}

\newcommand{\Addresses}{{
  \bigskip
  \footnotesize

  \textsc{Soham Chanda, University of Southern California}\\
 \indent\textit{E-mail address}: \texttt{sohamcha@usc.edu}\par
}}

\theoremstyle{definition}
\newtheorem{remark}[Theorem]{Remark}

\newtheorem{defn}[Theorem]{Definition}
\newtheorem{example}[Theorem]{Example}

\numberwithin{equation}{subsection}

\title {Bohr-Sommerfeld type surgeries and Disk Potential}

\author{Soham Chanda}

\begin{document}

\title[Bohr-Sommerfeld profile surgeries and Disk Potentials] { Bohr-Sommerfeld profile surgeries and Disk Potentials \\
 \vspace{0.5cm}  \bf{\footnotesize{--A recipe to create new Lagrangians--}}}

\maketitle
\begin{abstract}
We construct  a new surgery type operation for Lagrangian submanifolds by switching between two exact fillings of Legendrians which we call a BSP surgery. In certain cases, this surgery can preserve monotonicity of Lagrangians. We prove a wall-crossing type formula for the change of the disk-potential under surgery with Bohr-Sommerfeld profiles.  As an application, we show that Biran's circle-bundle lifts admit a Bohr-Sommerfeld type surgery.  We use the wall-crossing theorem about disk-potentials to construct exotic monotone Lagrangian tori in $\bP^n$. Moreover, we use Bohr-Sommerfeld surgeries to generalize the notion of Lagrangian disk surgeries.
\end{abstract}
\setcounter{tocdepth}{2}
\tableofcontents
\setlength{\parskip}{0.5em}

\pagebreak
\section{Introduction}

Surgeries provide a natural way to construct new manifolds from old ones. In the symplectic world, this idea was introduced in the pioneering paper of Polterovich \cite{pol:surg}. Subsequently, numerous Lagrangian surgery operations were developed using Polterovich surgery as the main framework, see e.g. \cite{yau17,haug,makwu}. Recently, a new surgery operation, local higher mutation, was introduced in \cite{Cha23} that resolved certain conical singularities modeled on the Clifford torus. The construction of local higher mutation was inspired by higher mutations of Lagrangian tori introduced in \cite{PT20}. 

A connecting thread among these surgery operations is that all of them switch between two  fillings of certain Legendrians. Given a symplectic manifold $(M,\omega)$  and a separating contact-type hypersurface $Z$, we can split   $M$  along $Z$  by viewing $M\sm Z$  as $ M\sm Z = M^+ \sqcup M^- ,$ where $M^+$ is a capping  and $M^-$ is a filling of the contact manifold $Z$.  If $L$ is a Lagrangian in $M$ such that it intersects the contact-type hypersurface $Z$ at a Legendrian $\Lambda$, then $Z$ splits $L$ similarly,
$ L \sm Z = L^+ \sqcup L^-$
where $L^+$ is a Lagrangian capping and $L^- $ is a Lagrangian filling of $\Lambda$. We can construct a new Lagrangian $\wt L$ by replacing the filling $L^-$ with a different filling $\wt L^-$ and gluing along the Legendrian $\Lambda,$
$$\wt L = (L^+ \cup \wt L^- )/ \sim_{glue} .$$

Although this idea is simple to state, it is quite challenging to understand how such surgeries affect Floer theoretic invariants. The Floer theoretic aspects of such surgeries have been studied  in certain special cases, such as for  Polterovich surgery  \cite{ch10,PW19,seidelFuk}, higher mutation  \cite{PT20,Cha23}.

In this article, we construct a large class of surgery type operations which generalize  Polterovich-surgery and local higher mutations and study the effect of such surgeries on disk-potentials. More precisely, for every Bohr-Sommerfeld Legendrian $\Lambda,$ as introduced in \cite{rg19}, we construct a surgery operation $BSP^\Lambda$.
The new surgery operation switches between exact fillings of a Legendrian $\Lambda_2$ \footnote{The subscript $2$ denotes that there are two connected components in this Legendrian as we will see in \S \ref{bohrsomhandle} } that is obtained from `doubling' a Bohr-Sommerfeld Legendrian $\Lambda$. We prove that such surgery operation preserves monotonicity of Lagrangians under certain assumptions. We call such a surgery operation \textit{ Bohr-Sommerfeld surgery with a $\Lambda$-profile}  and denote the surgered Lagrangian as $BSP^\Lambda(L)$. In the article, we will drop ${\Lambda}$ from the notation when the Legendrian is clear from context.

The main Floer-theoretic result in this paper is a wall-crossing style formula. This can be viewed as a generalization of the mutation formula for Lagrangian disk-surgery and higher mutations,  \cite{Aur07,PT20, PW19}. We will drop some hypotheses for the sake of clarity of exposition; the curious reader should refer to Theorem \ref{mainthm} for  details of the theorem.
 
\vspace{5pt}\vspace{5pt}
\begin{fthm}[Theorem \ref{mainthm}]\label{mainthmintro}
Under nice conditions, a Bohr-Sommerfeld profile surgery of $L$ with a $\Lambda$-profile results in the following change of disk potential
    \begin{equation*}\label{mainform}
        W_{BSP^\Lambda(L)}(x_1,\dots,x_k,z,w_1,\dots,w_l)
 =W_L(x_1,\dots,x_k,zW_{\Lambda}(x_1,\dots,x_k),w_1,\dots,w_l).     \end{equation*}
where $W_\Lambda$ is the augmentation polynomial of $\Lambda$.

\end{fthm}

\begin{remark}{(Lagrangian Mutations as a special case of BSP surgery)} By specializing to BSP surgery with a profile modeled on the real Legendrian circle $S^1_{\bR} = S^3 \cap \bR^2$ in the standard contact sphere $S^3 \hookrightarrow \bC^2$, we can recover the classical wall-crossing formula \cite{Aur07,PW19,ch10} for Lagrangian disk surgeries (Example \ref{ex:disksurg}). By specializing to the BSP surgery, corresponding to the Legendrian Clifford torus, we obtain the change of disk potential under a local higher mutation (Example \ref{ex:highermut}).
\end{remark}

\begin{remark}
    We expect that the change of potential formula in Theorem \ref{mainthm} can be upgraded to an invariance result in the Donaldson-Fukaya category. In particular, under a birational change of local systems, $\rho \mapsto \rho_\Lambda,$ we expect $(L,\rho)$ and $(BSP^\Lambda(L),\rho)$ to be quasi-isomorphic to each other. The proof strategy presented here using SFT breaking should work, but we do not pursue this direction in this article.
\end{remark}

\subsection{Applications}

We explore a few applications of BSP surgery before proceeding to the construction. Building on Biran's circle-bundle construction \cite{Biran_lagrangina_barrier,b06,bc09}, we show how Biran lifts can be Hamiltonian isotoped to include BSP handles, allowing for monotone BSP surgeries and thus producing exotic circle bundle lifts. We also show that BSP surgery provides a geometric realization of mutation of potentials as studied in \cite{21coates,cruzgalkin,minksum}. Furthermore, we apply these ideas to construct exotic monotone Lagrangian tori in $\bC P^n$.

\subsubsection{Lagrangian Cone Surgeries}

For Lagrangian surfaces that have a Lagrangian disk cleanly intersecting in a circle, one can perform a Lagrangian disk surgery as introduced by \cite{yau17} to get a new Lagrangians. This procedure can be thought of as performing an anti-surgery using the Lagrangian disk to create a doubly-immersed-singular point and then resolving it by doing a Polterovich surgery in the `opposite' direction. Such surgeries have been widely used, especially in studying exact Lagrangian fillings \cite{stw1,stw2,cginfty,cwmicro}.

We define the notion of \textit{ Lagrangian cone surgery}, a generalization of disk-surgery : Given a Bohr-Sommerfeld lift $\Lambda$ in the standard contact sphere $S^{2n+1}$, there is a natural Lagrangian cone $C(\Lambda)$ supported on $\Lambda$. We define a Lagrangian cone modeled on a Bohr-Sommerfeld Legendrian $\Lambda$ to be a topological embedding of 
 the cone $C(\Lambda) \cong \Lambda \times [0,1] / ( \Lambda \times \{ 0\})$ such that :
 \begin{enumerate}[({C}.1)]
     \item  away from the conical-singularity point, the embedding is a smooth Lagrangian embedding,
     \item in a Darboux chart centered at the conical point, the embedding looks like $C(\Lambda)$.
\end{enumerate}
In \S \ref{sub:conicalsurg} we generalize Lagrangian disk surgeries by using BSP surgeries with profiles coming from Legendrians in the standard contact sphere. In particular, we show the following result that proves the existence of Lagrangian cones satisfying properties C.1 and C.2, ensures that a Bohr-Sommerfeld surgery can be performed,

\begin{prop*}
    Let $L$ be a monotone Lagrangian and $C$ be a Lagrangian cone modeled on a Bohr-Sommerfeld Legendrian $\Lambda$ in the standard contact sphere $S^{2n+1}$. If $C$ intersects $L$ only at its boundary $\partial C$, and the intersection is clean, then $L$ admits a Bohr-Sommerfeld-Profile surgery modeled on $\Lambda$ such that the surgered manifold $BSP^{\Lambda}(L)$ is also monotone. 
\end{prop*}

\subsubsection{Exotic Lagrangian Circle-Bundles}\label{sub:exoticlift}

Biran's circle-bundle construction (\cite[\S 4]{b06}) provides a way to lift Lagrangians from symplectic divisors in symplectic manifolds. Later, Biran-Cornea \cite{bc09} proved that there are monotone circle-bundle lifts $\wt L \hookrightarrow M$ from monotone Lagrangians $L$ in symplectic divisors $D \hookrightarrow M$. The relation between the disk-potential of $\wt L$ and that of $L$ has been studied in the upcoming work \cite{DTVW} using \cite{bk13}.

We show the following result about constructing \textit{exotic lifts} in Proposition \ref{prop:biranbundlehandle},

\begin{prop*}
    Let $D\hookrightarrow M$ be a symplectic divisor such that the unit circle bundle $Z$ of the normal bundle $N_D$ is a pre-quantum bundle over $D$. Then, Biran's circle-bundle lift of any Lagrangian in $D$ can be Hamiltonian-isotoped to have a BSP handle and thus admits a BSP surgery. Further, we can stay in the monotone setting by starting from a monotone Lagrangian.
\end{prop*}
\noindent Thus, we can apply Theorem \ref{mainthmintro} to observe the following relationship between the disk potentials, 
\begin{equation} \label{eq:lift}
W_{BSP^{\Lambda}(\wt L)} (\cdots) = W_{\wt L} (\cdots, z W_\Lambda, \cdots),
\end{equation}
where $z$ is defined as the monomial corresponding to a cycle that passes through the BSP handle once. 

\begin{remark}
    $W_\Lambda$ is defined as a lift of the disk potential $W_L$ under the natural pullback map induced from $\Lambda \to L$. Just from knowing that disk-potentials are Laurent polynomials,  Equation \ref{eq:lift} indicates that the disk potential of the circle bundle lift ``sees'' that of the base. In particular, the potential of the circle bundle lift admits a change of variables such that swapping $z$ with $zW_\Lambda$ keeps the resulting rational map a Laurent polynomial. This, of course, is evident from the disk potential computation of the Biran circle-bundle Lagrangian in \cite{bk13,DTVW}.
\end{remark}

Recall that Vianna \cite{Via16,Via17} proved that there are infinitely many monotone Lagrangian tori in $\bC P^2$. In fact, Vianna proved that for every Markov triple\footnote{ A Markov triple $(a,b,c)$ is a positive integer solution to the Markov equation $a^2 + b^2 + c^2  = 3abc$.}  $(a,b,c)$, there is a monotone Lagrangian torus $T_\abc$.  Further, Vianna showed there are no symplectomorphisms of $\bC P^2$ which map the torus $T_\abc$ to the torus $T_{(a',b',c')}$ corresponding to different Markov triple $(a',b',c')$ .  Later in \cite{DTVW,chw} it was proved that the Biran lifts of the Vianna tori remain in distinct symplectomorphism classes. We will show that performing BSP surgery on the (iterated) Biran lifts of the Vianna tori will produce exotic monotone tori which have not yet appeared in the literature to the best of our knowledge.

\subsubsection{Geometric realization of mutations of polynomials}\label{subsec:geomrel}

Mutations of Laurent polynomial have turned out to be quite essential in mirror symmetry,  e.g. \cite{cruzgalkin,21coates,minksum}. The effect of BSP surgery on the disk-potential can be phrased in terms of such mutations, in particular BSP surgery on a Lagrangian corresponds  to performing certain combinatorial mutations of a Laurent polynomial.   

We recall the notion of combinatorial mutation of Laurent polynomials as studied in \cite{21coates}.
Given $n \geq 2$ let $\cL:= \bC[x_1^\pm,\dots,x_n^\pm]$ be the ring of Laurent polynomials in $n$ variables. The set of monomials in $\cL$ has an identification with $\bZ^n$ in an obvious way. The \emph{Newton polytope} of a Laurent polynomial $$f = \s{k\in \bZ^n}{a_kx_1^{k_1}\cdots x_n^{k_n}}\in \cL$$  is the closed convex hull
$$\newt(f) := \normalfont\text{conv}(\{k \in \bZ^n : a_k \neq 0\}).$$
This association is equivariant with respect to the $\normalfont\text{GL}(n,\bZ)$-action on $\cL$, defined in \cite[Remark 4.2]{PT20}, and the standard action on $\bR^n$. 

\noindent A Newton polytope  is defined to be a \textit{Fano polytope} ( \cite{minksum} ), if
\begin{enumerate}
    \item the polytope is convex;
    \item it contains $0$ in its interior;
    \item its vertices are primitive in $\bZ^n$.
\end{enumerate}

The change of disk-potential in Theorem   amounts to a combinatorial mutation with a factor corresponding to $W_\Lambda$ in the language of \cite{minksum}. Given a primitive integral vector $w$, and a polytope $P$, we define the max and min height with respect to $ w$, $$h_{max} = max\{ w(v) \mid v\in P  \},\; \; h_{min} = min\{ w(v) \mid v\in P  \} .$$ 
In the equations above, $w(v)$ should be interpreted as $\langle w,v \rangle$ where  $\langle \;,\; \rangle$ is the usual inner-product on the Euclidean space.  The $w$-\textit{width} of the polytope $P$ is defined to  $$width_w(P ) = h_{max}-h_{min}.$$

\noindent We define slices $P_i$ for integers $i$ as $$P_i = \text{conv} \{ a\in P | w(a) = i \}.$$ Note that $P_i = \emptyset$ for $i< h_{min}$ and $i>h_{max}$. Assume that there is a  convex lattice polytope $F \sub w^\perp$ \footnote{ $w^\perp$ is the subspace $\{ p | w(p) = 0 \}$.}, such that for every height $h_{min} \leq h < 0$ there exists a (possibly empty) lattice polytope $G_h$ satisfying  $$\{ v \in \text{vert}(P) | w(v) = i \} \subset G_h + |h| F \subset P_h.$$
where  $+$ denotes Minkowski sum. We call such an $F$, a factor for $P$ with respect to $w$. We define the combinatorial mutation given by width vector $w$, factor $F$, and polytopes {$G_h$} to be the convex lattice polytope

$$\text{mut}_{w}(P,F) := \text{conv}\left(  \bigcup_{h=h_{min}}^{-1}G_h \cup \bigcup_{h=0}^{h_{max}} \left(P_h + hF\right)   \right).$$
\noindent Although mutation uses the existence of polytopes $G_h$,  Proposition 1 \cite{minksum} proves that the mutation is independent of the choice of $G_h$.  
The effect of the mutation formula in Theorem \ref{mainthm} can be studied clearly when we write the disk potential $W_L$ of a monotone Lagrangian   as a polynomial $$W_L(x_1,\dots, z , y_1 \dots) = \sum_{i=k}^l z^iC_i,$$
where $k<0,$ and  $l>0$, $C_i$ are Laurent polynomials with variable $x_i,y_j$.
Particularly in the case of Biran circle bundle lifts, the Newton polytope changes by a combinatorial mutation with a factor corresponding to a codimension one face. From Theorem \ref{mainthmintro} we can see that $W_\Lambda^{-h}$ divides $C_h$ for all $k \leq h < 0$.

We have the following relationship between Newton polytopes under a Bohr-Sommerfeld-Profile surgery, $$\newt(W_{BSP^{\Lambda}(L)}) = \text{mut}_{\hat z}(\newt(W_L),\newt(W_\Lambda)),$$
 where $\hat z$ is the dual of the $z$ variable. Thus,  we see that the change of disk-potential in Theorem \ref{mainthmintro}  amounts to a combinatorial mutation with a factor corresponding to $W_\Lambda$.

\begin{figure}[ht]
    \makebox[\textwidth][c]{
   
        \def\svgscale{1}
\begingroup%
  \makeatletter%
  \providecommand\color[2][]{%
    \errmessage{(Inkscape) Color is used for the text in Inkscape, but the package 'color.sty' is not loaded}%
    \renewcommand\color[2][]{}%
  }%
  \providecommand\transparent[1]{%
    \errmessage{(Inkscape) Transparency is used (non-zero) for the text in Inkscape, but the package 'transparent.sty' is not loaded}%
    \renewcommand\transparent[1]{}%
  }%
  \providecommand\rotatebox[2]{#2}%
  \newcommand*\fsize{\dimexpr\f@size pt\relax}%
  \newcommand*\lineheight[1]{\fontsize{\fsize}{#1\fsize}\selectfont}%
  \ifx\svgwidth\undefined%
    \setlength{\unitlength}{450bp}%
    \ifx\svgscale\undefined%
      \relax%
    \else%
      \setlength{\unitlength}{\unitlength * \real{\svgscale}}%
    \fi%
  \else%
    \setlength{\unitlength}{\svgwidth}%
  \fi%
  \global\let\svgwidth\undefined%
  \global\let\svgscale\undefined%
  \makeatother%
  \begin{picture}(1,0.5)%
    \lineheight{1}%
    \setlength\tabcolsep{0pt}%
    \put(0,0){\includegraphics[width=\unitlength,page=1]{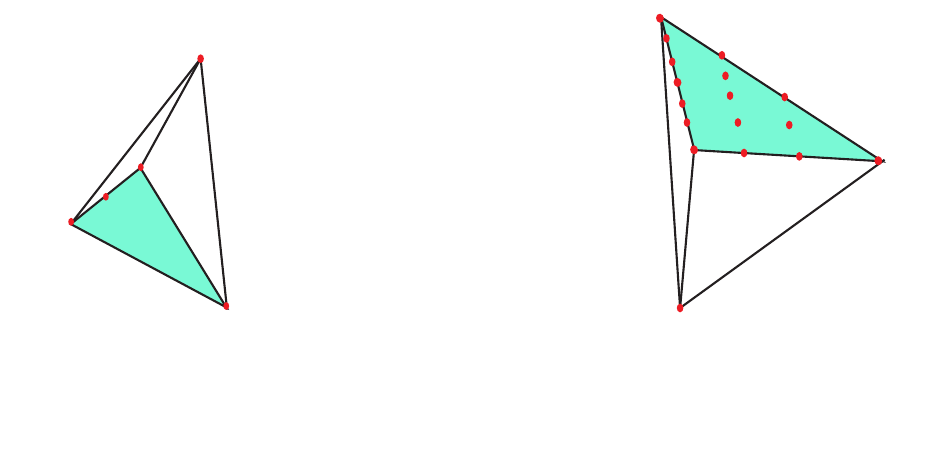}}%
    \put(0.08375209,0.09073143){\makebox(0,0)[lt]{\lineheight{1.25}\smash{\begin{tabular}[t]{l}$\newt(W_{\overline{T}_{(1,1,2)}})$\end{tabular}}}}%
    \put(0.60580683,0.09575656){\makebox(0,0)[lt]{\lineheight{1.25}\smash{\begin{tabular}[t]{l}$\newt(W_{\wt{T}_{(1,1,2)}})$\end{tabular}}}}%
  \end{picture}%
\endgroup%

 }
    \caption{Newton polytope of the disk potential of the lifted Vianna tori $\overline{T}_{(1,1,2)}$, and its BSP surgery $\wt {T}_{(1,1,2)} $. The red points correspond to the exponents of monomials.}
    \label{fig:newts}
\end{figure}

\subsubsection{Exotic tori in $\bP^n$}
We can use the  discussion in \S \ref{sub:exoticlift} to see that the lifted Vianna tori $\overline{T}_\abc$ in $\bP^3$, whose Newton polytopes are computed in \cite{chw}, have BSP handles and admit monotone BSP surgery. The Newton polytope of $BSP^{\Lambda_\abc}(\overline{T}_\abc)$ corresponds to performing a combinatorial mutation of the Newton polytope $\newt(W_{\overline{T}_\abc})$ along the two-dimensional facet corresponding to $\newt(W_{T_\abc})$. Thus, each Vianna torus $T_\abc$ admits potentially two distinct lifts in $\bP^3$. 

To see that the exotic lifts of the Vianna tori using BSP surgery are, in fact, different from the Biran lifts of Vianna tori, we study their Newton polytopes. More precisely, we study the effect of BSP surgery on the Newton polytope of the disk potential using combinatorial mutations. 

From \cite{chw} we know that $\newt(\overline{T_\abc})$ can be viewed as a cone over $\newt(T_\abc)$. Here $T_\abc$ is the two dimensional Vianna torus corresponding to the Markov triple $\abc$. Let $w$ be the primitive integral vector perpendicular to the triangular face $\newt(T_\abc)$ contained in $\newt(\overline{T_\abc})$. Note that the $w$-width of $\newt (W_{\overline{T_\abc}})$ is 4 and $h_{max} = 3$. Thus, from the combinatorial mutation formulation, we see that the Newton polytope of $BSP^{\Lambda_\abc}(\overline{T}_\abc)$ has a face isomorphic to $3\newt(T_\abc)$. Here $3\newt(T_\abc)$ refers to the triple Minkowski sum, $\newt(T_\abc) + \newt(T_\abc) + \newt(T_\abc)$.

More concretely, later in Section \ref{app} we will show that we can perform a zero-area BSP surgery to $T^n_\abc$ with profile $\Lambda^{n-1}_{\abc}$ which produces new monotone Lagrangian tori in $\bC P^n$ which are different from the ones already found in the literature such as \cite{PT20,chw,Y22}. For example, we can construct $ \wt T^3_{1,1,2} =BSP^{\Lambda_{(1,1,2)}}_0(\overline{T}^3_{1,1,2}) \hookrightarrow \bC P^3$ whose potential is given by $$W_{\wt T^3_{1,1,2}}(x,y,z)=W_{\overline{T}^3_{(1,1,2)}} (x,y,z(y+(1+x)^2)).$$ Under a judicious choice of basis of $H_1(\overline{T}^3_{(1,1,2)})$ we have 
$$W_{\overline{T}^3_{(1,1,2)}}(x,y,z) = \frac{y+(1+x)^2}{z} + \frac{z^3}{xy^2},$$
and thus, from our wall-crossing formula, we have 
$$W_{\wt T^3_{(1,1,2)}}(x,y,z) = \frac{1}{z} + \frac{z^3(y+(1+x)^2)^3}{xy^2}.$$
See Figure \ref{fig:newts} for a comparison of the Newton polytope of the above two disk potentials.

\begin{Theorem}\label{morevianna}[See Proposition \ref{newpolytopes} for description of the relevant Newton polytopes]
   The BSP surgered monotone Lagrangian torus $BSP(\overline{T}_{(a,b,c)})$ is distinct (up to symplectomorphism) from any of the lifted  Vianna tori $\overline{T}_{(a',b',c')}$ for any  pair of Markov triples $(a,b,c)$ and $(a',b',c')$. Moreover, the BSP surgered tori are pairwise distinct, i.e. if there is a symplectomorphism $\phi: \bP^n \to \bP^n$ such that $$\phi(BSP(\overline{T}_{(a,b,c)})) = BSP(\overline{T}_{(a',b',c')}), $$  then the Markov triples agree, $(a,b,c) = (a',b',c')$.
\end{Theorem}

\begin{figure}[ht]
    \makebox[\textwidth][c]{
   
        \def\svgscale{0.8}
    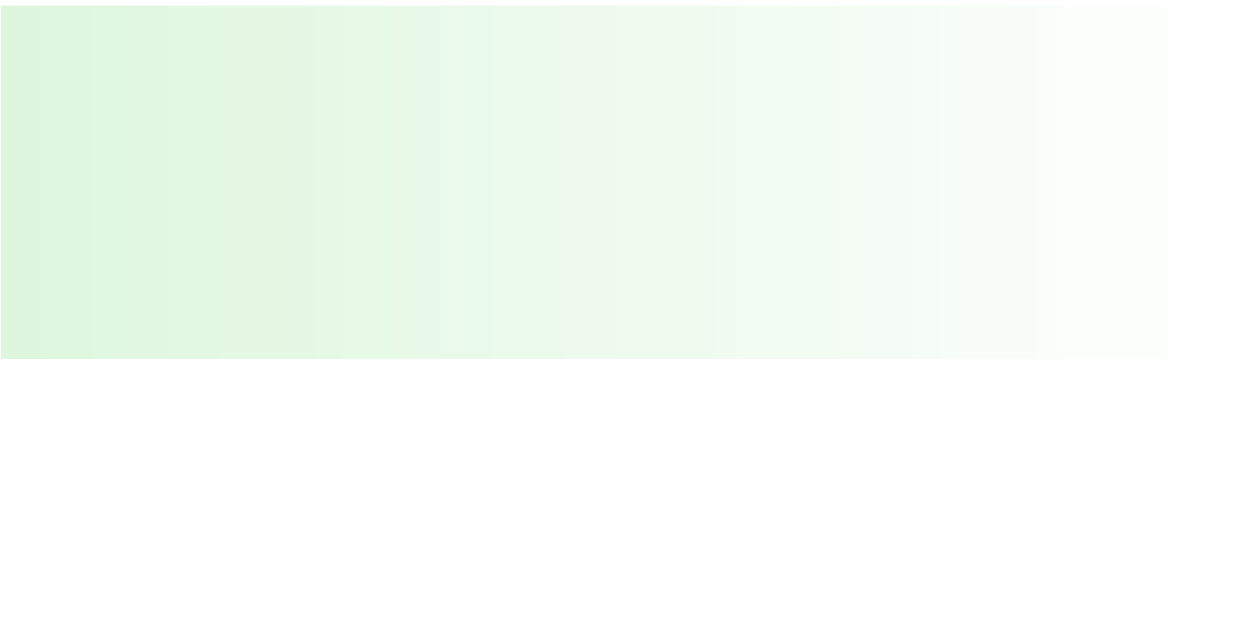

 }
    \caption{Mutation graph of $\newt(x+y+z+\frac{1}{xyz})$ labeled with corresponding monotone tori. The integral points on edges are labeled in pink. The blue arrow corresponds to an edge mutation, and the purple arrow is a face mutation.}
    \label{fig:newts}
\end{figure}

\begin{remark}

By proceeding inductively, we expect to create $2^n$ lifts of the Vianna torus $T_\abc$ in $\bP^{2+n}$. We provide a proof sketch. One can use \cite{DTVW,bk13} to show that performing a Biran-lift results in cone $C(N)$ of the Newton polytope. From \S \ref{sub:exoticlift} we see that an exotic BSP lift results in a combinatorial mutation $\text{mut}(S(N))$ with $N$ as a factor.   
\end{remark}

\subsection{Beginnings of a new surgery}
Let $(Z,\lambda) \xrightarrow{\pi_Y} (Y,\omega_Y)$ be a prequantum bundle, i.e. a circle bundle with a connection form $\lambda$ such that the curvature is given by $\omega_Y$, i.e. $$d\lambda = \pi_Y^* \omega_Y.$$ Note that such a one-form $\lambda$ is a contact form on $Z$. Assume $\omega_Y$ is a monotone symplectic form and $L_Y$ is a monotone Lagrangian in $Y$, i.e. the Maslov index functional and the symplectic area functional are positively proportional, $$\mu(A) = c\int _A\omega \text{ for all } A\in \pi_2(Y,L_Y) \text{ for some } c>0.$$
The canonical Bohr-Sommerfeld cover of $L_Y$ as defined in \cite[\S 4]{rg19} is a Legendrian $\Lambda$ in $Z$ that is a cover of the monotone Lagrangian $L_Y$.  In essence, the Bohr-Sommerfeld cover can be thought of as a horizontal lift (with respect to the connection $\lambda$) of the Lagrangian $L_Y$ to $Z$.

A surgery with Bohr-Sommerfeld profile $\Lambda$ switches between two different exact fillings of the disjoint Legendrian $\Lambda_2 := \Lambda \sqcup e^{\frac{i\pi}{k}}\Lambda$ where $k$ is the covering number of $\Lambda \to L_Y$.(  Here we are implicitly using the $S^1$ bundle structure on $Z$.)
Given an embedded path $c:[-1,1] \to \bC^*$ in the punctured complex plane such that $c(-1) = -1, c(1) = 1$, we can construct a filling of $\Lambda_2$, which we call $\cL_c$, see Definition \ref{def:Bsphand}.  We call such Lagrangian fillings $\cL_c$  Bohr-Sommerfeld Profile Lagrangian handles. There are two primitive homotopy classes of such paths. A Bohr-Sommerfeld profile surgery switches between these two homotopy classes. 

\begin{figure}[ht]

        \def\svgscale{0.6}
\begingroup%
  \makeatletter%
  \providecommand\color[2][]{%
    \errmessage{(Inkscape) Color is used for the text in Inkscape, but the package 'color.sty' is not loaded}%
    \renewcommand\color[2][]{}%
  }%
  \providecommand\transparent[1]{%
    \errmessage{(Inkscape) Transparency is used (non-zero) for the text in Inkscape, but the package 'transparent.sty' is not loaded}%
    \renewcommand\transparent[1]{}%
  }%
  \providecommand\rotatebox[2]{#2}%
  \newcommand*\fsize{\dimexpr\f@size pt\relax}%
  \newcommand*\lineheight[1]{\fontsize{\fsize}{#1\fsize}\selectfont}%
  \ifx\svgwidth\undefined%
    \setlength{\unitlength}{450bp}%
    \ifx\svgscale\undefined%
      \relax%
    \else%
      \setlength{\unitlength}{\unitlength * \real{\svgscale}}%
    \fi%
  \else%
    \setlength{\unitlength}{\svgwidth}%
  \fi%
  \global\let\svgwidth\undefined%
  \global\let\svgscale\undefined%
  \makeatother%
  \begin{picture}(1,0.50086665)%
    \lineheight{1}%
    \setlength\tabcolsep{0pt}%
    \put(0,0){\includegraphics[width=\unitlength,page=1]{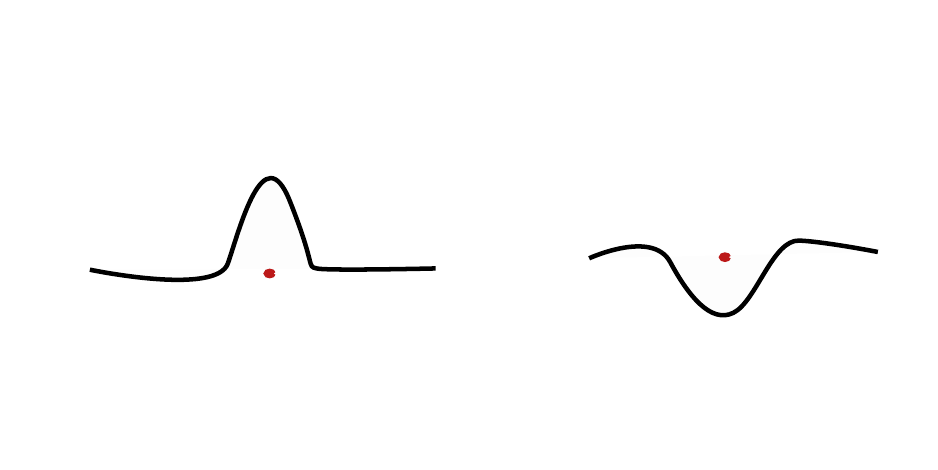}}%
    \put(0.74686736,0.09434391){\color[rgb]{1,0,0}\transparent{0.00784314}\makebox(0,0)[lt]{\lineheight{1.25}\smash{\begin{tabular}[t]{l}$\wt \gamma$\end{tabular}}}}%
    \put(0.20933896,0.10920023){\color[rgb]{0,0,0}\transparent{0}\makebox(0,0)[lt]{\lineheight{1.25}\smash{\begin{tabular}[t]{l}$\gamma$\end{tabular}}}}%
  \end{picture}%
\endgroup%

    \caption{Two primitive homotopy classes of embedded paths in $\bC^*$ with fixed end-points.  }
    \label{fig:gammapair}
\end{figure}

Intuitively, we call a Lagrangian $L$ in a symplectic manifold $M$ to \textit{have a Bohr-Sommerfeld-Profile handle} if under a symplectomorphism from $[a,b]\times Z$ to a codimension 0 submanifold with boundary in  $M$,    $L$ looks like $\cL_\gamma$ for some path $\gamma,$ see Definition \ref{bsphandle}. A Bohr-Sommerfeld-profile surgery of such a Lagrangian $L$ is defined as     \begin{equation*}
        BSP^{\Lambda}(L) = (L \sm \phi\inv ( L) )\cup _{  \Lambda_2 } \cL_{\wt \gamma}, 
    \end{equation*}  where $\wt\gamma$ denotes a path in $\bC^*$ with same end points as that of $\gamma$ but representing a different primitive homotopy class than $\gamma$.

In Lemma \ref{ref:mono} we show that if $L$ is a monotone Lagrangian and $\wt\gamma$ is chosen judiciously (see Equation \ref{eq:zeroar}), then a Bohr-Sommerfeld-profile surgery using $\wt \gamma$ preserves monotonicity, that is $BSP(L)$ is also monotone. We call such surgeries zero-area BSP surgeries and denote them  as $BSP_0(L)$.

\subsection{Acknowledgements}
We thank Chris Woodward, Denis Auroux, Mohammed Abouzaid, Georgios Dimitroglou-Rizell, Felix Schlenk and Dylan Cant for valuable discussions. We thank the anonymous referee whose suggestion helped remove a technical assumption (enough regular disks), making our results significantly stronger.  We also thank Amanda Hirschi, Julian Chaidez and Mohan Swaminathan for feedback on an earlier draft.  This work was partially supported by the grants NSF DMS 2105417 and NSF DMS 2345030.

\section{Construction of the Bohr-Sommerfeld Profile Surgery}\label{sec:Constr}

As we have already discussed in the introduction, Bohr-Sommerfeld profile surgeries are obtained from switching between two natural exact fillings of Bohr-Sommerfeld lifts. We start off with reviewing some preliminaries about constructing Bohr-Sommerfeld lifts of monotone Lagrangians. See \cite[\S~4]{rg19} for a more in-depth discussion of Bohr-Sommerfeld lifts.

\subsection{Preliminaries}\label{sub:prelim}

Let $(Y,\omega_Y)$ be a simply connected, compact, monotone, symplectic manifold with an integral symplectic form. Let $(Z,\lambda) \to Y$ be a prequantization bundle. Recall that a prequantization over $Y$ is a circle-bundle with a connection one-form $\lambda$ whose curvature is equal to the pullback of the symplectic form $\omega_Y$.

\begin{conv}
    In this section, whenever we take a parallel transport, it is with respect to the connection induced from $\lambda$.
\end{conv}

Let $L_Y$ be a monotone Lagrangian submanifold whose minimal Maslov number (which we will denote by $N_L$) is two.  
\begin{assum}\label{as1} We assume that the symplectic areas of cycles with boundary on the Lagrangian $L_Y$ satisfy the following,     
    $$[\omega_Y] \in \frac{\pi}{k} \im((H_2(Y,L_Y,\bZ)) \hookrightarrow H_2(Y,L_Y,\bR)),$$ for some $k\in \bN$. 
\end{assum}

\begin{conv}
   From now on, we will choose the minimal positive number $k$ which satisfies Assumption \ref{as1}.
\end{conv}

\begin{defn}[Parallel transport along  base]\label{def:pll}

Let  $E \to B$ be a circle-bundle with a connection $\lambda$ such that $E$ is flat. Let $P\sub E_p$ be a subset of the fiber $E_p$ such that $P$ is fixed by the holonomy group with base point at $p$, i.e. for any $\gamma:S^1 \to B$ with $\gamma(1)=p$,  the image of $P$ under the parallel transport map along $\gamma$ is $P$ itself. We define the set $\mathfrak{P}$, the parallel transport of $P$ along $B$ to be the set of points $x\in E$ such that $x\in PT_\gamma(P)$ for some curve $\gamma:[0,1] \to B$ with $\gamma(0)=p$.
\end{defn}

\begin{lemma}\label{lemm:pllmanifold}
    The set $\mathfrak{P}$ is a submanifold of $E$ such that the projection map restricts to a covering map. 
\end{lemma}
\begin{proof}
    Since the set $P$ is fixed under holonomy, we have that the set $\mathfrak{P}$ restricted to each fiber is equal to $P$ up to the natural $S^1$ action. Now the result directly follows from the fact that $E$ has no curvature, so locally $\mathfrak P$ is a horizontal section over an open neighborhood in $B$. 
\end{proof}

\begin{defn}[Canonical Bohr-Sommerfeld lift]
Let $P_k = \left\{ e^{i\frac {2\pi} {k} l} \; | \; \; l\in \{ 1, 2,\dots k\}\right\}$ be the set of $k-$th roots of unity. We define the canonical Bohr-Sommerfeld cover $\Lambda \sub Z$ as the manifold obtained by parallel transporting (see Definition \ref{def:pll}) $P_k$ along the Lagrangian $L_Y$. 
\end{defn}
Assumption \ref{as1} and Lemma \ref{lemm:pllmanifold}  imply that $\Lambda$ indeed is a submanifold of $Z$. Moreover,  $\Lambda$ is a Legendrian submanifold of $(Z,\lambda)$.  See \cite[\S~4.1]{rg19} for an algebraic-topology flavoured construction of $\Lambda$. Note that $\Lambda$ depends on an identification of $S^1$ with a fiber over $L_Y$, so there is actually a $ S^1/ \bZ_k \cong S^1 $ family of $\Lambda$.

The symplectization of $Z$, $\bR \times Z$ carries a natural $\bR \times S^1 \cong \bC^*$ action induced from the $S^1$ action on $Z$ and translation action on the $\R$-component. Thus, we can view $\bR \times Z \to Y$ as a $\bC^*$-bundle. We define a symplectic form on $\bR \times Z$ as $$\omega= d(e^s \lambda) = e^s\pi^*\omega_Y  + e^s ds\wedge \lambda,$$
and any almost complex structure $J_Y$ on $Y$ lifts uniquely to a cylindrical almost complex structure $J_c$ on $\bR \times Z$ such that 
\begin{itemize}
    \item $J_c (\partial s) = R$,
    \item $J_c|_{\ker \lambda} = \pi^* J_Y$.
    \item for any $a\in \bR$, $J_c = \tau_a^*J_c$, where $\tau_a$ is the translation function.
\end{itemize}
The almost complex structure $J_c$ on $\bR \times Z$ makes the projection map, $\pi: \bR \times Z \to Y$, a pseudoholomorphic map.

\begin{defn}[$k$-th roots of curve]
    Given an embedded curve $\gamma:[0,1] \to \bC^*$, we define $$\gamma^{1/k} := \left \{ z\in \bC^* \mid z^k \in \gamma([0,1]) \right\}. $$ Clearly $\gamma^{1/k}$ is a collection of $k$ embedded curves which are related by multiplication by $k$-th roots of unity. 
\end{defn}

\begin{figure}[ht]
    \centering

        \def\svgscale{1}
\begingroup%
  \makeatletter%
  \providecommand\color[2][]{%
    \errmessage{(Inkscape) Color is used for the text in Inkscape, but the package 'color.sty' is not loaded}%
    \renewcommand\color[2][]{}%
  }%
  \providecommand\transparent[1]{%
    \errmessage{(Inkscape) Transparency is used (non-zero) for the text in Inkscape, but the package 'transparent.sty' is not loaded}%
    \renewcommand\transparent[1]{}%
  }%
  \providecommand\rotatebox[2]{#2}%
  \newcommand*\fsize{\dimexpr\f@size pt\relax}%
  \newcommand*\lineheight[1]{\fontsize{\fsize}{#1\fsize}\selectfont}%
  \ifx\svgwidth\undefined%
    \setlength{\unitlength}{450bp}%
    \ifx\svgscale\undefined%
      \relax%
    \else%
      \setlength{\unitlength}{\unitlength * \real{\svgscale}}%
    \fi%
  \else%
    \setlength{\unitlength}{\svgwidth}%
  \fi%
  \global\let\svgwidth\undefined%
  \global\let\svgscale\undefined%
  \makeatother%
  \begin{picture}(1,0.5)%
    \lineheight{1}%
    \setlength\tabcolsep{0pt}%
    \put(0,0){\includegraphics[width=\unitlength,page=1]{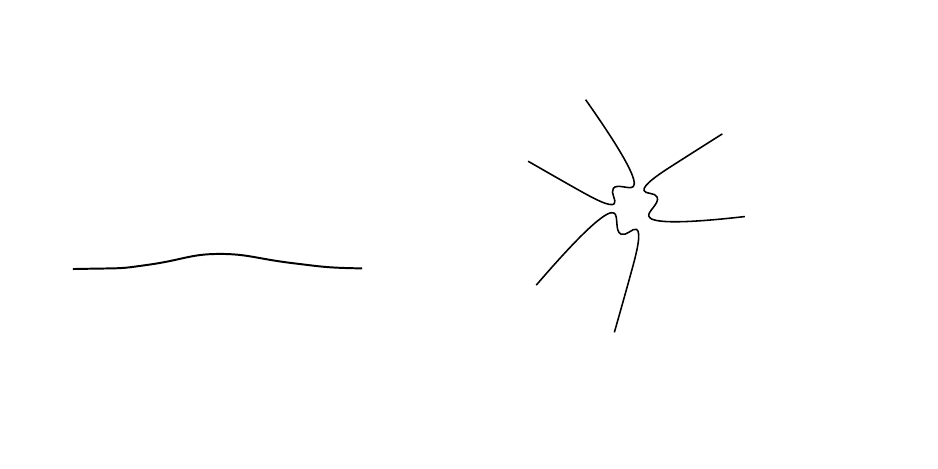}}%
    \put(0.17282051,0.09473684){\makebox(0,0)[lt]{\lineheight{1.25}\smash{\begin{tabular}[t]{l}$\gamma$\end{tabular}}}}%
    \put(0.63002696,0.10040487){\makebox(0,0)[lt]{\lineheight{1.25}\smash{\begin{tabular}[t]{l}$\gamma^{\frac{1}{k}}$\end{tabular}}}}%
    \put(0,0){\includegraphics[width=\unitlength,page=2]{kthgamma.pdf}}%
    \put(0.36518219,0.0506073){\makebox(0,0)[lt]{\lineheight{1.25}\smash{\begin{tabular}[t]{l}$\bC$\end{tabular}}}}%
    \put(0.82429148,0.04736847){\makebox(0,0)[lt]{\lineheight{1.25}\smash{\begin{tabular}[t]{l}$\bC$\end{tabular}}}}%
  \end{picture}%
\endgroup%

    \caption{$k$-th root of curve $\gamma$}
\end{figure}
\begin{defn}[Admissible curve]
    An admissible curve $\gamma$ is an embedded curve in an annulus $A(a,b) = \left \{ z \in \bC | |z|\in [a,b] \right \}$ which satisfies the following,
    \begin{itemize}
        \item $\gamma([0,1]) \sub \bH$
        \item $\gamma(0) = b, \gamma(1) = -b$
        \item there is a  $c > 0 $ such that  $\gamma(t) \in \bR$  for $t\in [0,c) \cup (1-c,1] .$
    \end{itemize}
    
\end{defn}

\begin{figure}[ht]
    \centering

        \def\svgscale{1}
\begingroup%
  \makeatletter%
  \providecommand\color[2][]{%
    \errmessage{(Inkscape) Color is used for the text in Inkscape, but the package 'color.sty' is not loaded}%
    \renewcommand\color[2][]{}%
  }%
  \providecommand\transparent[1]{%
    \errmessage{(Inkscape) Transparency is used (non-zero) for the text in Inkscape, but the package 'transparent.sty' is not loaded}%
    \renewcommand\transparent[1]{}%
  }%
  \providecommand\rotatebox[2]{#2}%
  \newcommand*\fsize{\dimexpr\f@size pt\relax}%
  \newcommand*\lineheight[1]{\fontsize{\fsize}{#1\fsize}\selectfont}%
  \ifx\svgwidth\undefined%
    \setlength{\unitlength}{225bp}%
    \ifx\svgscale\undefined%
      \relax%
    \else%
      \setlength{\unitlength}{\unitlength * \real{\svgscale}}%
    \fi%
  \else%
    \setlength{\unitlength}{\svgwidth}%
  \fi%
  \global\let\svgwidth\undefined%
  \global\let\svgscale\undefined%
  \makeatother%
  \begin{picture}(1,0.16666667)%
    \lineheight{1}%
    \setlength\tabcolsep{0pt}%
    \put(0,0){\includegraphics[width=\unitlength,page=1]{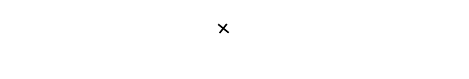}}%
    \put(0.85518565,0.03259234){\makebox(0,0)[lt]{\lineheight{1.25}\smash{\begin{tabular}[t]{l}$\R$\end{tabular}}}}%
    \put(0.16592566,0.02703835){\makebox(0,0)[lt]{\lineheight{1.25}\smash{\begin{tabular}[t]{l}$\gamma$\end{tabular}}}}%
    \put(0.44555634,0.03481366){\makebox(0,0)[lt]{\lineheight{1.25}\smash{\begin{tabular}[t]{l}$0$\end{tabular}}}}%
    \put(0,0){\includegraphics[width=\unitlength,page=2]{gamma.pdf}}%
  \end{picture}%
\endgroup%

    \caption{Admissible curve $\gamma$}
\end{figure}

\begin{defn}[Conjugate of an admissible curve]
For any admissible curve $\gamma$, a conjugate of that curve is any curve $\widetilde {\gamma} $ such that 
\begin{itemize}
    \item the curves match near the ends, i.e.  $$\gamma (t) = \wt \gamma (t) \text{ for  } t\in [0,c) \cup (1-c,1]$$
    \item the concatenated path $\gamma  \# {\wt \gamma}\inv $ has winding number one in $\bC^*$, i.e. $$[\gamma  \# {\wt \gamma}\inv]= 1   \text{ in } H_1(\bC^*) $$

\end{itemize}

\end{defn}

\subsection{Bohr-Sommerfeld Profile Lagrangian Handles} \label{bohrsomhandle}

We will construct Bohr-Sommerfeld profile surgeries by switching between fillings of the `double' Bohr-Sommerfeld lift $\Lambda_2$.  This Legendrian has two connected components, 

$$\Lambda_2 = \Lambda \cup e^{i\frac{\pi }{k}}.\Lambda.$$
In $\Lambda_2$, the connected components are related by a `phase-shift' i.e. action of some element of $S^1$,  of the other. 
\begin{figure}[ht]
    \centering

        \def\svgscale{1}
\begingroup%
  \makeatletter%
  \providecommand\color[2][]{%
    \errmessage{(Inkscape) Color is used for the text in Inkscape, but the package 'color.sty' is not loaded}%
    \renewcommand\color[2][]{}%
  }%
  \providecommand\transparent[1]{%
    \errmessage{(Inkscape) Transparency is used (non-zero) for the text in Inkscape, but the package 'transparent.sty' is not loaded}%
    \renewcommand\transparent[1]{}%
  }%
  \providecommand\rotatebox[2]{#2}%
  \newcommand*\fsize{\dimexpr\f@size pt\relax}%
  \newcommand*\lineheight[1]{\fontsize{\fsize}{#1\fsize}\selectfont}%
  \ifx\svgwidth\undefined%
    \setlength{\unitlength}{225bp}%
    \ifx\svgscale\undefined%
      \relax%
    \else%
      \setlength{\unitlength}{\unitlength * \real{\svgscale}}%
    \fi%
  \else%
    \setlength{\unitlength}{\svgwidth}%
  \fi%
  \global\let\svgwidth\undefined%
  \global\let\svgscale\undefined%
  \makeatother%
  \begin{picture}(1,1)%
    \lineheight{1}%
    \setlength\tabcolsep{0pt}%
    \put(0,0){\includegraphics[width=\unitlength,page=1]{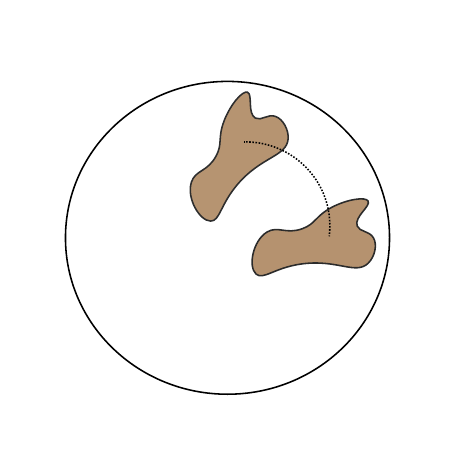}}%
    \put(0.64696355,0.4){\makebox(0,0)[lt]{\lineheight{1.25}\smash{\begin{tabular}[t]{l}$\Lambda$\end{tabular}}}}%
    \put(0.36465317,0.68123751){\makebox(0,0)[lt]{\lineheight{1.25}\smash{\begin{tabular}[t]{l}$e^{\frac{i\pi}{k}}\Lambda$\end{tabular}}}}%
    \put(0.68502024,0.60890686){\makebox(0,0)[lt]{\lineheight{1.25}\smash{\begin{tabular}[t]{l}$e^{\frac{i\pi}{k}}$\end{tabular}}}}%
    \put(0.46720647,0.20971659){\makebox(0,0)[lt]{\lineheight{1.25}\smash{\begin{tabular}[t]{l}$Z$\end{tabular}}}}%
  \end{picture}%
\endgroup%

    \caption{$\Lambda_2 = \Lambda \cup e^{i\frac{\pi }{k}}.\Lambda$}
    \label{fig:lagangle}
\end{figure}

We now proceed towards constructing fillings of the Legendrian $\Lambda_2$. The truncated symplectization $[r_1, r_2] \times Z$ is an annulus bundle over $Y$ which carries a connection induced by the connection form $\lambda$ on $Z$.

Bohr-Sommerfeld Profile handles are Lagrangian fillings of  $\{r_2\} \times \Lambda_2$ in the truncated symplectization $[r_1, r_2] \times Z$ constructed by parallel transporting. Definition \ref{def:pll} of parallel transporting along base can be extended naturally from $S^1$ bundles to their associated $\bC^*$ bundles so we omit the definition. 
In certain computations regarding $\cL_\gamma$, using the identification $A(a,b) := [\log a, \log b] \times S^1$ turns out to provide a clearer picture.

\begin{defn}[Bohr-Sommerfeld Profile handle]\label{def:Bsphand}
    Given an admissible curve $\gamma : [0,1] \to A(a,b)$, we define $$\cL_\gamma$$ to be the parallel transport of $\gamma^{1/k}$ along the Lagrangian $L_Y\hookrightarrow Y$ where we view $[\frac {\log a}{k},\frac {\log b}{k}  ] \times Z$ as an annulus bundle over $Y$.
\end{defn}
Recall that the restriction of the bundle $Z$ to the Lagrangian $L_Y\hookrightarrow Y$ has no curvature, since  $d\lambda$ vanishes on $L_Y$.

\begin{figure}[ht]
    \makebox[\textwidth][c]{
   
        \def\svgscale{1.3}
\begingroup%
  \makeatletter%
  \providecommand\color[2][]{%
    \errmessage{(Inkscape) Color is used for the text in Inkscape, but the package 'color.sty' is not loaded}%
    \renewcommand\color[2][]{}%
  }%
  \providecommand\transparent[1]{%
    \errmessage{(Inkscape) Transparency is used (non-zero) for the text in Inkscape, but the package 'transparent.sty' is not loaded}%
    \renewcommand\transparent[1]{}%
  }%
  \providecommand\rotatebox[2]{#2}%
  \newcommand*\fsize{\dimexpr\f@size pt\relax}%
  \newcommand*\lineheight[1]{\fontsize{\fsize}{#1\fsize}\selectfont}%
  \ifx\svgwidth\undefined%
    \setlength{\unitlength}{450bp}%
    \ifx\svgscale\undefined%
      \relax%
    \else%
      \setlength{\unitlength}{\unitlength * \real{\svgscale}}%
    \fi%
  \else%
    \setlength{\unitlength}{\svgwidth}%
  \fi%
  \global\let\svgwidth\undefined%
  \global\let\svgscale\undefined%
  \makeatother%
  \begin{picture}(1,0.5)%
    \lineheight{1}%
    \setlength\tabcolsep{0pt}%
    \put(0,0){\includegraphics[width=\unitlength,page=1]{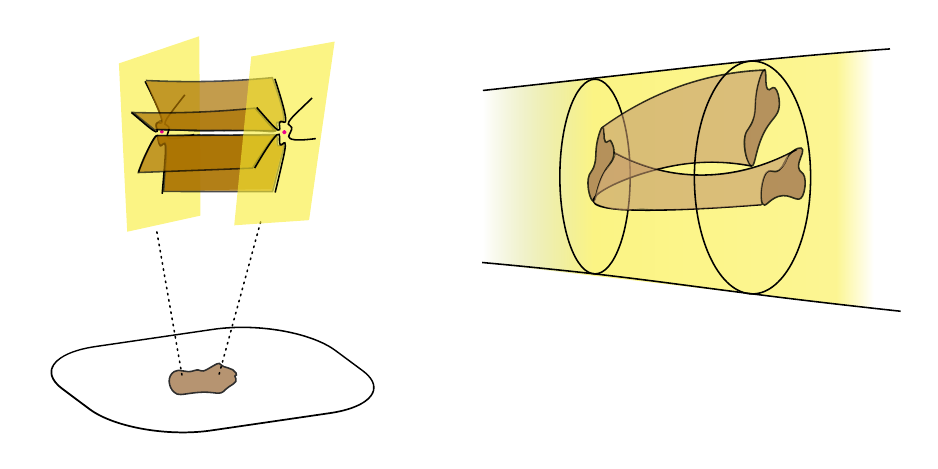}}%
    \put(0.14990508,0.06061417){\makebox(0,0)[lt]{\lineheight{1.25}\smash{\begin{tabular}[t]{l}$L_Y$\end{tabular}}}}%
    \put(0.30512563,0.07509771){\makebox(0,0)[lt]{\lineheight{1.25}\smash{\begin{tabular}[t]{l}$Y$\end{tabular}}}}%
    \put(0.79614736,0.14546621){\makebox(0,0)[lt]{\lineheight{1.25}\smash{\begin{tabular}[t]{l}$\bR\times Z$\end{tabular}}}}%
    \put(0.74251256,0.26997767){\makebox(0,0)[lt]{\lineheight{1.25}\smash{\begin{tabular}[t]{l}$\Lambda_2$\end{tabular}}}}%
    \put(0.60573425,0.36041318){\makebox(0,0)[lt]{\lineheight{1.25}\smash{\begin{tabular}[t]{l}$\cL_\gamma$\end{tabular}}}}%
    \put(0.26157453,0.28001116){\makebox(0,0)[lt]{\lineheight{1.25}\smash{\begin{tabular}[t]{l}$\bC^*$\end{tabular}}}}%
    \put(0.14208822,0.32970408){\makebox(0,0)[lt]{\lineheight{1.25}\smash{\begin{tabular}[t]{l}$\cL_\gamma$\end{tabular}}}}%
    \put(0.26771637,0.39838079){\makebox(0,0)[lt]{\lineheight{1.25}\smash{\begin{tabular}[t]{l}$\gamma^{\frac{1}{k}}$\end{tabular}}}}%
  \end{picture}%
\endgroup%

 }
    \caption{$\cL_\gamma$ from two perspectives}
\end{figure}
 We can write an admissible curve using this identification as $$\gamma(t) = (\gamma_l (t),\gamma_\theta (t)).$$ Then, one connected component of $\gamma^{1/k}$ is given by $\gamma_1 (t)  = (\frac{1}{k}\gamma_l (t),\frac{1}{k}\gamma_\theta (t))$. The other components are given by similar formulae where the $S^1$ co-ordinate is shifted by a constant angle.

\begin{remark}\label{rem:emb}
    A different description of $\cL_\gamma$ is given by the following explicit embedding
    \begin{align}\label{eq:handemb}      
    i_\gamma:[0,1]\times \Lambda &\to [\log a, \log b] \times Z \\
     i_\gamma(t,p) &= \left(\frac{1}{k}\gamma_l (t),e^{i\frac{1}{k}\gamma_\theta (t)}.p\right)
    \end{align}
\end{remark}

\begin{remark}[Generalization of BSP surgery]
    From  Remark \ref{rem:emb} one will notice that the construction of BSP can be naturally generalized when we have the following setup:
    Let $(Z,\lambda)$ be a contact manifold with a Legendrian $\Lambda$ such that the Reeb flow $\phi_t$ is periodic on $\Lambda$ with the same period $T$ i.e. $\phi_T|_{\Lambda} = Id$.  Further assume that there is a $t<T$ such that $\phi_t(\Lambda) = \Lambda$. Then we can use the Reeb flow  to define two fillings - one by going from time $0$ to $t$ and another by going from time $0$ to $-(T-t)$.
\end{remark}

\begin{lemma}
    The manifold $\cL_\gamma$ is a Lagrangian.
\end{lemma}
\begin{proof}
    There are multiple ways to see it. A direct way is to check that the embedding in Remark \ref{rem:emb} is a Lagrangian embedding. Indeed, we can see that for a fixed $t_0$, the embedding $i_\gamma(t_0,\_ )$ embeds $\Lambda$ to a phase-shift ( i.e. shifted by the circle action) of the Legendrian $\{\frac{\gamma_l(t_0)}{k}\} \times \Lambda$ in $\{\frac{\gamma_l(t_0)}{k}\}\times Z$.  Notice that we have $$d(e^s\lambda) (R,v) = 0,$$  $$d(e^s\lambda) (\partial_s,v)=0,$$ where $v \in T\Lambda$ and $R$ is the Reeb field. Thus, the result then follows from  the fact that  $(i_\gamma)_* (\partial_t)$ pairs trivially with $(i_\gamma)_* (v)$ for any $v\in T\Lambda$.
\end{proof}
\begin{prop}\label{prop:exact_in_neck}
    The BSP Lagrangian handle $\cL_\gamma$ is exact in $([\log a, \log b] \times Z, d(e^s\lambda))$. Furthermore, we can choose a function $f_\gamma$ on $[0,1]\times \Lambda$ such that $e^s\lambda|_{\cL_\gamma} = d (f_\gamma \circ i_\gamma\inv)$ where $i_\gamma $ is defined as in Equation \ref{eq:handemb}. In addition, we can choose $f_\gamma$ to depend only on the $[0,1]$ component, i.e.
    $f_\gamma (t,p) = \Tilde{f}_\gamma(t)$.
\end{prop}

\begin{proof}
    It is enough to check that for any loop $c$ representing a cycle $[c]$ in $H_1(\cL_\gamma),$ the integral $\int_c e^s\lambda$ vanishes. Let $c = (c_\bR, c_Z)$, we can homotope $c$ to ensure that $c_\bR$ is a constant map. Now use that fact that $e^{i\theta} \Lambda$ is a Legendrian for any $\theta \in \bR$ to conclude that $$\int_c e^s\lambda = 0.$$

    \noindent Thus, we can define a function $f_\gamma$ on $[0,1]\times \Lambda$ by integrating $e^s\lambda$ along paths such that 
    $$ d (f_\gamma \circ i_\gamma\inv) = e^s\lambda.$$Note that since $e^{i\theta}.\Lambda$ is a Legendrian, we can ensure $f_\gamma(t,p)= \Tilde{f}_\gamma(t)$. Moreover, we have the following explicit formula by choosing $(0,p) \in \{ 0\} \times \Lambda$ as a base point and integrating paths $\kappa: [0,1] \to [0,1]\times \Lambda$ such that $\kappa(0)= (0,p) , \kappa(1)=(s_q,q)$. We can homotope $\kappa$ to ensure that $\kappa|_{[0,1/2]}=(2ts_q,p)$ and $\kappa|_{[1/2,1] }$ is a path in $\{s_q \} \times Z$.  

    \begin{align}
    \;\;\int_\kappa i_\gamma^* e^s\lambda  &= \int_0^{1/2} \kappa^*  i_\gamma^* e^s \lambda + \int_{1/2}^1 \kappa^* 
  i_\gamma^* e^s \lambda\\
    &=\int_0^{1/2} \kappa^*   i_\gamma^* e^s \lambda + 0
    \end{align}
    Now note that $i_\gamma \circ \kappa (t) = \left (\frac{\gamma_l(2ts_q)}{k}, e^{i\frac{\gamma_\theta(2ts_q)}{k}}.p\right),$ thus we see that the integral $\int_\kappa i_\gamma^* e^s\lambda$ depends solely on $s_q$, thus there is a $\Tilde{f}$ such that $f_\gamma(t,q)=\Tilde{f}_\gamma(t)$.
\end{proof}

\begin{remark}\label{rem:ctnyingamma}
    Note that the proof above shows that the primitive $\Tilde{f}_\gamma$ of  $i_\gamma^* e^s\Lambda$  depends continuously on the path $\gamma$ and its derivative $\gamma'$.  By shrinking the interval $[a,b]$ in the embedding $\phi : [a,b] \times Z \to M$, we can assume that the action of Reeb chords bounding $\Lambda_2$ in $\{ b\}\times Z$ are rational. From the continuous dependence of $\widetilde{f}_\gamma$ on $\gamma$, we see that, by perturbing $\gamma$ appropriately, we can assume $\widetilde{f}_\gamma (1) - \widetilde{f}_\gamma (0) $ is a rational number. Thus, we can assume any disk with boundary on the closure  $ \overline{\cL_\gamma}$ of the BSP handle in the symplectic cutting of $[a,b]\times Z$ at $\{a,b\}\times Z$ has rational symplectic area. These rationality results will be pertinent while constructing broken divisors to make regularize  moduli space of SFT buildings using domain dependent perturbations.
\end{remark}

\subsection{Bohr-Sommerfeld Profile Surgery}

As we have already mentioned,  BSP surgery is obtained by switching exact fillings. We now describe a class of Lagrangians on which we can perform a BSP surgery. These Lagrangians are obtained by gluing Lagrangian fillings and cappings of the Legendrian $\Lambda_2$. 

\noindent Let $(M,\omega_M)$ be a compact symplectic manifold and $L$ be an embedded Lagrangian.

\begin{defn}[Lagrangians with Bohr-Sommerfeld Profile Handles]\label{bsphandle}
    We call a Lagrangian $ L \hookrightarrow M$ to have a Bohr-Sommerfeld Handle if there is a symplectic embedding $\phi : [a,b] \times Z \to M$, for a prequantization bundle $(Z,\lambda) \to (Y,\omega_Y)$ such that,
   $$\phi (\cL_\gamma) \sub L$$for an admissible curve $\gamma$. We say that $L$ has a handle modelled the curve $\gamma$ with profile $\Lambda.$
\end{defn}

A natural question at this stage is, how does the choice of the curve $\gamma$ affect the Hamiltonian isotopy class of a Lagrangian-with-BSP handle? We now investigate this line of thought.  We define the notion of relative Hamiltonian isotopy which lets us extend a Hamiltonian isotopy of $\cL_{\gamma_\tau}$ in $\phi([a,b]\times Z)$ to an isotopy of the closed Lagrangian $L_\tau.$

\begin{defn}[Relative Hamiltonian Isotopy]
Let $$\Phi: K\times [0,1] \to M $$ be a Lagrangian isotopy of a Lagrangian $K$ in $M$ which is constant on the boundary $\partial K$. We say that $\Phi$ is a relative Hamiltonian isotopy if there is a family of Hamiltonian functions $\{H_\tau\}_{\tau \in \interval}$ such that the isotopy induced from it is $\Phi$ and $H_\tau$ is equal to $0$ on the boundary $\partial K$, i.e.$$H_\tau (\partial K) = \{ 0 \}.$$
    
\end{defn}

\begin{lemma}
    Let $K$ be a Lagrangian in an open subset $U \sub (M,\omega)$ with boundary $\partial K = \partial U \cap K$. Let $L$ be a closed Lagrangian in $M$ such that $L \cap U = K$. Assume that $K \sub L$. Any relative Hamiltonian isotopy of $K$ in $U$ can be extended to a Hamiltonian isotopy of $L$ in $M$.
\end{lemma}
\begin{proof}
    From the definition of relative Hamiltonian isotopy, we know there is a family $H_\tau$ which realizes the isotopy of $K$ such that $H_\tau (\partial K) = 0.$ Note that changing $H_\tau$ outside a small normal neighborhood of $K_\tau$ still preserves the induced isotopy. We can extend  upto changing $H_\tau$ outside a small normal neighborhood of $K$, we can extend $H_\tau$ by defining it to be $0$ on $L$ and extending to the whole of $M$. The Hamiltonian isotopy from this extension would fix $L\setminus U$ and realize the isotopy of $K$ in $U.$
\end{proof}

We now explore how the choice of the curve $\gamma$ determines the relative Hamiltonian isotopy class of $\cL_\gamma$. We use the notion of exact Lagrangian isotopy as in \cite[\S~6.1]{pol01} to construct Hamiltonian isotopy. Define the following one-form on $(-\infty,\infty) \times S^1$ with coordinates $(s,\theta)$  , $$\lambda_k = e^{s/k} d\theta.$$ 

\begin{prop}
    Let $\gamma_\tau$ be an isotopy of embedded curves in the annulus $A(a,b)$. The isotopy of BSP Lagrangian handles, $\cL_{\gamma_\tau}$ is a relative Hamiltonian isotopy if and only if $$\int_{\gamma_\tau} \lambda_k $$ is constant for all $\tau.$
\end{prop}
\begin{proof}
    The proof technique has already appeared in \cite[Corollary~2.5]{lm10}, \cite[Lemma~2.16]{Cha23}. Consider the map $$\Phi: \cL_{\gamma_0} \times [0,1] \to [a,b] \times Z,$$ induced from the isotopy $\gamma_\tau.$ Since this is a Lagrangian isotopy, we have $$\Phi ^* d(e^s\lambda) = \alpha_\tau \wedge d\tau.$$ We will now show that under the hypothesis of this proposition, $\alpha_\tau$ is exact for all $\tau$ such that $\alpha_\tau = dH_\tau$ where $H_\tau|_{\partial \cL_{\gamma_\tau}} \equiv 0.$ We already know that for any curve $c:[0,1] \to \{r\} \times  e^{i\theta}.\Lambda$, we have $$\int_c e^s \lambda = 0.$$ We now check that $$\int_{c_\tau} \alpha_\tau =0$$ for a curve $c_\tau:[0,1] \to [0,1]\times \Lambda \cong \cL_{\gamma_\tau}$. We select a family of paths as follows, fix a point $p\in \Lambda$, 
    $$P:[0,1]\times [0,1] \to Z$$
    $$P(\tau,t) = i_{\gamma_\tau}(t,p). $$
    Note that   
    \begin{equation} \label{eq;diffgeom}
         \int_{P_\tau} \alpha_\tau = 0 \iff \frac{d}{d\tau}\int_{P_\tau} e^s\lambda = 0 
    \end{equation} 
    Write the isotopy of curves in cylindrical coordinates as  $\gamma_\tau = (\gamma^\tau_l, \gamma^\tau_\theta),$ then from Equation \ref{eq;diffgeom} we have  
    $$  \int_{P_\tau} \alpha_\tau = 0 \iff \frac{d}{d\tau}\int_0^1 e^{\gamma^\tau_l (t) /k} {{\gamma^\tau}'}_{\theta }(t) dt =0  \iff \int_{\gamma_\tau} \lambda_k \text{ is constant} $$
\end{proof}

\begin{figure}[ht]
    \makebox[\textwidth][c]{
   
        \def\svgscale{1.3}
\begingroup%
  \makeatletter%
  \providecommand\color[2][]{%
    \errmessage{(Inkscape) Color is used for the text in Inkscape, but the package 'color.sty' is not loaded}%
    \renewcommand\color[2][]{}%
  }%
  \providecommand\transparent[1]{%
    \errmessage{(Inkscape) Transparency is used (non-zero) for the text in Inkscape, but the package 'transparent.sty' is not loaded}%
    \renewcommand\transparent[1]{}%
  }%
  \providecommand\rotatebox[2]{#2}%
  \newcommand*\fsize{\dimexpr\f@size pt\relax}%
  \newcommand*\lineheight[1]{\fontsize{\fsize}{#1\fsize}\selectfont}%
  \ifx\svgwidth\undefined%
    \setlength{\unitlength}{450bp}%
    \ifx\svgscale\undefined%
      \relax%
    \else%
      \setlength{\unitlength}{\unitlength * \real{\svgscale}}%
    \fi%
  \else%
    \setlength{\unitlength}{\svgwidth}%
  \fi%
  \global\let\svgwidth\undefined%
  \global\let\svgscale\undefined%
  \makeatother%
  \begin{picture}(1,0.5)%
    \lineheight{1}%
    \setlength\tabcolsep{0pt}%
    \put(0,0){\includegraphics[width=\unitlength,page=1]{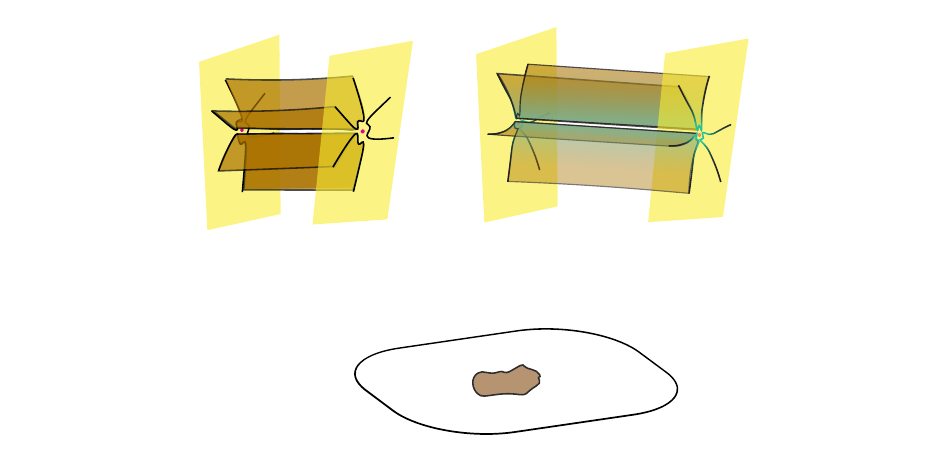}}%
    \put(0.34298157,0.2414182){\makebox(0,0)[lt]{\lineheight{1.25}\smash{\begin{tabular}[t]{l}$\cL_\gamma$\end{tabular}}}}%
    \put(0.71300949,0.24089335){\makebox(0,0)[lt]{\lineheight{1.25}\smash{\begin{tabular}[t]{l}$\cL_{\widetilde{\gamma}}$\end{tabular}}}}%
    \put(0,0){\includegraphics[width=\unitlength,page=2]{comp.pdf}}%
    \put(0.52093801,0.05723058){\makebox(0,0)[lt]{\lineheight{1.25}\smash{\begin{tabular}[t]{l}$L_Y$\end{tabular}}}}%
  \end{picture}%
\endgroup%

 }
    \caption{$\cL_\gamma$ and $\cL_{\widetilde{\gamma}}$ from the bundle perspective}
\end{figure}

\begin{figure}[ht]
    \makebox[\textwidth][c]{
   
        \def\svgscale{1}
\begingroup%
  \makeatletter%
  \providecommand\color[2][]{%
    \errmessage{(Inkscape) Color is used for the text in Inkscape, but the package 'color.sty' is not loaded}%
    \renewcommand\color[2][]{}%
  }%
  \providecommand\transparent[1]{%
    \errmessage{(Inkscape) Transparency is used (non-zero) for the text in Inkscape, but the package 'transparent.sty' is not loaded}%
    \renewcommand\transparent[1]{}%
  }%
  \providecommand\rotatebox[2]{#2}%
  \newcommand*\fsize{\dimexpr\f@size pt\relax}%
  \newcommand*\lineheight[1]{\fontsize{\fsize}{#1\fsize}\selectfont}%
  \ifx\svgwidth\undefined%
    \setlength{\unitlength}{449.17501831bp}%
    \ifx\svgscale\undefined%
      \relax%
    \else%
      \setlength{\unitlength}{\unitlength * \real{\svgscale}}%
    \fi%
  \else%
    \setlength{\unitlength}{\svgwidth}%
  \fi%
  \global\let\svgwidth\undefined%
  \global\let\svgscale\undefined%
  \makeatother%
  \begin{picture}(1,0.72332609)%
    \lineheight{1}%
    \setlength\tabcolsep{0pt}%
    \put(0,0){\includegraphics[width=\unitlength,page=1]{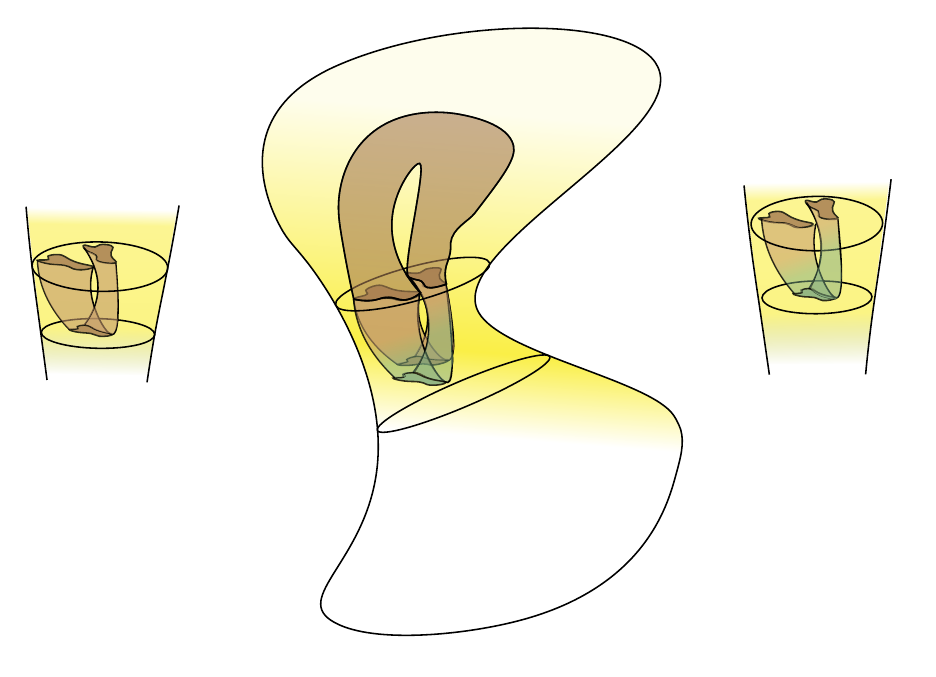}}%
    \put(0.0650115,0.27527389){\makebox(0,0)[lt]{\lineheight{1.25}\smash{\begin{tabular}[t]{l}$\cL_\gamma$\end{tabular}}}}%
    \put(0.85510617,0.29050185){\makebox(0,0)[lt]{\lineheight{1.25}\smash{\begin{tabular}[t]{l}$\cL_{\widetilde{\gamma}}$       \end{tabular}}}}%
    \put(0.51013521,0.10015289){\makebox(0,0)[lt]{\lineheight{1.25}\smash{\begin{tabular}[t]{l}$M$\end{tabular}}}}%
    \put(0.57338965,0.56606857){\makebox(0,0)[lt]{\lineheight{1.25}\smash{\begin{tabular}[t]{l}$L\sm\phi^{-1}(L)$\end{tabular}}}}%
    \put(0.62056333,0.3690561){\makebox(0,0)[lt]{\lineheight{1.25}\smash{\begin{tabular}[t]{l}$BSP(L)$\end{tabular}}}}%
    \put(0,0){\includegraphics[width=\unitlength,page=2]{bspsurg.pdf}}%
  \end{picture}%
\endgroup%

 }
    \caption{Bohr-Sommerfeld-Profile Surgery on $L$}
    \label{fig:bspsurg}
\end{figure}
\noindent We can now define a Bohr-Sommerfeld Profile-surgery.
\begin{defn}\label{def:bspsurg}
    Let $L$ be a Lagrangian with a handle modelled on $\gamma$ with profile $\Lambda \subset Z$. Choose a conjugate path $\wt \gamma$ of the curve $\gamma$. The Bohr-Sommerfeld Profile (BSP) surgery of $L$ with path modelled on $\wt \gamma$ is defined as 
    \begin{equation}
        BSP^{\Lambda}(L) = (L \sm \phi\inv ( L) )\cup _{  \phi(\partial \cL_\gamma) } \cL_{\wt \gamma} 
    \end{equation}
\end{defn}

\begin{defn} [Zero-Area BSP Surgery] \label{def:zeroarea}
    We call a BSP surgery to be of zero-area if the conjugate path $\wt \gamma$ satisfies
    \begin{equation}\label{eq:zeroar}        \int_\gamma \lambda_k = \int_{\wt \gamma} \lambda_k.
    \end{equation}
    We denote the result of such a surgery as $BSP^{\Lambda}_0 (L) \subset M.$
\end{defn}

\begin{defn}\label{def:reebfill}
    We call $\Lambda$ to be \textit{Reeb fillable} if every Reeb chord $c$ with $\Lambda_2$ endpoints can be filled to smooth map from  the half-disk with boundary on both $\cL_\gamma$ for any $\gamma$. In particular, there exists $u_\gamma:\bD^2_+ \to [a,b] \times Z$ such that $\partial u ( S^1_+) \sub \cL_\gamma$ and $\partial u([-1,1])=c$. Here $\bD^2_+$(resp.$ S^1_+$) denotes the intersection of the unit disk (resp. circle) with the closed upper half plane $\bH = \{ (x,y) | y\geq 0\}.$ 
\end{defn}
\begin{remark}
     Note that Reeb fillability is equivalent to the following condition: for every Reeb chord $c$ with endpoints in $\Lambda_2$, there is a smooth punctured disk, i.e. a map from $\bD^2 \sm \{ 1\}$, in $\bR \times Z$ with boundary on $\cL_\gamma$ which is asymptotic to $c$ near the boundary puncture.
\end{remark}

\begin{lemma}\label{ref:mono}
    If $L$ is a monotone Lagrangian in $M$, then a zero-area BSP-surgery with a Reeb fillable Legendrian profile  $BSP_0(L)$ is also monotone.
\end{lemma}
\begin{proof}
    The proof is similar to the proof of \cite[Lemma~2.21]{Cha23}.  We will show that for any disk $u:\bD^2 \to (M,BSP_0(L))$, there is a disk $u':\bD^2 \to (M,L)$ such that the Maslov number and symplectic action of $u,u'$ are same. This will be enough for the proof since we know $L$ is monotone.
    
    Let $[\theta_1,\theta_2]$ be a maximal arc in $\partial \bD^2$ such that $u([\theta_1,\theta_2]) \sub \phi(\cL_{\wt \gamma}).$ By smooth isotopy of $u$ we can assume that there is a closed domain $U_{[\theta_1,\theta_2]} \subset \bD^2 $ containing $[\theta_1,\theta_2]$ such that $\partial U = [\theta_1,\theta_2] \cup \alpha_{[\theta_1,\theta_2]}$  where $\alpha_{[\theta_1,\theta_2]}$ is a smooth curve in $\bD^2$ such that it gets mapped to a $Z$-slice under $u$, i.e., $u(\alpha_{[\theta_1,\theta_2]}) \subset \phi(\{ * \} \times Z)$. Since $Y$ is simply connected, we can isotope $u$ to arrange further that $\alpha_*$ always is a Reeb chord.
    
    We will construct $u'$ by redefining $u$ on $U_{[\theta_1,\theta_2]}$ such that $u'= u$ on $\alpha_{[\theta_1,\theta_2]}$. Note that since Bohr-Sommerfeld covers are Maslov 0 from \cite[\S~4.2]{rg19}, it is enough to check that such a $u'$ can be constructed. The equality of symplectic area of $u$ and $u'$ would then follow from  the exactness of $\cL_{\gamma}, \cL_{\wt\gamma}$ and 
     \begin{equation*}        \int_\gamma \lambda_k = \int_{\wt \gamma} \lambda_k.
    \end{equation*}
   Since we assumed $\Lambda $ was Reeb fillable, we see that such a $u'$ exists.
    
\end{proof}
\subsection{Examples and relations to other surgeries}

The primary motivation behind introducing Bohr-Sommerfeld-Profile surgery is to have a generalized surgery operation on Lagrangians which preserves the Maslov class. We now recover classical Polterovich surgery on Lagrangians from Bohr-Sommerfeld profile surgeries.

\subsubsection{Polterovich Surgery}
The setup for Polterovich surgery (\cite{pol:surg}) starts off with an immersed Lagrangian $L_0$ that has  a single immersed double point $p$. By a version of Darboux's theorem one can arrange a chart $(U,\psi)$ around $p$ such that $f\inv (L_0) = \bR^n \cup i\bR^n.$ Recall that a Polterovich surgery of $L_0$ at $p$ replaces $f\inv( L_0)$ with a Lagrangian handle $S^{n-1}\times [0,1]$ and there are two classes, $H_+, H_-$, of such handles.

\noindent A BSP surgery with $\Lambda= S^{n-1}$, $Y= \bC P^{n-1}, L_Y = \bR P^{n-1}$ is equivalent to changing swapping a class $H_+$ handle in Polterovich-surgery with a class $H_-$ handle.

\subsubsection{ Local Higher Mutation}
If we set $Y = \bC P^{n-1}, L_Y = T_{Cliff} $ and take the Clifford Legendrian tori as $\Lambda = \Lambda_{Cliff}$, a zero area BSP-surgery is equivalent to a local higher mutation as in \cite{Cha23}.

\subsection{Topology of BSP-Surgeries}\label{sec:topsurger}

Assume that the Bohr-Sommerfeld Legendrian  $\Lambda$ is a $k-$fold cover of $L_Y$ via the restriction of the projection map $Z \to Y$.  Then, the  action of the element $e^{i\frac{2\pi}{k}} \in S^1$ on $Z$ restricts to a diffeomorphism (and a deck-transformation over $L_Y$) of $\Lambda$. We denote this diffeomorphism as $\tau$.

\begin{lemma}
    If $\tau : \Lambda \to \Lambda$, the multiplication map defined by $p \mapsto e^{i\frac{2\pi}{k}}\cdot p$ to $\Lambda,$ is homotopic to the identity map, then the BSP-Surgery preserves the homeomorphism class, i.e., $$BSP(L) \cong_{homeo} L.$$
\end{lemma}
\begin{proof}
We can write any Lagrangian $L$ with BSP-Handle as $$L\sm \phi(\Lambda \times \interval) \bigcup (\Lambda \times \interval)/\sim, $$ where $(p,i) \sim \phi(p,i)$ for $i=0,1$. From the construction of the BSP-Surgery, we see that 
\begin{equation} \label{eq:bsptop}
BSP(L) \cong_{\text{Homeo}} L\sm \phi(\Lambda \times \interval) \bigcup (\Lambda \times \interval)/\sim, 
\end{equation}where $(p,0) \sim \phi(p,0)$ and $(p,1) \sim \phi(\tau(p),1)$. 
The Lemma then  follows from (\ref{eq:bsptop}).
\end{proof}

\begin{prop}
    If the second cohomology group of the Lagrangian $L_Y$ is torsion free and  $\Lambda \cong_{homeo} L_Y$, then Bohr-Sommerfeld-surgery  preserves the homeomorphism class, i.e.,
    $$BSP(L) \cong L$$
\end{prop}
\begin{proof}
    From  Lemma \ref{eq:bsptop} it is enough to show that the restriction of the multiplication by $e^{\frac{i2\pi}{k}}$ map (which we call $\tau$) to $\Lambda$ is isotopic to identity. Indeed, if it were not the case, then the circle bundle $\Lambda \times [0,1] / \sim $ where the equivalence relation $\sim$ is  $(p,0) \sim (\tau(p),1)$ over $L_Y$ will be a non-trivial circle bundle.  We know that the restriction $Z|_{L_Y}$ is a trivial bundle. Indeed, since $H^2(L_Y)$ is torsion-free and the Euler class of $Z|_{L_Y}$ is given by the curvature form $\omega_Y$ which vanishes on $L_Y$, we have that $[\omega_Y]= 0 $ in $H^2(L_Y,\bZ)$. This contradiction implies that $\tau$ is homotopic to identity and thus $$BSP(L) \cong L.$$
\end{proof}

\begin{remark}[BSP Surgeries that change the topological type]
    Note that if we take $L_Y$ to be an even-dimensional real projective space, $\bR P^{2n}$ as a monotone submanifold in $\bC P^{2n}$, then the Bohr-Sommerfeld surgery with profile the real Legendrian sphere $S^{2n}$ in $S^{4n+1}$ changes the topological type. This follows from the fact that the map $\tau$ in Lemma \ref{eq:bsptop} is the antipodal map, which is not isotopic to identity. This is a shadow of a similar phenomenon in Polterovich's surgery : Polterovich-surgeries of two odd-dimensional Lagrangians $L_1, L_2$ admit two distinct smoothings, $L_1 \# L_2 \# P$ and $L_1 \# L_2 \#  Q$. Here $$P= S^{2n} \times S^1, \; \; Q = S^{2n} \times [0,1]/\sim,$$ where the relation $\sim$ is the one which glues the ends using the antipodal map on the spheres. 
    
\end{remark}
\section{Zero Area BSP and Disk Potential} \label{sec:zero-dp}

As we have seen in Section \ref{sec:topsurger}, Bohr-Sommerfeld-Profile surgeries can construct homeomorphic Lagrangians. One of the main application of Bohr-Sommerfeld-Profile (BSP) surgery is to construct exotic symplectic classes of Lagrangians. We will study the effect of BSP surgeries on the disk-potential of oriented compact spin monotone Lagrangians. The disk-potential has been used to distinguish between the infinitely many Vianna tori \cite{Via16,Via17} and their lifts \cite{chw,DTVW}. We provide a change of potential formula for BSP-surgery under certain assumptions, see Assumptions \ref{ass:diskpot}.

\subsection{Background on disk-potential}

Assume $L_Y$ is a compact, oriented, spin monotone Lagrangian in $Y$. Choose $J$ to be a compatible almost complex structure on $Y$. We  denote the moduli space of $J$-holomorphic disks with boundary on $L_Y$ in class $\beta \in H_2(Y,L_Y)$ with one boundary marking by $\cM(\beta, J)$.   Let $\mu$ be the Maslov functional on $H_2(Y,L_Y)$.

For a generic choice of $J$, the moduli space of once boundary marked $J$-holomorphic disks of Maslov number 2 is a compact oriented manifold. There is a natural evaluation map, $ev_\beta$, associated with $\cM(\beta,J)$ which is defined by evaluation at the boundary marking. The \textit{disk potential} $W_{L_Y}$ is a polynomial obtained from counting Maslov 2  $J$-holomorphic disks passing through a generic point on $L_Y$. Choose a basis $(e_1,\dots,e_k)$ of $H_1(L_Y)$.
\[
\begin{aligned}
W_{L_Y} :\;& \mathrm{Hom}\big(H_1(L_Y), \mathbb{C}^*\big) \longrightarrow \mathbb{C}, \\
& (x_1,\dots,x_k) \longmapsto 
\sum_{\substack{\beta \in H_2(Y,L_Y) \\ \mu(\beta)=2}}
(\mathrm{sgn})\, n_\beta \, x^{\partial \beta}.
\end{aligned}
\]
where $(sgn)$ denotes a choice of sign based on the orientation of moduli space of disks,  $n_\beta$ denotes the degree of the evaluation map $ev_\beta$ and $x^{\partial\beta}$ denotes the monomial $\Pi_{1}^k x_i^{b_i}$ where $\partial \beta = b_1e_1 +\dots+b_ke_k.$ 

\noindent Recall that there is a natural pullback map $$p^*:\Hom(H_1(L_Y),\bC^*)\to \Hom(H_1(\Lambda),\bC^*) $$ induced from the projection map $p: \Lambda \to L_Y$. Let $u$ be a $J-$holomorphic disk which has a non-zero contribution to $W_{L_Y}$, then for $v=\partial u$, the polynomial 
\begin{equation}\label{eq:movedpot}
x^{-v} W_{L_Y} :
\mathrm{Hom}\big(H_1(L_Y), \mathbb{C}^*\big)
\longrightarrow \mathbb{C}
\end{equation}
\[
\text{lifts to}
\]
\begin{equation}
W_\Lambda :
\mathrm{Hom}\big(H_1(\Lambda), \mathbb{C}^*\big)
\longrightarrow \mathbb{C}.
\end{equation}
\noindent This follows from the fact that if $v'$ is another cycle bounding a Maslov two disk, then the cycle $v' - v $ has a horizontal lift\footnote{Here we are using the monotonicity of our set up.} to a cycle in $\Lambda$. We call $W_\Lambda$ the augmentation polynomial of the lifted Legendrian $\Lambda$. See \cite[\S 2.1]{bcsw3} for results relating the augmentation polynomial with the Chekanov-Eliashberg algebra, $CE(\Lambda),$ of the Legendrian $\Lambda.$

\begin{example}
    The disk-potential of the Clifford torus $T^n_{Cl} \hookrightarrow \bC\bP^n$ is the polynomial $$W_{Cl}(x_1,\dots,x_n) = x_1 + x_2 + \dots x_n + \frac{1}{x_1\dots x_n}. $$ The corresponding augmentation polynomial of the Clifford Legendrian torus $\Lambda_{Cl} \hookrightarrow S^{2n+1}$ is obtained as follows,
\[
\begin{aligned}
\left(\frac{1}{x_1 \cdots x_n}\right)^{-1} W_{Cl}
&= 1 + x_1^2 x_2 \cdots x_n 
   + x_1 x_2^2 \cdots x_n 
   + \dots, \\
\text{set } y_k \text{ such that } 
\quad p_*(y_k) &= x_1 \cdots x_k^2 \cdots x_n, \\
\text{then} \quad
W_{\Lambda_{Cl}} &= 1 + y_1 + \dots + y_n.
\end{aligned}
\]
Here $y_k$ are cycles in $T_{Cl}$ that lift to a basis of the Legendrian tori.
\end{example}

\begin{example}
The disk-potential of the lifted $T_{(1,1,2)}$ tori is given by $$W_{\ol T_{(1,1,2)}}(x_1,\dots,x_n) = x_2 + x_3+\dots +x_n+ \frac{(1+x_1)^2}{x_1x_2^2x_3\cdots x_n},$$ see \cite[Example 4.2]{chw}.
We see that $$W_{\Lambda_{(1,1,2)}} = (y_2 + \dots + y_n + (1+y_1)^2).$$ 
\end{example}

\begin{defn}[Regular Almost Complex Structure for Disk-Counting]
    We call an almost complex structure $J$ on $Y$ to be regular for disk-counting in $L_Y$ if the moduli space of $J-$holomorphic disks with boundary on $L_Y$ is regularly cut out for all classes of Maslov index less or equal to two. 
\end{defn}


\begin{assum}\label{ass:diskpot}
We state the assumptions on the Lagrangian $L$ which are essential for our proof strategy :
\begin{enumerate}[{A}.1]
    \item $L$ has minimal Maslov number 2, is compact, oriented, spin and monotone.
      \item \label{p3} $\Lambda$ is a Bohr-Sommerfeld lift of a Lagrangian $L_Y$ that has minimal Maslov number 2.
    \item The first homology has the following splitting $H_1(L) = i_*H_1(\Lambda \times \{0\}) \oplus \bZ\langle \zeta \rangle \oplus R$ where $i$ is the inclusion of the BSP Handle, $\zeta$ is Poincare-dual to $\Lambda\times \{ 0\}$ and $R$ is any torsion free abelian group. \label{spl}
    \item  $L\sm \phi(\Lambda \times \interval)$ is connected.
    \item $L$ allows a zero-area BSP surgery
    \item \label{p5}The deck-transformation $\tau$ in (\ref{eq:bsptop}) is homotopic to identity, thus BSP-surgery preserves the homeomorphism type of $L$.
\end{enumerate}
\end{assum}

\begin{remark}
One can recover a similar disk potential relation by relaxing  A.\ref{p5} above which allows the for the homeomorphism class of the Lagrangian to change while preserving the splitting property in A.\ref{spl}. We do not have applications of such a result in our article so we avoid pursuing such a result.
\end{remark}

\noindent For the remainder of the section we will suppose that the Lagrangian $L$, on which we perform Zero-BSP surgery, satisfies Assumption \ref{ass:diskpot}. Notice that if $L$ satisfies all the assumptions in Assumptions \ref{ass:diskpot}, the zero-area BSP surgery, $BSP_0(L)$ also satisfies them.

\subsection{Convention for first homology}
We  select a splitting of $H_1(BSP_0(L))$ corresponding to the splitting in A.2 which  will be used to state the change in disk-potential formula. The $\zeta$ component in the splitting A.2 is of main importance. By a slight abuse of notation, let $\zeta : S^1 \to L$ be the cycle representing $\zeta$ such that it intersects $i(\Lambda \times \{1 \})$ at $p$, say $\zeta (\theta_p)  =  p .$ Note
\begin{equation}\label{eq:surghom}
    H_1(BSP_0(L)) = i_*(H_1(\Lambda \times \{0\}) \oplus  \bZ\langle \widetilde\zeta \rangle  \oplus R,
    \end{equation}

where $i$ is the inclusion of the BSP Lagrangian handle, $\widetilde \zeta$ is the cycle obtained from $\zeta$ in $L$ by attaching a path between $\tau (p)$ to $p.$ More precisely, note that $\zeta : S^1\sm \{ \theta_p \} \to L$ uniquely determines a path $\widetilde\zeta : S^1\sm \{ \theta_p \} \to BSP_0(L)$ whose left and right-handed limits at $\theta_p$ are $p$ and $\tau (p).$ We extend $\widetilde\zeta$ to a cycle by attaching a path $\psi: [0,1] \to \Lambda $ from $p$ to $\tau p.$ We call this path $\psi$  a capping path.

\begin{figure}[ht]
    \makebox[\textwidth][c]{
   
        \def\svgscale{1}
\begingroup%
  \makeatletter%
  \providecommand\color[2][]{%
    \errmessage{(Inkscape) Color is used for the text in Inkscape, but the package 'color.sty' is not loaded}%
    \renewcommand\color[2][]{}%
  }%
  \providecommand\transparent[1]{%
    \errmessage{(Inkscape) Transparency is used (non-zero) for the text in Inkscape, but the package 'transparent.sty' is not loaded}%
    \renewcommand\transparent[1]{}%
  }%
  \providecommand\rotatebox[2]{#2}%
  \newcommand*\fsize{\dimexpr\f@size pt\relax}%
  \newcommand*\lineheight[1]{\fontsize{\fsize}{#1\fsize}\selectfont}%
  \ifx\svgwidth\undefined%
    \setlength{\unitlength}{450bp}%
    \ifx\svgscale\undefined%
      \relax%
    \else%
      \setlength{\unitlength}{\unitlength * \real{\svgscale}}%
    \fi%
  \else%
    \setlength{\unitlength}{\svgwidth}%
  \fi%
  \global\let\svgwidth\undefined%
  \global\let\svgscale\undefined%
  \makeatother%
  \begin{picture}(1,0.5)%
    \lineheight{1}%
    \setlength\tabcolsep{0pt}%
    \put(0,0){\includegraphics[width=\unitlength,page=1]{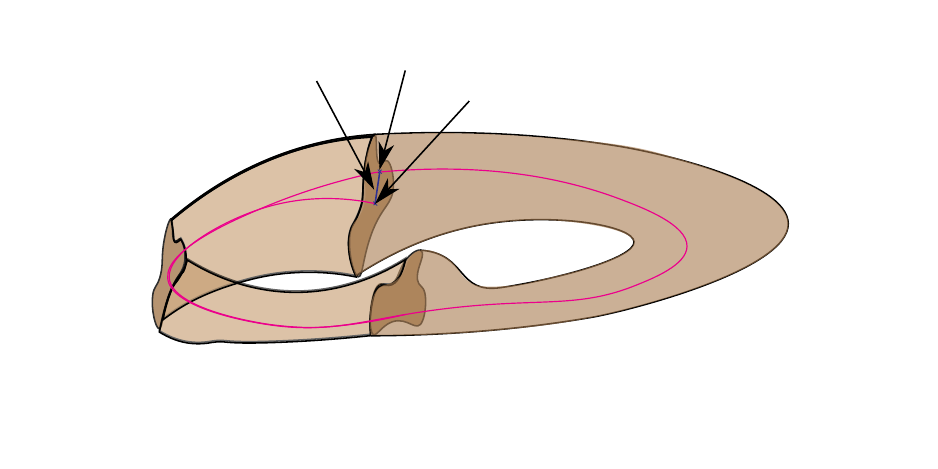}}%
    \put(0.32719152,0.42015634){\makebox(0,0)[lt]{\lineheight{1.25}\smash{\begin{tabular}[t]{l}$\psi$\end{tabular}}}}%
    \put(0.4265773,0.43243998){\makebox(0,0)[lt]{\lineheight{1.25}\smash{\begin{tabular}[t]{l}$p$\end{tabular}}}}%
    \put(0.49972087,0.40228923){\makebox(0,0)[lt]{\lineheight{1.25}\smash{\begin{tabular}[t]{l}$\tau(p)$\end{tabular}}}}%
    \put(0,0){\includegraphics[width=\unitlength,page=2]{homcyc.pdf}}%
    \put(0.17392518,0.39544947){\makebox(0,0)[lt]{\lineheight{1.25}\smash{\begin{tabular}[t]{l}$\zeta$\end{tabular}}}}%
    \put(0.24811559,0.41164154){\makebox(0,0)[lt]{\lineheight{1.25}\smash{\begin{tabular}[t]{l}$\widetilde{\zeta}$\end{tabular}}}}%
  \end{picture}%
\endgroup%

 }
    \caption{Bohr-Sommerfeld-Profile Surgery on $L$}
\end{figure}

\begin{conv}\label{conv:homo}
    We will choose the path $\psi(t)$ to be equal to $l(1-t)$ where $l$ is the lift of the boundary $\partial u,$ of a Maslov 2 disk $u$ which has non-zero contribution in the disk potential $W_{L_Y}$.
\end{conv}
\subsection{Main theorem}
We can now state the main theorem of the article, which is a precise version of Theorem \ref{mainthmintro}. 
\begin{Theorem}[Wall-crossing formula for Zero-Area BSP]\label{mainthm}
    Let $L$ be a Lagrangian satisfying Assumptions \ref{ass:diskpot}. Choose a splitting of the first homology of $L$ as in property A.2 and choose a corresponding splitting of the first homology of $BSP_0(L)$ following Convention \ref{conv:homo}. Let $W_\Lambda$ denote the augmentation polynomial by taking the lift of $x^{-\del u} W_{L_Y}$ where $u$ is the Maslov-2 disk whose boundary is used as capping path in Convention  \ref{conv:homo}. Then we have the following wall-crossing formula for disk-potential,
    \begin{equation}\label{mainform}
        W_{BSP_0(L)}(x_1,\dots,x_k,z,w_1,\dots,w_l)
 =W_L(x_1,\dots,x_k,zW_{\Lambda}(x_1,\dots,x_k),w_1,\dots,w_l),     \end{equation}
 where the variables $x_1,\dots,x_k$ correspond to $\Hom(i_*(\Lambda \times \interval),\bC^*)$, $z$ corresponds to $\Hom(\bZ \langle \zeta \rangle, \bC^*)$ or $\Hom(\bZ \langle \zeta \rangle, \bC^*)$ and $w_i$ correspond to $\Hom(R,\bC^*).$
\end{Theorem}

\begin{example}[Lagrangian Disk-Surgery]\label{ex:disksurg}
    When we perform a zero-BSP surgery with profile $\Lambda = S^1 = S^3 \cap \bR^2 \hookrightarrow S^3$ where we view $S^3$ as a unit sphere in $\bC^2$, $Y= \bC P^1$, $L_Y = \bR P^1$, our wall-crossing formula matches those which are already in the literature, see \cite{Aur07,PW19,ch10}. Indeed, notice that performing such a zero area BSP is equivalent to performing a Lagrangian disk-surgery and the change of potential is given by $$W_{BSP_0(L)}(x,y,w_1,...) = W_L (x,y(1+x),w_1,...),$$ where the augmentation polynomial of the Legendrian $S^1 \sub S^3$ is given by  $ W_{\Lambda} (x) = (1+x)$.
\end{example}
\begin{example}[Local Higher Mutations] \label{ex:highermut}
Performing a zero-BSP surgery with the Clifford Legendrian profile $\Lambda = \Lambda_{Cl}, Y=\bC P^{n-1}, L_Y = T^{n-1}_{Cl}$ is equivalent to performing a local higher mutation, see \cite{PT20,Cha23}. The wall-crossing formula in equation (\ref{mainform}) matches with the results already in the literature,
    $$W_{BSP_0(L)}(x_1,\dots,x_{n-1}, z, w_1,\dots) = W_L(x_1,\dots,x_{n-1},z(1+x_1 + \dots x_{n-1}), w_1,\dots)).$$
\end{example}
\begin{example}[BSP with lifted Vianna tori profiles]
A novel application of Theorem \ref{mainthm} is understanding how performing zero-area BSP-surgery with profile $\Lambda=\Lambda_{\abc}$,  $Y=\bC P^{n-1}, L_Y = T^{n-1}_{\abc}$ where $T^{n-1}_\abc $ is a lifted Vianna tori as described in \cite{chw}. See \S \ref{sub:exonvia} for further details.
    
\end{example}

\section{Technical SFT Results}\label{appendix}

\subsection{Background on Neck-Stretching} \label{sftbackground}

We will discuss the necessary background for neck-stretching and holomorphic buildings or broken holomorphic disks. This treatment is very similar to the set up in \cite{floflip,Cha23} and should be treated as  a quick recap. Refer to \cite{BEH03,floflip,CM05,Abbasbook} for more details about neckstretching and SFT compactness.  

\begin{assum}\label{assneck}[Neck-stretching assumptions]
Let $(M,\omega)$ be a compact symplectic manifold with a coisotropic separating hypersurface $Z$ with a neighborhood embedding $\phi: [a,b]\times Z$ of a truncated symplectization where   $(Z,\lambda) \to (Y,\omega_Y)$ is a prequantization bundle. Let $L$ be a compact Lagrangian such that $L \cap \phi ((b-\epsilon,b]\times Z)  = \Lambda_2$ for a Bohr-Sommerfeld lift $\Lambda$ in $Z$. Let $J$ be a compatible almost complex structure on $M$ such that under $\phi$ it is translation invariant in the $\bR$ coordinate. Moreover assume that the restriction of  $J$ to $\ker \lambda$ in $\phi((a,b)\times Z)$ is compatible with $\pi_Y^* \omega_Y$ and $J(\dd{s}) = R$ where $s$ is the $\bR$ coordinate and $R$ is the Reeb field on $Z$. We denote the space of Reeb chords with end points mapping to the Legendrian $\Lambda_2$ as $\cR(\Lambda_2)$

\end{assum}

We will perform a simultaneous neck-stretching at $\phi(\{a',b'\} \times Z )$, two disjoint coisotropic embedding of $Z$.  

\begin{defn}[Double-stretched manifold]
Given Assumption \ref{assneck}, choose $a'>a$ and $b'<b$ and $\epsilon >0$ such that $(a'-\epsilon , a'+\epsilon) \cup (b'-\epsilon, b'+\epsilon) \subset (a,b)$    A T-double stretched manifold $M^T$ is defined as 
\begin{equation*}
    M^T= M\sm \phi (((a'-\epsilon , a'+\epsilon) \cup (b'-\epsilon, b'+\epsilon)) \times Z) \bigcup_{\sim^1} ([-T,T] \times Z )\bigcup_{\sim^2}( [-T,T] \times Z  )
\end{equation*}
where $\sim^1$ is the pair of identifications $\phi(
\{a'-\epsilon\}\times Z
) = \{ -T\} \times Z$
and $\phi(
\{a'+\epsilon\}\times Z
) = \{ T\} \times Z$ and $\sim^2$ is the similar pair of identifications at $\{b'\pm \epsilon \} \times Z$. We define a T-double-stretched almost complex structure $J^T$ by extending the cylindrical almost complex structure naturally to $[-T,T] \times Z$.

\end{defn}

\noindent We define a counterpart of broken manifolds as defined in \cite{floflip}. The main difference is that we perform simultaneous stretching. 

\begin{defn}[Broken Manifold]
Define $\bM [l,k]$ to be the  collection of manifolds 

$$\widehat{M}_\outm, \underbrace{M_{-l}:=\bR \times Z ,\cdots, M_{-1}:=\bR \times Z}_{l \text{ times}} , {M}_\inn := \bR \times Z , \underbrace{M_1 := \bR \times Z ,\cdots,M_k := \bR \times Z}_{k \text{ times}}, \widehat{M}_\outp  $$
where $\widehat{M}_\out$ is obtained by attaching natural symplectization ends to $M \sm \phi([a,b] \times Z).$ The manifolds $\bR \times Z$, between $M_\inn$ and $\M_{out\pm}$  are called the neck pieces. If they are above or below the inside piece $M_\inn$, they are referred to as lower or upper neck pieces.

In case $Z$ was a separating hypersurface, we call the two disjoint components of $\widehat{M}_\out$ to be $\widehat{M}_\outm$ and $\widehat{M}_\outp$ depending on which end of symplectization is added.
\end{defn}

\begin{remark}
    Note that we use two indices to keep track of neck pieces because we are performing stretching at two regions simultaneously.

\end{remark}
\begin{remark}
    
It is sometimes helpful to consider the broken manifolds as collection of compact manifolds obtained by attaching the divisor $Y$ at the ends of each piece. Thus, $\widehat{M}_{out \pm}$ are compactified to manifolds obtained by performing symplectic cut at $\{a\} \times Z$ and $\{b\} \times Z$. Similarly $\widehat{M}_\inn$ is compactified to a $\bP^1$ bundle over $Y$ obtained by projectvizing the bundle $E\oplus \bC$ where $E$ is the associated complex bundle on $Y$ with the circle bundle $Z$. The compactified pieces in a broken manifold are notated with an overline, e.g. the compactified inner piece is $\overline{M}_\inn$.
\end{remark}

\begin{defn}[Broken and stretched Lagrangian]
    If $L$ is a Lagrangian with BSP handle $\cL_\gamma$ in $\phi([a,b] \times Z)$, we can define a T-stretched Lagrangian by adding in appropriate cylindrical Lagrangians. We can define broken Lagrangians similarly as broken manifolds. Note that the neck regions of broken Lagrangians lie only in the upper neck regions.
\end{defn}

 \begin{defn}[Broken Disk]\label{def:brokdisk}
 
 A broken disk with $k-$necks ( or in the case of double stretching,  $(l,k)-$necks ) $(\mathfrak{u,D})$ is a map from a  nodal disk with marked points $\mathfrak{D}$ to a broken manifold with $k$ (resp. $(l,k)$) necks with the following properties. 
 \begin{itemize}
     \item $\mathfrak{D}$ is a nodal disk with labels $i\in \{\text{out-},-l, \dots,-1,\text{in}=0,1,2,\dots,k,\text{out+}=k+1\}$ such that neighboring disks have labels differing at most by 1. We call the nodes that are between two components with different labels as ``puncture nodes''. We write the nodal disks as  $ \mathfrak{D} = (\bD_\outm,\bD_{-l}\dots, \bD_{in},\bD_1,\dots,\bD_{out+})$ where $\bD$ refers to a collection of nodal disks with punctures. 
     \item  $ \mathfrak{u} = (u_\outm,u_{-l}, \dots , u_{in}, u_1,\dots,u_{out})$ is a map from $\mathfrak{D}$ to $\mathbb{M}[l,k]$ such that  $u_i$ is a map from $\D_i$ to $W_i$ with boundary on $L_i$ and near the puncture nodes of $\mathfrak{D}$ between components of differing label $u_i$ and $u_{i+1}$ asymptote to the same Reeb chord or Reeb orbit. This forces all the domains $\bD_\outm,\bD_{-l},\dots, \bD_{-1}$ to be collections of punctured nodal spheres since the boundary condition is empty in these regions.
     \item (\textit{Stability}) for all $i \in [1,\dots,k], j\in [-l,\dots,-1]$, at least one of the component in each neck $\bD_i$ and $\bD_j$ has to be a non-trivial strip, or have enough markings to make it stable domain.
 \end{itemize}
  \end{defn}
\begin{remark}
    Broken disks are just holomorphic buildings in SFT where the domain is a nodal, we choose this naming convention over usual SFT nomenclature for the sake of brevity.
\end{remark}

Broken disks occur as limits of neck stretching when we let the neck length go to infinity. In the statement of the theorem, by energy $E(u)$ of a holomorphic disk, we mean the symplectic area $\int_u \omega$.  The following compactness theorem has already appeared in the SFT literature in different formats, see \cite{BEH03,Abbasbook,CM05,floflip,cant2022dimension}. For a proof which is enough for our situation, refer to \cite
[\S 3]{Cha23}.
\begin{Theorem}[SFT Compactness]
If $u_n : \D^2 \to (M
^n,J^n,L^n) $ is a sequence of $J^n$ holomorphic disc to the $n-$double-stretched manifolds, with boundary on $L^n$, and with bounded symplectic energy, i.e. $E(u) < E< \infty$, then there is a subsequence of $u_n$ that converges to a stable broken disc.
\end{Theorem}
\begin{proof}
    Although the only difference of this setup from Theorem 3.13 in \cite{Cha23} is the simultaneous stretching of two necks, the proof follows almost verbatim. The only change needed is to keep track of labels of upper and lower neck pieces, which is omitted here. 
    \end{proof}
\subsection{Fredholm Theory for Broken Disks}\label{sec:Freddy}

We will setup the required Fredholm theory to cut out the moduli space of broken disks as a zero set of a Fredholm section. Since we have punctures going to cylindrical ends, we use weighted Sobolev spaces to create Banach spaces of maps.

Let $\Sigma$ be a punctured disk with punctures at $\cE(\Sigma)$. For each puncture $p\in \cE(\Sigma)$ we choose a Reeb chord $\gamma_p$ or orbit depending on whether $p$ is on boundary or interior. Equip $\Sigma$ with cylindrical  holomorphic charts $\phi_p$ (i.e. domain of the chart is $[0,1]\times [0,\infty)$ or $S^1\times[0,\infty)$ ) near each puncture $p$. We define the set of \textit{model maps} as
follows,

\[
\begin{aligned}
\Maps^0(\Sigma, \mathbb{R} \times Z, \mathbb{R} \times \Lambda_2)
= \Big\{ \,
& u \in C^0(\Sigma, \mathbb{R} \times Z)
\ \Big|\ 
u(\partial \Sigma) \subset \mathbb{R} \times \Lambda_2, \\
& u(\phi_p(s,t)) = (k s, \gamma_p(k t))
\ \text{for all } p \in \mathcal{E}(\Sigma)
\, \Big\},
\end{aligned}
\]
\[
\text{where } k \text{ is the action of } \gamma_p.
\]
This set consists of all continuous maps which are equal to trivial strips/cylinders in each of the cylindrical end of $\Sigma$.

Select a cylindrical Riemannian metric $g$ on $\bR \times Z$ and a cylindrical connection $\nabla$ on the tangent bundle.  Equip $\Sigma$ with a 
metric which is of cylindrical form on the charts at each puncture. For every end $p \in \mathcal{E}(\Sigma)$ we will allow Sobolev sections which exponentially decay to parallel transport of sections over the asymptotic Reeb chord or orbit at the puncture, $\eta(p) \in \Omega^0(\gamma_p^*(TZ) \oplus \mathbb{R})$.

\noindent Note that for any $u\in \Maps^0(\Sigma,\bR \times Z , \bR \times \Lambda_2)$ which we write as $(u_\bR, u_Z)$ using the standard projection maps in $\bR\times Z$ , there is a natural splitting $$u^*(T (\bR \times Z)) = \overline{\bR} \oplus u_Z^* TZ $$
where $\overline{\bR}$ is the trivial bundle. Given any section, $\eta(p) \in \Omega^0(\overline{\bR} \oplus \gamma_p^* TZ),$ we can construct a section over the cylindrical chart $\phi_p$ in $\Sigma$ by taking parallel transport $$PT_\nabla(\pi_Z (\eta(p)))\in \Omega^0(u_Z^* TZ)$$ of the $\gamma_p^* TZ$ component along the compactification of the path $u_Z(s,\cdot)$. Recall that from the exponential decay of $u_Z$ to $\gamma_p$ we have the existence of the compactification of the path $u_Z(s,t)$ for every $t.$ Now from the triviality of the bundle $\overline{\bR}$, we have a natural extension of the $\overline{\bR}$ component $\eta(p)$ to the cylindrical chart near $p$. To get a section $ES(\eta(p))$ over $\Sigma$ we multiply an appropriate bump function which is equal to 1 near the puncture $p$ and is 0 outside the cylindrical chart at $p$. 

$$ES(\eta(p))(s,t) = \beta(s)(\pi_\bR (\eta(p)) (t) + PT_\nabla(\pi_Z(\eta(p))) ) \in \Omega^0(u^*(T(\bR \times Z))),$$
where $\beta : [0,\infty) \to \bR$ is a bump function such that $\beta(s)=0 $ for $s<1$ and $\beta(s) =1 $ for $s>2.$

Define the space of sections of a Reeb chord with boundary along Legendrian as follows,
$$
  \Omega^0(\gamma_p^*(TZ) \oplus \mathbb{R}, \partial \gamma_p ^* T\Lambda)  := \{ \eta(p) \in \Omega^0(\mathbb{R} \oplus \gamma_p^* TZ  ) |   \eta(e)(\gamma_p(\{0,1 \})) \subset  \mathbb{R} \oplus T\Lambda_2 \} . 
$$

\noindent From now we assume that any locally $W^{k,p}$ section $\eta$ of $u^*T(\bR \times Z)$ comes with a choice of $\eta(p) \in \Omega(\overline{\bR} \oplus \gamma_p^*{TZ},\partial \gamma_p^* T\Lambda_2) $ for each puncture $p\in \cE(\Sigma).$  For Sobolev constants $k, p$ with $kp>2$ and $u\in \Maps^0(\Sigma, \bR\times Z, \bR \times \Lambda_2)$, consider the weighted Sobolev space

\[
W^{k, p, \delta}(u^*T(\bR \times Z), \partial u^*T\Lambda_2):=\{ \eta\in W^{k, p, \delta}_{{loc}}(u^*T(\bR \times Z))\mid \| \eta\|_{k,p,\delta}< +\infty, \eta|_{\partial \Sigma} \in T\Lambda_2 \}.
\]
Here the Sobolev $(k,p,\delta)$-norm is defined as
\[
\lvert \lvert \eta\rvert \rvert^p_{k,p,\delta}:= \sum_{p'\in \cE(\Sigma)} \lvert \eta(p') \rvert^p + \int_{S} (\lvert \eta - ES(\eta(p'))\rvert^p + \lvert \nabla \eta\rvert^p+ \cdots + \lvert \nabla^{k} \eta\rvert^p)e^{p\kappa^{\delta}}.
\]
where $ES(\eta(p'))$ the section with support in the cylindrical and strip-like end for each puncture $p'$ obtained by taking parallel transport of a section defined above and $\kappa^\delta$ is a function defined such that

\begin{align*}
    &\kappa^\delta_\Sigma : \Sigma \to \R \\
    &\kappa^\delta_\Sigma (z) = 0 \text{ for  } z \in \Sigma \setminus \bigcup_i im(\phi_p) \\
    &\kappa^\delta_\Sigma (\phi_p (s,t)) = \delta \beta( s) \text{ for } p  \in  \cE(\Sigma).
\end{align*}

\noindent We define the Banach manifold of Sobolev $(k,p,\delta)$ maps as follows 
\begin{align*}
\Maps_{k,p,\delta}(\Sigma,\bR \times Z,\bR \times \Lambda_2) = \{ \exp_u(\eta) | u\in \Maps^0(\Sigma,\bR \times Z,\bR \times \Lambda_2),\\ \eta \in W^{k, p, \delta}(u^*T(\bR \times Z), \partial u^*T\Lambda_2) \}
\end{align*}
where the exponentiation is taken with respect to a metric for which $\bR \times \Lambda_2$ is totally geodesic. We use geodesic exponentiation to define charts of the space $\Maps_{k,p,\delta}$ to give it a structure of Banach manifold.

To facilitate the use domain dependent perturbations of the almost complex structures, we will cut out our moduli space of broken disks from a product of space of domains with the space of maps, which we have already described above. This discussion follows the setup of domain dependent perturbations as described in \cite{floflip,PW19,bcsw1}.

By a \textit{combinatorial type} $\Gamma$ of a disk, we mean the data of  interior marked points $\cI$, puncture node $\cE$ , and by $\overline{\cM}_\Gamma$ we denote the compactified moduli space of disk domains of combinatorial type $\Gamma$. The interior marked points will be the points which get mapped to a Donaldson divisor $D$, see Definition \ref{def:div}. We can cover $\cM_\Gamma$ by charts $\cM_{\Gamma_i}$ for $i\in \cA$ such that the universal moduli space $\cU_{\Gamma_i}$ above them admit trivializations. For each chart $\cM_{\Gamma_i}$, we can construct a base space $$\cB_{\Gamma_i}  := \cM_{\Gamma_i} \times \Maps_{k,p,\delta} (\Sigma, \bR \times Z, \bR \times \Lambda_2).  $$
Over $\cB_{\Gamma_i}$ we have a vector bundle $$\cE_{\Gamma_i}|_{(C,u)} = W^{k-1,p,\delta}(\overline{\Hom}(T\Sigma,u^*(T\bR \times Z))) $$ and a section $\cF_{\Gamma_i} : \cB_{\Gamma_i} \to \cE_{\Gamma_i}$ given by $$\cF_{\Gamma_i}(C,u) = \overline{\partial}_{j(C)}(u).$$ We have that $\cF$ is a Fredholm map near its zero locus and denote the linearization at $C,u$ as $\Tilde{D}_{C,u}$. For a fixed $C\in \cM_\Gamma$, we define $D_u$ to be the linearization of the restriction of the map $\cF_\Gamma : \{ C\} \times \Maps_{k,p,\delta}$.

Refer to \cite{wendl:sft,cant2022dimension,Cha23} for proofs of Fredholmness of the linearization. We have a number of evaluation maps from the zero locus of $\cF_{\Gamma}$. 

\[
\begin{aligned}
\mathrm{ev}_{\mathrm{int}} :\;& \mathcal{F}^{-1}(0) \longrightarrow M^{|\mathcal{I}|}, \\
& u \longmapsto \big( u(z_m) \big)_{m \in \mathcal{I}},
\end{aligned}
\]
\[
\text{where } z_m \text{ denotes the } m\text{-th interior marked point.}
\]
We denote $\cM_{\Gamma}(\bR\times Z, \bR \times \Lambda_2, D)$ to be the set $ev_{int}\inv(D^{|I|})$ of holomorphic maps of combinatorial type $\Gamma$ which map the interior marked points to the Donaldson divisor by taking the union over all of the trivializing charts $\cM_{\Gamma_i}$. We similarly define moduli space of holomorphic disks in the inner and outer piece. We have an evaluation map to the starting point of Reeb chord/orbit at each puncture
\[
\begin{aligned}
\mathrm{ev}_{\mathrm{rc}} :\;& \mathcal{F}^{-1}(0) \longrightarrow \Lambda_2^{|\mathcal{E}|}, \\
& u \longmapsto \big( \gamma_e(0) \big)_{e \in \mathcal{E}},
\end{aligned}
\]
\[
\text{where } \gamma_e \text{ is the Reeb chord/orbit at the puncture } e.
\]

\noindent Given combinatorial types for each level of a broken disk with $(l,l')$ levels, $$ \overline{\Gamma}=(\Gamma_\outm,\Gamma_{-l},\dots,\Gamma_\in,\Gamma_1,\dots,\Gamma_{l'},\Gamma_\outp),$$ the moduli space of broken disks of combinatorial type $\overline{\Gamma}$ is the fiber product over Reeb chord/orbit evaluation $$\cM_{\overline{\Gamma}}(\bM[l,l']) = \cM_{\Gamma_\outm}(M_\outm)\times_{ev_{rc}}\dots \cM_{\Gamma_i}(M_i) \times_{ev_{rc}} \dots  $$
where we have dropped the notation for Lagrangian boundary and interior markings going to Donaldson divisors just to avoid clutter.
\begin{defn}[Virtual dimension]
For a broken a disk $(\mathfrak{u,D})$, the virtual dimension is given by the formula $$\vir (\mathfrak{u}) = \Ind(\wt D_\mathfrak{u}) - N - \dim \aut(\mathfrak{D})$$ where the contribution $N$ comes from taking fiber product at the puncture nodes.
\end{defn}

\begin{defn}[Rigid broken disk]
We call a broken disk $\mathfrak{u}$ rigid if $\vir(\mathfrak{u}) = 0$. 
\end{defn}

In our setup we will be concerned with counting rigid broken disks which pass through a point fixed in $L_\outp.$ We call a broken disk to be \textit{once boundary marked} if there is exactly one marking on the boundary of $\bD_\outp.$ Fix $p\in L_\outp$, a once boundary marked broken disk with a point constraint $p$ would be the moduli space of once marked broken disks which map the boundary marking to $p$. The virtual dimension $\vir_{const}$ of this moduli space is $\vir - \dim(L)$ where the $-\dim(L)$ is the contribution from constraining the marked point to map to $p$.
\begin{defn}[Rigid broken disk with a point constraint]
    A once marked rigid broken disk $\mathfrak{u}$ with point constraint $p$ is called rigid if $\vir_{const} = 0.$
\end{defn}

\begin{defn}[High-index broken disk]
We call a point constrained broken disk $\mathfrak{u}$ high-index if $\vir_{const}(\mathfrak{u}) > 0$.

\end{defn}

Given combinatorial type $\overline{\Gamma}$ of a broken disk, we define the stabilized type $\Gamma^{st}$ by forgetting all the pieces that are not stable along with the puncture nodes that connected the unstable piece to a stable piece. If removing such a puncture node renders an otherwise stable piece into an unstable piece, we forget that as well in $\Gamma^{st}$.

We will now define the notion of domain dependent perturbation of almost complex structure. Let $\cM_{\overline{\Gamma}}$  be the moduli space of nodal disk with combinatorial type $\overline{\Gamma}$. Let $\mathcal{U}_{\overline{\Gamma}} \xrightarrow{\wt\pi_{\overline{\Gamma}}} \cM_{\overline{\Gamma}}$ be the universal moduli space whose fiber over $S\in \cM_{k}$ is the set of pairs $(S,z)$ where $z$ is a point on $S$.  Fix an almost complex structure $ J_{fix}$ on each piece of a broken manifold obtained by neck stretching.  We denote $\mathcal{J}_{cyl}(\widehat M_{out\pm},\omega,D)$ to be the set of almost complex structures satisfying the following:
\begin{enumerate}
    \item compatible with the symplectic form $\omega$ on $\widehat M_{out\pm}$,
    \item is translation invariant in the cylindrical ends,
    \item satisfies $J(\dd{s}) = R$ in the cylindrical ends,
    \item $D$ is adapted to it.
\end{enumerate} 

We define $\cJ_{cyl}(\bR \times Z)$ similarly as the set of cylindrical almost complex structures which are compatible with $\omega_Y$ on $\ker \lambda$ and satisfies $J(\dd{s})=R$.

\begin{defn}[Domain dependent almost complex structure for broken disks ]\label{def:domdepstruc}
A domain dependent almost complex structure of a disk with $k$ interior markings is a smooth map $\mathcal{J}_{\overline{\Gamma}}: \mathcal{U}_{\overline{\Gamma}} \to \mathcal{J}_{cyl}(\widehat M_i,\omega,D)$ such that near the puncture nodes, the almost complex structure is equal to $J_{fix}$, i.e., there is an open neighborhood $U$ in $\mathcal{U}_{\overline{\Gamma}}$ near the locus of puncture nodes such that  $\mathcal{J}_{\overline{\Gamma}} (w) = J_{fix}$ for $w\in U$ . 
\end{defn}

Given a domain dependent almost complex structure on a stabilized combinatorial type of broken strip $\overline {\Gamma^{st}}$ such that near the puncture nodes $J_{\overline{\Gamma}}$ is equal to $J_{fix}$, we can get a domain dependent almost complex structure on the unbroken  $T$-double-stretched manifold $M^T$ by gluing $\Gamma^{st}$ at the puncture nodes. A $T - $ gluing across a puncture node, where $T>0$ is a parameter measuring the ``length of the neck", is obtained by identifying truncated neighborhoods (deleting $(T,\infty)$ from positive end and $(-\infty,-T)$ from negative end) of the strip/cylinder coordinate near both sides of the puncture node.

For notational simplicity we illustrate the gluing construction precisely in the case where there is only one puncture node, the general construction is done similarly. Fix a parameter $T>0$. For now, assume $\Gamma^{st}$ has only one puncture node. Let $\Gamma^{st}_{gl}$ be the combinatorial type obtained from gluing $\Gamma^{st}$ across a puncture node. Thus, the moduli space of disks $\M_{\Gamma^{st}}$ of type $\Gamma^{st}$ is a lower dimensional stratum in the compactified moduli space $\ol \M_{\Gamma^{st}_{gl}}$. We can choose a chart of $\M_{\Gamma^{st}_{gl}}$ near the strata $\M_{\Gamma^{st}}$ by first choosing strip-like parameterization of the surfaces near the puncture node and then gluing the strips.  Let $[\mathfrak{s}]\in \M_{\Gamma^{st}}$, and $C_{[\mathfrak{s}]}$ be the surface corresponding to $[\mathfrak{s}]$, let $p$ be the puncture node lying between the components $C_-,C_+$ of $C_{[\mathfrak{s}]}$. Let $U_\pm$ be  neighborhoods of $p$ in $C_\pm$ such that we can choose charts  $\phi_\pm$ on them as follows,
\begin{align*}
    \phi_-:U_- \to (-\infty ,0] \times \interval \\
    \phi_+:U_+ \to [0,\infty) \times \interval.
\end{align*}
\noindent We obtain the $T$-glued domain $C^T_{[\mathfrak{s}]}$ by removing $\phi_-((-\infty,-T)\times \interval)$  and $\phi_+((T,\infty)\times \interval) $ from the components $C_\pm$ of $C_{[\mathfrak{s}]}$ and gluing the remnant surfaces by the relation $(s,t) \sim (s+T,t)$.

We now explain how to glue domain dependent almost complex structure. Since $\J_{\overline{\Gamma}}$ is equal to $J_{fix}$ near the locus of the puncture node, we can choose a strip-like neighborhood $U_\pm$ as above, so that $\J_{\overline{\Gamma}}$ restricted to $U_\pm$ is equal $J_{fix}$. Recall that we do not perturb the almost complex structure on the inside piece. We define a $T$-glued domain dependent almost complex structure $J^T;$ on the $T$-glued domain $C^T_{[\mathfrak{s}]}$, by defining $J^T_k(z)$ to be equal to the $T$-stretched almost complex structure in the $J^T$ neck region for all $z$.

\begin{align*}
    J^T_k(z) (m) =
    \begin{cases}
     J^T_{fix}        & \text{ if } z \in \phi_-((-\infty,T])\\
     J_k(z)(m)    & \text { if } z\in C^T_{[\mathfrak{s}]} \sm \phi_-(\infty,T]
     \end{cases}
\end{align*}

\begin{defn}[Coherent families of perturbation data for broken disks]
We call the set $$\mathcal{P} = \{ \mathcal{J}_{\overline{\Gamma}} \mid \overline{\Gamma} \text{ is a combinatorial type of broken disk} \}$$, a coherent family of perturbation data if the $T-$glued domain dependent almost complex structure forms a coherent family of perturbation for $M^T$ as in Definition 3.14 in \cite{cwstab} for any $T>0$.

\end{defn}

We explain the notion of being holomorphic with respect to a domain dependent choice of almost complex structure.  A map $u:S_{\textbf{z}} \to \widehat M_{out}$  from a $k$ marked disk $S_{\textbf{z}} = \wt \pi_k \inv ([\textbf{z}]) $ where $[\textbf{z}] \in \M_k$ is called  holomorphic with respect to the domain dependent almost complex structure $\mathcal{J}$ if it satisfies the domain-dependent version of the Cauchy Riemann equation 
\begin{equation*}
   du(z) + J_k(z,(u(z)))\of du \of j_S = 0 .
\end{equation*}
The expected dimension of moduli space of holomorphic maps with respect to domain-dependent choice almost complex structure is equal to domain independent almost complex structure if we add divisor constraints on the marked point.

\begin{defn}
    We call a perturbation scheme $\cP$ to have enough regularity if any holomorphic broken disk with this choice of domain dependent complex structure is either of high index, or is regular. Thus having enough regularity ensures that all rigid broken disks are regular.
\end{defn}

\begin{defn}[Broken divisors]\label{def:div}
    Given a broken manifold $(M_\inn, M_\out = \sqcup M_{out\pm})$ with broken along a divisor $Y$ and a broken Lagrangian $(L_\inn,L_\out)$,
    a broken divisor is a set of codimension two almost complex submanifolds $\bD = (D_\outm, D_\inn, D_\outp)$ such that $D_i$ is 
    Poincare dual to a large multiple $k[\omega_i]$ of the symplectic class for $k\in \bN$ where $i\in \{ \inn, \outm, \outp  \}$ and $D_\inn \cap Y = D_\out \cap Y$.
    We say a broken divisor is adjusted to a broken Lagrangian $(L_\inn,L_\outp)$ if the broken Lagrangian lies in the complement
    of the divisor and is exact. We call each piece of a broken divisor to be a \textit{Donaldson divisor} in the respective piece of a broken manifold.
\end{defn}
\begin{remark}
In our situation where we perform breaking along disjoint coisotropic submanifolds $\phi(\{a,b\} \times Z)$ for a Lagrangian with BSP handle, the inside piece is the projective bundle $\bP ( E \oplus \bC )$ where $E$ is the associated complex line bundle with the circle bundle $Z$ over $Y$ and the outside piece is the disjoint manifold obtained from the two symplectic cuts at $a$ and $b$. We can take the inside piece of the broken divisor, $D_\inn$, to be the bundle over a divisor $D_Y$ on $Y$ such $L_Y$ is exact in the complement of $D_Y$ and $D_Y$ is Poincar{\'e} dual to a large multiple $k[\omega_Y]$. This ensures any holomorphic disk $u$ which is not contained in a single fiber of $\bP(E\oplus \bC)$ intersects with $D_\inn$. We can then use Theorem \ref{brokdiv} to find divisors in the potentially disjoint outer pieces.
\end{remark}

The moduli space of broken disks with domain $\overline{\Gamma}$ to the broken manifold $\mathbb{M}[l,k]$ with boundary on the broken Lagrangian $\mathbb{L}[k]$  is denoted as $\mathcal{M}_{\overline{\Gamma}}(\mathbb{M}[l,k],\mathbb{L}[k],\bD)$ and the moduli space of broken discs is denoted by $\mathcal{M}_{Br}(\mathbb{M}[l,k],\mathbb{L}[k],\bD)$.

Assume $\Gamma^{st}$ is the stabilized combinatorial type of the domain. Let $C$ denote the nodal surface $\mathfrak{D}$. Let $C_i$ denote the surface in $i\th$ piece of the domain, $\mathbb D_i$, regarded as a smooth manifold with nodes and let 
\begin{equation}\label{domtriv}
    \mathcal{U}_{\Gamma^{st}_i,l} \to \mathcal{M}_{\Gamma^{st}_i,l} \times C_i
\end{equation}
be a collection of trivializations which cover $\cM_{\Gamma^{st}}$, indexed by $l$, of the universal moduli space of Riemann surfaces of type $\Gamma$. These trivializations induce a family of complex structures on the fibers of the universal space \begin{equation}\label{domchoice}
    \M_{\Gamma^{st}_i,l} \to \mathcal{J}(C_i), \quad m\to j(m).
\end{equation}
\noindent Let $\mathcal{E}_i$ denote the bundle of $(0,1)$ forms over the space $\Maps^{k,p}_{\delta}(\Sigma_i,M_i ,L_i$, i.e. the fiber $\mathcal{E}_i$ over a map $u_i$ is 
\begin{equation*}
    \mathcal{E}_{i,u}:= W^{k,p}_\delta ( \overline{\operatorname{Hom}}( T\Sigma_i,u_i^*TM_i)).
\end{equation*}
We can define the following map by setting it to $\delbar_{j(m),J_i}$ on each component 

\begin{align*}
    \delbar_{Br} :\mathcal{M}_{\Gamma^{st},l}\times \Maps^{k,p}_{\delta}(\Sigma_{in},\C^n,T_\gamma) \times \Maps^{k,p}_{\delta}(\Sigma_1,\R \times S^{2n-1},\R \times \Lambda)\\ \times \dots \Maps^{k,p}_{\delta}(\Sigma_{out},\widehat M_{out+},L_{out+},K)\to &\mathcal{E}_{out-}\times \dots \mathcal{E}_{out+}.
\end{align*}
\noindent The zero set $\delbar_{Br}\inv (0)$ is the space of holomorphic maps from each piece of the broken disk to the corresponding level of a broken manifold \textit{without} the matching of Reeb asymptotes. The moduli space of broken disks  is obtained as the inverse image of the diagonal of space of Reeb chord matchings,
$$EV: \delbar_{Br}\inv (0)  \to \cR(\Lambda_2)^{|\cE_{\overline{\Gamma}}|} \times M_\outm^{\cI_\outm} \times \dots M_\outp^{\cI_\outp} $$
$$\mathcal{M}_{\overline{\Gamma}}(\mathbb{M}[l,k],\mathbb{L}[k],\bD) = ev\inv(\Delta \times D_\outm ^{\cI_\outm}  \times \dots D_\outp^{\cI_\outp} )$$
where $\Delta$ is the diagonal submanifold of the product manifold $\cR(\Lambda_2)^{|\cE_{\overline{\Gamma}}|}$, $\cI_i$ is the index set of interior punctures in $\Gamma_i$ and $D_i$ is the $i\th$ piece of a broken divisor.

We say a broken disk $(\mathfrak{u},\mathfrak{D})$ is \textit{regular} if the linearization of $\delbar_{Br}$ is surjective at $\mathfrak{u}$ and the evaluation map $EV$ is transverse to $\Delta \times D_\outm ^{\cI_\outm}  \times \dots D_\outp^{\cI_\outp}$ at $EV(\mathfrak{u}).$

\subsection{Gluing broken things}

We recall the notation relevant to gluing from \cite{PW19}.  Given a nodal domain $\mathfrak{S}$ with $l,k$ levels, and gluing parameters $\delta_{-l},\dots \delta_{-1},\delta_1,\dots \delta_k > 0$, we define the glued domain $S^{\delta_1,\dots,\delta_k}$ by gluing strip-like $(-L,L) \times \R \times \interval$ or cylindrical necks $(-L,L) \times S^1$ of length $2L =|\ln (\delta_i)|$ at the i$^{th}$ puncture node. Given a $l,k $ broken manifold $\mathbb M$, we can similarly define $M_{\delta_{-l},\dots,\delta_k}$ by gluing necks of appropriate length. Note that the notational convention we choose is to use subscript to denote parameters of a glued broken manifold while superscript is used to denote a stretched manifold. We use \cite{PW19} for gluing results and  bijection of moduli of rigid broken disks and rigid disks from the gluing construction. We skip the proofs of these gluing theorems and refer readers to \cite[\S 6.5]{PW19}

\begin{Theorem}[Theorem 6.8 in \cite{PW19}]
Let $(\mathfrak{u,S})$ be a regular broken disk with boundary on broken Lagrangian $\bL$  limiting eigenvalues $\lambda_1,\dots\lambda_k$ of Reeb chords or orbits. Then, there exists  $\delta_0$ such that for each $\delta \in (0,\delta_0)$ there exists an unbroken strip $u_\delta: S^{\delta/\lambda_1,\dots,\delta/\lambda_k} \to M_\delta$ such that $u_\delta \to \mathfrak{u}$ as $\delta \to 0$.
\end{Theorem}

\noindent We restate Theorem \ref{gluingglimpse} here. This theorem asserts that if we ensure that every broken disk is either rigid-regular or of high index, then the gluing map gives a bijection between rigid broken disks and rigid disks in the glued unbroken manifold.

\begin{Theorem}[Theorem 4.25 \cite{Cha23}, Theorem 6.25 in \cite{PW19}]\label{glu}
    Let $\mathbb M$ be a broken symplectic manifold with no neck levels and $\mathbb L$  be a broken Lagrangian with no neck levels.  Assume perturbations have been chosen so that every strip in $\mathcal{M}^{<E}_{Br,\mathcal{P}}(\mathbb{M},\mathbb{L},D)_0$ is regular. There exists $\delta_0$ such that for $ \delta <\delta_0$ the correspondence $\mathfrak{u} \mapsto u_\delta$ from gluing defines a bijection between the moduli spaces $\M^{<E}_{Br,\mathcal{P}}(\mathbb{M},\mathbb{L})_0$ and $\mathcal M ^{E_{Hor}< E}(M_{\delta},L)_0$.

$$\M^{<E}_{Br,\mathcal{P}}(\mathbb{M},\mathbb{L})_0 \leftrightarrow \mathcal M ^{E_{Hor}< E}(M_{\delta},L)_0$$
\end{Theorem}

\subsection{Regularity of broken buildings}\label{sub:reg}

We use broken divisors as in \cite{floflip} to regularize our broken disks by allowing domain dependent perturbations of the almost complex structure. The idea of using domain dependent perturbations by considering domains with marked point has been widely utilized in the literature to cut moduli spaces out transversely, see  \cite{cm07} for its development in the context of Gromov-Witten theory and \cite{cwstab} for the Floer theoretic setting.

\begin{defn}[Rational Broken manifold]
We call a broken manifold $(M_\inn,  M_\out = \sqcup M_{out\pm})$ rational if the induced symplectic form on each piece is a rational form.
Moreover, we call a broken Lagrangian to be rational if the symplectic form is rational in the relative cohomology, in other words, if the
symplectic area of any two cycle with boundary on each piece of the broken Lagrangian is a rational number.
\end{defn}

\begin{example}
    The complex projective space $\bC P^n$ with an appropriately scaled symplectic form has and embedding of the standard symplectic ball $(B^{2n}(0),\omega_0)$.
    We break the complex projective space $\bC P^n$ along the disjoint union of two spheres $S_a^{2n-1} \sqcup S_b^{2n-1}$ of rational radii to obtain 
    a rational broken manifold where $M_\inn = \bP(\cO(-1)\oplus \bC)$, $M_\outm = \bC P^n $ and $ M_\outp = \bP(\cO(-1)\oplus \bC)$.

\end{example}

\begin{defn}[Lagrangians broken along a Bohr-Sommerfeld-profile]
    We call a broken Lagrangian $(\overline{L}_\inn, \overline{L}_\out)$ if it is obtained as a breaking of an unbroken Lagrangian $L$ along a coisotropic submanifold $\phi( \{b\} \times Z)$ where $\phi$ is the symplectomorphism in the local model of a Lagrangian with BSP handle as in Definition \ref{bsphandle}.
\end{defn}

\begin{Theorem}\label{brokdiv}
If $(\overline{M}_\inn,\overline{M}_\out)$ is a rational broken manifold equipped with line bundles $\Tilde{M}_\inn,\Tilde{M}_\out$ with connections whose curvature are $\omega_\inn, \omega_\out$, such that the restriction of the bundles on $Y$ agree, then there is a broken divisor. Moreover if $(\overline{L}_\inn,\overline{L}_\out)$ is a rational broken 
Lagrangian, broken along a Bohr-Sommerfeld profile and both $\overline{M}_\inn,\overline{M}_\out$ are simply connected, we can assume that the broken divisor is adjusted to the broken Lagrangian.
\end{Theorem}

\begin{proof}
This result is a modified version of arguments already well known in the literature about existence of divisors in complement of cleanly intersecting Lagrangians. This construction was first developed by Donaldson in \cite{donaldson:symplsub} and later extended in \cite{aurmohgay:symplchyp, floflip} and for the argument with clean intersection case in \cite{PT20}.
First, note that from Lemma 4.3 \cite{bcsw1} that after up to taking a large tensor power of the bundles $\Tilde{M}_i$, there are flat sections over $\overline{L}_\inn, \overline{L}_\out $.

The first part of the result follows from  Theorem 7.15 in \cite{floflip} which uses Gaussian functions in the normal direction to extend asymptotically holomorphic sections to a neighborhood of $Y$, call these as $\pi^* \sigma_Y$. We can construct a sequence of sections $s_k$ on $Y$ which are concentrated on $L_Y$, asymptotically holomorphic and uniformly transverse to 0. Note that since the Lagrangian is broken along a Bohr-Sommerfeld profile, we can choose the normal bundle to $Y$ locally so that the Lagrangians are $\cup_{l=0}^{k-1} \bR e^{\frac{il\pi}{k}}$ in the normal direction. The same proof as in Lemma 4.5 in \cite{bcsw1} can be used to prove a lower bound on $\pi^*\sigma_Y$ in a neighborhood of $Y$. Now using the argument in Proposition 4.4 of \cite{bcsw1} using $\pi^*\sigma_Y$ instead of $\sigma_\cap$, it follows that one can construct a sequence of asymptotically holomorphic sections uniformly transverse to zero which is also concentrated on $L_\inn$ and $L_\out$. Moreover, we can use Lemma 7.17 in \cite{floflip} to construct sections such that their restriction to $Y$ matches with $s_k$. The existence of the claimed broken divisor then follow from Donaldson's construction.
\end{proof}

\begin{Theorem}\label{thm:regularity}
    There exists a perturbation scheme $\cP$ such that any building is either regular or is of high index.
\end{Theorem}
\begin{proof}
We sketch the change one needs to make to the already existing literature(\cite{floflip,PW19}) for construction of coherent perturbation schemes to attain enough regularity of moduli space of broken disks. Note that any punctured $J_{cyl}$ holomorphic disk map in the inside region and neck region which lies in a vertical fiber over $Y$ is regular due to a version of automatic regularity result, Lemma 13.2 of \cite{seidelFuk}. Let $u:\Sigma \to \bR \times Z$ be a holomorphic punctured disk to a fiber, then one sees that for any $X\in Vect(\Sigma)$ which extend to vector-fields over $\bD^2$ that vanish at the punctures, $du(X)$ extends to a section in the compactified vertical bundle. Thus from Lemma 4.3.2 \cite{JholoWendl} we can extend the domain of $D_u^v$ to $W^{k,p}(\overline{\Hom}(T\Sigma)) \oplus W^{k,p,\delta}(u^*T^v(\bR \times Z),\partial u^*(T^v \cL_\gamma)).$ Then we can use Lemma 13.2 of \cite{seidelFuk} to achieve regularity of the fiber maps. Regularity of maps to a fiber also follows from the regularity results of \cite[\S 5.1]{Cha23} since the fibers for a fixed choice of $\gamma$ are same irrespective of the Bohr-Sommerfeld profile.

For maps which are not contained in a single fiber, we know that the projection is non-constant and thus must intersect the divisor $D_\inn$. Following \cite{cwstab,floflip} we can choose a coherent perturbation scheme with enough regularity for maps which intersect the broken divisor.
\end{proof}

We will now prove a partial regularity statement about broken maps where the perturbation scheme is fixed in the inner piece $M_\inn.$ To be precise, we call a perturbation scheme to be \textit{fixed the interior}  if $\cP$ is fixed to be the cylindrical complex structure induced from $J_Y$ in $M_\inn.$ We will show that a certain class of broken disks can be made regular while fixing the interior almost complex structure. This allows us to exclude multiple punctured disks in the interior for a close enough perturbation scheme with enough regularity.

\begin{lemma}\label{lemm:fixedininterior}
    There is a perturbation scheme $\cP_0$ fixed in the interior such that any broken holomorphic map where the interior disks have single boundary puncture and are sections over regular disks in $(Y,L_Y)$ is regular as a broken disk.
\end{lemma}

\begin{proof}
    From Lemma \ref{regularityfromregularity} we know that sections over regular disks are regular. Thus, using an inductive argument similar to \cite{cwstab,floflip} we can use the domain dependent perturbations on the neck and outside pieces to find such a $\cP_0$.
\end{proof}

\begin{remark}
    Since we can ensure to find a $J_Y$ which makes Maslov $2$ disks regular, we can ensure that the perturbation scheme $\cP$ is fixed in the interior where the interior disks have minimal possible markings. Since $L_Y$ is monotone, minimal possible marking ensures that  the Maslov number is minimal, thus equal to $2$.
\end{remark}

\begin{prop}\label{thm:goodrc}
    There exists a regular enough perturbation scheme $\cP$ such that rigid broken disks will have Reeb chords over $p\in L_Y$ such that $p$ is a regular value for the evaluation map of Maslov 2 disks.  
\end{prop}
\begin{proof}
    From Theorem \ref{classifthm} we know that any rigid broken disk in a perturbation scheme with enough regularity has only disks with single puncture in the inside piece where the single puncture is asymptotic to a minimal Reeb chord. We first deal with the case of rigid broken disk with interior disk lying over a Maslov 2 disk.
    
    The moduli space of \textit{unparametrized} interior disk with a single boundary puncture asymptotic to a minimal Reeb chord is of dimension $n-1$ from Lemma \ref{lemm:singlepuncindex} which is same as the dimension of the space of Reeb chords, $\cR(\Lambda_2)$. We denote this unparameterized moduli space as $\cM^{up}_{\Gamma_\inn}(M_\inn)$ Thus, choosing a perturbation scheme to cut out the matching condition at each Reeb chord between the inside and outside piece of rigid broken disk ensures that the Reeb chords which occur in rigid broken disks are regular values under the evaluation map from interior pieces.

     The natural map (induced from $\pi_Y : \Lambda_2 \to L_Y$)  from the space of Reeb chords to $L_Y$ $$\cR(\Lambda_2) \to L_Y$$ is a covering map. We can construct the composition from moduli space of once punctured disks  $$\pi_Y\of ev_{\Gamma_\inn}: \cM^{up}_{\Gamma_\inn} (M_\inn) \to \cR(\Lambda_2) \to L_Y $$
    where $\Gamma_\inn$ is the combinatorial type corresponding to a once boundary punctured disk.
    Since $\cR(\Lambda_2) \to L_Y$ is a smooth covering map, we have that the points $p\in L_Y$ which appear as projection of Reeb chords in rigid broken building are regular values of $\pi_Y \of ev_{\Gamma_\inn}$.

     In the commutative diagram where we abuse $\pi_Y$ to denote the induced map from the projection,
\[
\begin{tikzcd}
\cM^{up}_{\Gamma_\inn}(M_\inn,\cL_\gamma) \arrow{r}{ev} \arrow[swap]{d}{\pi_Y} & \cR(\Lambda_2) \arrow{d}{\pi_Y} \\
\cM^{up}_{\mu=2,\Gamma_\inn}(Y,L_Y)  \arrow{r}{ev} & L_Y
\end{tikzcd}
\]
  the vertical maps are covering maps. Thus if $p\in L_Y$ is a regular value of $\pi_Y \of ev_{\Gamma_\inn}$, then it is a regular value of evaluation on moduli space of once boundary marked Maslov 2 disk. If not, then there would be a Maslov 2 disk $v$ for which the evaluation would be not submersion at that point, by taking a lift of $v$ using Theorem \ref{modelthm} we can construct a building such that the evaluation map is not a submersion at the puncture.

We now finish the proof by dealing with broken disks whose interior part is contained in a fiber over $Y$. The evaluation map at the interior puncture is clearly a surjective submersion to the space of Reeb chords.

There are finitely many rigid regular broken disks from  SFT compactness, (gluing) Theorem \ref{gluingglimpse} and area bound on rigid broken disks from monotonicity of the Lagrangian $L$ in $M$. Thus there are finitely many Reeb chords appearing as asymptotic ends of rigid regular broken disks. From a dimension count we see that of moduli space of outside disks with point constraint which appear in rigid broken disks is of dimension 0. We know that the evaluation map from the universal moduli space is a submersion. Thus, we can perturb $\cP_{out}$ to ensure the evaluation from the outside component lies outside the nowhere dense set of critical values.

\end{proof}

\begin{remark} 
    If we take a Biran circle bundle lift $L$ from a monotone symplectic divisor $X$ in a monotone symplectic manifold $M$,
    we can perform an exact isotopy of the $L$ bundle to ensure that the lifted Lagrangian 
    intersects the circle bundle over the divisor transversely such that we can perform a breaking along
    the circle bundle over 
\end{remark}

\subsection{Disk-Potential from SFT buildings}

Equipped with the bijection from Theorem \ref{gluingglimpse}, we can define a version of disk potential by counting broken disks. The definition begins with fixing a point $p\in L_\outp$ and counting rigid broken
disks which pass through $p$.
\begin{defn}
    Let $\cP$ be a perturbation scheme with enough regularity for broken disks with one outside constraint at $p$. We define the 
    broken disk potential as 
\begin{equation*}
        W^{Br}_{\bL} (x_1,\dots) = \sum_{[\mathfrak{u}] \in \cM_{0}{(\bM,\bL,\bD})} (sgn) x^{\partial u_g}
    \end{equation*}
where $u_g$ is obtained by a gluing of the broken disk $\mathfrak{u}$. Note that the [$\partial u_g$] does not depend on the gluing parameter because of exponential convergence to trivial strips near each puncture.
\end{defn}


\begin{Theorem}\label{thm:brokissame}
Broken disk-potential is equal to the usual disk potential.
\end{Theorem}

\begin{proof}
    This proof is similar to the proof of Lemma 7.4 in \cite{Cha23}. From monotonicity of the Lagrangian we know that there is an upper bound on the energy of broken disks which appear in the broken potential. Thus, by Theorem \ref{glu} we have that for a $\delta-$gluing with a small $\delta$, there is a bijection between rigid regular disks. Moreover, the domain-dependent perturbation data $\cP_\delta$ obtained from gluing has enough regularity in the sense that all disks of virtual dimension zero are regular, and all other disks have higher virtual dimension. Let $W^{\cP_\delta}_L$ denote the disk-potential by counting rigid $\cP_\delta$-holomorphic disks with point constraint. The gluing bijection gives us $$W^{Br}_{\bL} =  W^{\cP_\delta}_L$$
    
    Recall that the disk-potential is obtained from the rigid Maslov 2 disks with point constraints. There is an oriented cobordism  between moduli space of rigid regular disks  with the perturbation scheme $\cP_\delta$ and the moduli space of rigid regular disks with some domain independent choice of $J$. The oriented cobordism can be thought of as a collection of disjoint copies of circles $S^1$ and intervals $[0,1]$. If the boundary of an interval piece lies on the same side of the cobordism, note that the contribution to the disk potential from these boundary components cancel out due to orientation. After such cancellations, we can take the boundary of the remain intervals to obtain $$W^{\cP_\delta}_{L^\delta} + -(W^{J}_{L^\delta}) = 0,$$ where $W^J_{L^\delta}$ refers to the disk-potential obtained from counting rigid $J-$holomorphic disks with point constraint. Combining the two equations above we have the required result.
\end{proof}

\section{Proof of Theorem \ref{mainthm}}

The neck-stretching technique in Symplectic Field Theory as introduced in \cite{egh00,BEH03} has proved to be a robust tool in Floer theory. See \cite{mw18,mw19,PW19,Via16,Cha23} for applications of neck-stretching to computing Floer theoretic invariants. In this section, we use counts of broken holomorphic disks (also known as holomorphic buildings of disks) to compare the disk-potential of a Lagrangian with its zero-area BSP surgery to prove Theorem \ref{mainthm}.  We postpone the technical discussion on SFT neck-stretching to Section \ref{appendix}. One can read Section \ref{sftbackground} before proceeding to the rest of the section to get familiarized with the necessary ideas revolving neck-stretching.

The main slogan for the heuristic of using neck-stretching in our context is as follows:
\vspace{5pt}
\begin{fslogan}
    Counts of rigid holomorphic buildings  can be used to define disk-potential. 
\end{fslogan}
The slogan is based on the SFT compactness theorem and the following gluing theorem which asserts that after assuming sufficient regularity, counts of rigid broken disks match with rigid unbroken disk of a neck-stretched manifold.

\begin{Theorem}[Theorem 4.25 \cite{Cha23}, Theorem 6.25 in \cite{PW19}]\label{gluingglimpse}
    Let $\mathbb M$ be a broken symplectic manifold with no neck levels and $\mathbb L$  be a broken Lagrangian with no neck levels.  Assume perturbations have been chosen so that every strip in $\mathcal{M}^{<E}_{Br,\mathcal{P}}(\mathbb{M},\mathbb{L},D)_0$ is regular. There exists $\delta_0$ such that for $ \delta <\delta_0$ the correspondence $\mathfrak{u} \mapsto u_\delta$ from gluing defines a bijection between the moduli spaces $\M^{<E}_{Br,\mathcal{P}}(\mathbb{M},\mathbb{L})_0$ and $\mathcal M ^{E_{Hor}< E}(M_{\delta},L)_0$.

$$\M^{<E}_{Br,\mathcal{P}}(\mathbb{M}[l],\mathbb{L}[l])_0 \leftrightarrow \mathcal M ^{E_{Hor}< E}(M_{\delta},L)_0$$
\end{Theorem}

Let $L$ be a compact monotone oriented spin Lagrangian in $M$  with a BSP Handle $\phi(\cL_\gamma)$. Recall from Definition \ref{bsphandle} that $\phi : [a,b]\times Z$ is a symplectic embedding. A zero-area BSP surgery modifies the Lagrangian $L$ only inside the neighborhood $\phi([a,b] \times Z)$. 

By using the neck-stretching heuristic, one only needs to understand the change of disk-count in the ``interior"  of the neck-stretching process. We postpone the detailed discussion of moduli space of holomorphic buildings to Section \ref{appendix}.

\subsection{Classification of Rigid Broken Disks}

In this subsection, we will study the moduli space of  rigid broken holomorphic disks with boundary on Lagrangians with Bohr-Sommerfeld Profile.

\begin{defn}[Positive and Negative minimal Reeb Chord]
    We call a Reeb chord $c: [0,1] \to Z$, between $\Lambda$ and $e^{i\frac{\pi}{k}}\Lambda$, a positive minimal Reeb chord, if it is of length $\pi/k$ , $c(0) \in  \Lambda \text{ and } c(1) \in e^{i\frac{\pi}{k}}. $ If the Reeb chord is of length $\pi/k$ and $c(0)\in e^{i\frac{\pi}{k}}\Lambda, c(1) \in \Lambda,$  we call $c$ to be a negative minimal Reeb chord.
\end{defn}

\begin{remark}
    Holomorphic maps from disks with single boundary puncture to $\bR \times Z$ with boundary of $\cL_\gamma$ play an important role throughout the article. Counts of such disks contribute to the disk-potential. As a convention, we use the identification of the upper half plane with the disk with a single boundary puncture, $\bH \cong \bD^2 \sm \{1\} $.
\end{remark}

For the rest of the section, we will use a judicious  choice of perturbation scheme for moduli space of holomorphic buildings, as done in \S \ref{sub:reg}.We now state the main result of the subsection, which classifies rigid holomorphic broken disks. 

\begin{Theorem}\label{classifthm}
    Assuming a perturbation scheme with enough regularity has been chosen for broken maps, if $\mathfrak{u}$ is a rigid broken disk with boundary on $L$ or $BSP_0(L)$, then \begin{itemize}
        \item $u_{in}$ is a single punctured disk,
        \item $\pi_Y \of u_{in}$ is either a constant map or a Maslov 2 map,
    \end{itemize}
Moreover, if the boundary of $u_{in}$ is $L$ (or $BSP_0(L)$), we have the following result about the asymptotic Reeb chord and the single boundary puncture,
\begin{itemize}
    \item $u_{in}$ is asymptotic to a positive (negative) minimal Reeb chord iff $\pi_Y\circ u_{in}$ is a constant map,
    \item  $u_{in}$ is asymptotic to a negative (positive) minimal Reeb chord iff $\pi_Y \of u_{in}$ is a Maslov 2 disk.
\end{itemize}
\end{Theorem}

Assuming Theorem \ref{classifthm} and the development of neck-stretching theory in Section \ref{appendix}, we provide a proof for the change of disk potential formula in Theorem \ref{mainthm}.

\begin{proof}[Proof of Theorem \ref{mainthm}]
    By assumption, we know that $L$ has a Bohr-Sommerfeld-Profile handle modelled on an admissible curve. Let $\phi: [a,b] \times Z \to M$ be the associated embedding such that $\phi\inv(L) = \cL_\gamma$ for an admissible path $\gamma.$ We perform a double neck-stretching along the contact-type hypersurface with two connected components, $\phi (\{ a,b\} \times Z)$. From the regularity results in Section \ref{sub:reg},  Theorem \ref{thm:regularity} and Proposition \ref{thm:goodrc},  we can assume that there is a perturbation scheme for broken holomorphic disks with boundary on $L$ such that any broken disk is either rigid and regular or is of too high index. Moreover, we can assume that the almost complex structure in $\bR \times Z$ is the cylindrical complex structure induced from $J_Y$ where $J_Y$ is a complex structure on $Y$ for which the space of Maslov 2 disks is regularly cut out. Now from Theorem \ref{thm:brokissame} we see that it is enough to compare the broken disk potential of $L$ with $BSP_0(L)$. Finally, Theorem \ref{classifthm} shows that there is a correspondence between rigid broken disks with boundary $L$ and $BSP_0(L)$. Let $( u_\inn,u_\outp)$ be a rigid regular disk with boundary on $L$,
    \begin{itemize}
        \item if $u_\inn$ is asymptotic to a positive minimal Reeb chord $c$, then we can construct $|ev\inv ( \pi_Y (c))|$ many rigid regular broken disks $\mathfrak{v_i}= (v^{(i)}_\inn, u_\outp)$ with boundary on $BSP_0(L)$ where $v^{(i)}_\inn$ are lifts of Maslov 2 $J_Y$ holomorphic disks in $Y$ passing through $\pi_Y (c),$
        \item if $u_\inn$ is asymptotic to a negative minimal Reeb chord $c$, then
        we can form a family of $|ev\inv(\pi_Y(c))|$ many rigid regular disks $\mathfrak{u}_i$ with boundary on $L$ by taking lifts of Maslov 2 disks such that $\mathfrak{u}_1$ is $(u_\inn, u_\outp)$ and this family corresponds to a single rigid broken disk $(v_\inn, u_\outp)$ with boundary on $BSP_0(L).$
    \end{itemize}
We can perform the above replacement argument because of the choice of perturbation scheme as in Theorem \ref{thm:goodrc}. Now by comparing the change in boundary homology cycle under the above correspondence, we see that 
$$  W_{BSP_0(L)}(x_1,\dots,x_k,z,w_1,\dots,w_l)
 =W_L(x_1,\dots,x_k,zW_{\Lambda}(x_1,\dots,x_k),w_1,\dots,w_k).  $$
    
\end{proof}

\subsection{Model Bundle Pairs}

To prove Theorem \ref{classifthm} we start off with studying holomorphic disks in the `inside' piece of holomorphic buildings. From now on, until the end of this section, $\Sigma$ will denote a disk with punctures. Notice that any $J_{cyl}$ holomorphic map
$$u:\Sigma \to (\bR \times Z, \cL_\gamma),$$   with finite Hofer energy, descends to a holomorphic disk $\overline{ u} = \pi_Y \of u$ with boundary $L_Y.$

On the other hand, given any $J-$holomorphic map $\overline{u}: \bD^2 \to (Y,L_Y)$, we can construct the following pull-back bundle pair $$(\overline{u}^* (\bR \times Z), \partial{\overline{u}}^* \cL_\gamma) \to (\bD^2,S^1).$$ Here we take the pullback as a $\bC^*$ bundle along with the connection one-form $\lambda$ and $\partial{\overline{u}}^* \cL_\gamma$ is defined as the following fiber-product, $$\partial{\overline{u}}^* \cL_\gamma=\left\{ (p,z) \in \partial\bD^2 \times \bR \times Z | z\in \cL_\gamma|_{\overline{u}(p)}  \right\}.$$ One can check that $\partial{\overline{u}}^* \cL_\gamma$ can be identified (up to scaling) as the parallel-transport of $\gamma^{1/k}$ in a fiber over a point on $\partial\bD^2$. Recall that $\bR \times Z$ has an induced complex linear connection from the connection $\lambda$ on $Z$. Indeed, the parallel transport maps in $\bR \times Z$ are just multiplication by $e^{i\theta}$.  The pullback bundle $\overline{u}^* (\bR \times Z)$ has an induced complex linear connection from $\bR \times Z$. We use this complex linear connection to view the pullback bundle $\overline{u}^*(\bR \times Z)$ as a holomorphic bundle.

Any holomorphic map $u:\Sigma \to (\bR \times Z, \cL_\gamma)$ can be realized as a section of the holomorphic bundle
\begin{equation*}
(\overline{u}^* (\bR \times Z), \partial{\overline{u}}^* \cL_\gamma) \to (\bD^2,S^1).
\end{equation*}
From the Oka-Grauert principle, we know that any holomorphic line bundle over $\bD^2$ is a trivial bundle. Thus, we have a trivialization \begin{equation}\label{initialiso}  
(\overline{u}^* (\bR \times Z), \partial{\overline{u}}^* \cL_\gamma) \cong (\bD^2 \times \bC^*, f(\theta) \gamma^{1/k} )
\end{equation}where $f:[0,2\pi] \to \bC^*$ such that, $$f(0)=0, f(2\pi) = e^{i2\int_{\overline{u}}\omega_Y}.$$ From Assumption \ref{as1}, we know $\int_{\overline{u}}\omega_Y  = \frac{\pi l}{k}$ for some $l \in \bZ$. Here $f$ is given by parallel transport along the boundary.

Since $L_Y$ is a monotone Lagrangian, we have that $$\int_{\overline{u}}\omega_Y = \alpha \mu(\overline{u}),$$ for a constant $\alpha.$ Also, since $L_Y$ has minimal Maslov number 2, we can suppose that $\alpha = \frac{m\pi}{2k}$ where area of a Maslov 2 disk is $\frac{m \pi}{k}.$ We can also define $\frac{m \pi}{k}$ to be the smallest possible positive $\omega_Y$ area of a disk with boundary on $L_Y$. Thus, we know that $$f(2\pi)=e^{i\frac{m\pi}{k} \mu(\overline{ u})}.$$ With a slight abuse of notation, we can  define a map $f^k : S^1 \to \bC^*$ as $f^k([\theta]) = f(\theta)^k$ where we identify $S^1 \cong [0,2\pi]/ \{  0 \sim 2\pi \}$. This is well-defined as $L_Y$ is orientable thus Maslov number of disks are even. The following theorem proves that the bundle pairs $(\overline{u}^* (\bR \times Z), \partial{\overline{u}}^* \cL_\gamma)$ depend only on the Maslov number of $\overline{u}$ and the curve $\gamma$.
\begin{Theorem}[Model Bundle Pairs]\label{modelthm}
    Let $\overline{u} : \bD^2 \to (Y,L_Y)$ be a $J$ holomorphic map. We have an isomorphism of holomorphic line  bundles which induces an isomorphism on the following pairs,
    $$(\overline{u}^* (\bR \times Z), \partial{\overline{u}}^* \cL_\gamma) \cong (\bD^2 \times \bC^*, e^{i\frac{m\theta}{2k}\mu(\overline{u})} \gamma^{1/k} ),$$ where $\mu(\overline {u})$ is the Maslov number and $\frac{m\pi}{2k}$ is the monotonicity constant.
\end{Theorem}

\begin{proof}
    From Equation (\ref{initialiso}) we have 
    \begin{equation*}
(\overline{u}^* (\bR \times Z), \partial{\overline{u}}^* \cL_\gamma) \cong (\bD^2 \times \bC^*, f(\theta) \gamma^{1/k} ).
\end{equation*} So it is enough to prove the following isomorphism,
\begin{equation}\label{secondiso}
(\bD^2 \times \bC^*, f(\theta) \gamma^{1/k} ) \cong (\bD^2 \times \bC^*, e^{i\frac{m\theta}{2k}\mu(\overline{u})}\gamma^{1/k} ).
\end{equation}
Notice that it is enough to prove isomorphism between the $k^{th}$ tensor power of each bundle pair, i.e.,
\begin{equation}\label{thirdiso}
(\bD^2 \times \bC^*, f^k(\theta) \gamma ) \cong (\bD^2 \times \bC^*, e^{i\frac{m\theta}{2}\mu(\overline{u})}\gamma ).
\end{equation}
Indeed if the isomorphism in Equation \ref{thirdiso} is given by a holomorphic map $h : \bD^2 \to \bC^*,$ then we have a $k^{th}$ root, $h^{1/k}$, of $h$ since $z \mapsto z^k$ is a $k:1$ covering of $\bC^*$ and $\bD^2$ is simply connected. Thus $h^{1/k} : \bD^2 \to \bC^*$ can serve as an isomorphism in Equation \ref{secondiso}.

Now to obtain the isomorphism in Equation \ref{thirdiso}, we first show that by changing trivialization, we can assume $|f^k| \equiv 1$, i.e., $f^k : S^1 \to S^1 \hookrightarrow \bC^*.$ Denote the function $|f^k (\theta)|$ as $ r(\theta). $ Thus we have $$
\log (r) : S^1 \to \bR^+,$$ a smooth function. We can extend it harmonically to $\bD^2.$ By using the harmonic conjugate of the extension as an imaginary coordinate we thus have a holomorphic function $\wt h : \bD^2 \to \bC$ such that $Re(\wt h)|_{\partial \bD^2} = \log (r).$ By setting $h= e^{\wt h} : \bD^2 \to \bC^*,$ we see that $|h| = r = |f|^k$ as functions on $S^1$. Thus, by changing the trivialization using $h$, we can assume $|f^k| = 1.$

Now we can directly apply Oh's Theorem 1 in \cite{riemHilb} to find a change of trivialization $g:\bD^2 \to \bC^*$ such that $g(\theta).f^k(\theta)= e^{i\frac{m\theta}{2}\mu(\overline{u})}$, which will serve as an isomorphism in Equation \ref{thirdiso}.
\end{proof}

The upshot of Theorem \ref{modelthm} is that finding holomorphic maps from punctured disk to $\bR \times Z$ with boundary on $\cL_\gamma$ becomes equivalent to solving a ``model" problem which depends on finding sections of bundle-pairs of the form $(\bD^2 \times \bC^*, e^{i\frac{m\theta}{2k}\mu(\overline{u})} \gamma^{1/k} )$ which depends solely on $\gamma$ and the Maslov number. See \S 2 \cite{bcsw2} for a similar approach to classify holomorphic disks. By using Theorem \ref{modelthm}, we can generalize many classification-of-holomorphic-disk results from \cite{Cha23}. We collect some of them here.

\begin{lemma}\label{lemm:singlepuncindex}
    If $u:\bH \to (\bR \times Z, \cL_\gamma) $ is a single-punctured disk with the puncture being asymptotic to a Reeb chord of length $l \frac{\pi}{k}$, i.e., $l$ multiple of the smallest possible Reeb chord, then the index of the linearized Fredholm operator $D_u$ is $n+l$, $$\Ind(D_u) = n + l.$$
    Here $D_u$ is the linearized CR operator at $u$ with weighted Sobolev domain $W^{k,p,\delta}$ for $0 < \delta < \pi/k$, see Section \ref{sec:Freddy}.
\end{lemma}

\begin{proof}
    Consider the short exact sequence 
    $$ 0 \to T^v (\bR \times Z) \to T(\bR \times Z) \xrightarrow{D\pi_Y} \pi_Y^* TY \to 0, $$ where $T^v (\bR \times Z)$ is the vertical bundle, $\ker D\pi_Y$. Let $u_Y = \pi_Y \of u.$ Notice that the vertical bundle of pullback bundle is isomorphic to the pullback of vertical bundle. Thus, we have, $$u^* T^v(\bR \times Z) \cong T^v (u_Y^* (\bR \times Z)).$$
    Using the s-e-s above, we have that $$\Ind (D_u) = \Ind (D^v_u) + \Ind (D_{u_Y}),$$ where $D^v_u$ is the restriction of $D_u$ to the vertical component and $D_{u_Y}$ is the linearization at $u_Y$. 
    
    Let $w$ be a single punctured disk in $\bR \times S^{2k-1}$ with the puncture asymptotic to a Reeb chord of length $l\frac{\pi}{k}$ such that its projection to $\bP^{k-1}$ is of Maslov index $\mu(u_Y)$. Now, using Theorem \ref{modelthm} we see that the vertical indices for $v$ and $w$ are the same. The vertical index for $w$ was computed in Remark 5.4 of \cite{Cha23}. Thus, we see that 
    \begin{align*}
        \Ind (D^v_u)  &=  \Ind (D^v_w) \\ &= l - \mu(\pi \of w)  \\&= l - \mu(u_Y) \\
        \implies \Ind (D_u) &= l - \mu(u_Y) + \Ind (D_{u_Y})\\
        &= n + l
    \end{align*}
\end{proof}

\begin{lemma}[Regularity of sections over regular maps]\label{regularityfromregularity}
    If  $\overline{u} :\bD^2 \to (Y,L_Y)$ is a regular $J_Y$ holomorphic disk and $u: \Sigma \to (\bR \times Z, \cL_\gamma)$ is a section over a punctured disk $\bD^2 \sm \{1 \} \subset \bD^2$, then $u$ is regular.
\end{lemma}
\begin{proof}
    Note that $u^*T(\bR \times Z)$ has a splitting $$u^*T(\bR \times Z) \cong \overline{u}^* ((TY) \oplus  \overline{\bC}) $$ in a small punctured neighborhood around $\{1\}$ in  $\bD^2\sm \{1\}$ where the trivial bundle $\overline{\bC}$ is obtained by a trivialization of the fiber bundle $Z$ over $Y$ near each puncture.
    Proposition 6.16 of \cite{PW19} shows that this Cauchy-Riemann(CR) problem can be compactified by adding trivial bundles near each puncture such that the kernel and cokernel of the compactified CR problem is isomorphic to the cylindrical CR problem. 
    The short exact sequence 
    $$0 \to T^v{(\bR \times Z)} \to T(\bR \times Z) \to \pi_Y^* TY \to 0$$ induces  $D_u$ to splits as $D_u^v \oplus D_{\overline{u}}$ and from assumption $D_{\overline{u}}$ is surjective. 
    
    To show the vertical component is surjective, we compute the Maslov index of the compactified vertical bundle, and then use \cite{riemHilb} to conclude surjectivity. From Theorem \ref{modelthm} and classification result of disks in \cite[\S 6.3]{Cha23} we can find a map $w:\bD^2\sm \{ 1\} \to \bC^k\sm \{ 0 \}$ such that $(\pi_{\bP^{k-1}} \of w) ^* \underbrace{\cO(-1)\sm{\underline{0}_{section}}}_{\cong \bC^k \sm \{ 0\}} \cong \overline{u}^* ({\bR \times Z}) $ as line bundles along with the Lagrangian boundary condition and $w \cong u$ as sections under the isomorphism. Note that we can choose $w$ such that the Maslov index of the projection $\pi_{\bP^{k-1}} \of w$ is at most $2\lceil\frac{l}{2}\rceil$. Thus, from the index computation in Lemma \ref{lemm:singlepuncindex} we see that the index of the vertical part $(D^v_w)^{comp}$ is at least $n+l - (n-1 + 2\lceil\frac{l}{2}\rceil  )$ , thus the Maslov index is at least $-1.$ Thus, we know that compactified vertical bundle $(D^v_u)^{comp}$ has Maslov index at least -1, and thus surjective from \cite{riemHilb}.

    \noindent 
    Finally, by using Proposition 6.16 \cite{PW19}, we have that the uncompactified $D^v_u$ is surjective.
    \end{proof}

\begin{lemma}\label{anypathhassinglepunct}
When $L_Y$ has minimal Maslov 2, then $\Lambda$ is Reeb fillable. See Definition \ref{def:reebfill}.
\end{lemma}

\begin{proof}

Without loss of generality, let us assume $c$ starts from $\Lambda$, i.e. $c(0) \in \Lambda,$ and $\pi_Y \of c = p$ where $\pi_Y : Z \to Y$ is the projection map. Let the length of $c$ be equal to $\frac{l\pi}{k}$, i.e. the length is $l$ times the length of the smallest possible Reeb chord length. 

    Let $$\overline{u}: (\bD^2,S^1) \to (\bP^{k-1}, T^{k-1}_{Cl})$$ be a $J_0$ holomorphic disk with boundary on the Clifford torus and of Maslov number $2\lfloor \frac{l}{2} \rfloor$.

    Let $\overline{u}_Y$ be a  Maslov $2\lfloor \frac{l}{2} \rfloor$ smooth disk with boundary on $L_Y$ passing through the point $p$, say $\overline{u}_Y (1)=p$. Such a disk exists from the simply connectedness of $Y$. There is a isomorphism as smooth $\bC^*$ bundles of the  bundle pair over $\bD^2$, $$(\overline{u}_Y^* (\bR \times Z), \partial \overline{u}_Y^* \cL_\gamma) \xlongleftrightarrow{\cong} (\overline{u}^* (\bR \times S^{2k-1}), \partial \overline{u}^* T_\gamma).$$

\noindent    From Proposition 6.18, Remark 6.19 and 6.20 and in \cite{Cha23} we see that there is a lift of the map $\overline{u}$ on  $\bD^2 \sm \{1 \} \cong \bH$, i.e. there is a map $$u: \bH \to (\bR \times S^{2k-1}, T_\gamma)$$ such that $\pi_{\bP^{k-1}} \of u = \overline{u}. $ Thus by an application of Theorem \ref{modelthm} we see that there is a single punctured disk $u_c$. 
\noindent  The only change needed when $c$ starts at $e^{i\frac{\pi}{k}}\Lambda$ is to consider a Maslov $2\lceil \frac{l}{2} \rceil$ disk in $\bP^{k-1}.$

\end{proof}

\noindent We now show that using Lemma \ref{anypathhassinglepunct}, we have an area-lowering result for SFT buildings where we keep the complex structure in the neck region $[a,b]\times Z$ to be equal to the cylindrical metric $J_{cyl}$ induced from $J_Y$. We collect some area computations in the inside piece below.

\begin{defn}[Energy of a building with no neck levels]
    Let $\mathfrak{u}$ be a building with no neck levels. We define the energy $E(\mathfrak{u})$ of such a building to be the sum of areas of the inside and outside components measured with the symplectic form obtained from symplectic cutting at hypersurface $Z$ along which we perform the neck stretching. We define the symplectic form $\omega_\inn$ and $\omega_{out}$ as the symplectic forms obtained from such a symplectic cut process.
\end{defn}

It is clear from the SFT compactness theorem that if $u_n$ is a sequence of holomorphic disks which converge to a building $\mathfrak{u}$ without any neck levels , then we have $$\lim_n E(u_n) = E(\mathfrak{u}).$$ 
\noindent

\begin{defn}[Gapped Reeb action]
    For a Reeb chord $c$ with end-points on $\Lambda_2 \subset \{s\} \times Z$, we define the gapped Reeb action $A_g(\gamma)$ as $$A_g(c) = f_\gamma(c(1)) - f_\gamma(c(0))+ \int_c e^s\lambda. $$ Here we use $f_\gamma$ to denote the primitive of $e^s\lambda$ as in Proposition \ref{prop:exact_in_neck}.
\end{defn}

\noindent  It is clear that the gapped Reeb action $A_g(c)$ measures the symplectic area of a disk with a single puncture with asymptote $c$. Moreover, existence of holomorphic punctured disks with minimal Reeb chord asymptotes implies that the gapped Reeb actions are positive, ie $A_g > 0.$

\begin{lemma}\label{lemm:area_comp_inside}
    Let $u_\inn$ be a smooth punctured disk with boundary punctures at $\{p_i \}_{i\in I}$ with corresponding Reeb chord data $\{ \gamma_i\}_{i\in I}$. Then we have the following area relation,
    \begin{equation}
        \int_{u_{in}}\omega_\inn  = \sum_{i\in I}A_g(\gamma_i)
    \end{equation}
\end{lemma}

\begin{proof}
    This is just an application of Stokes' theorem.
\end{proof}

\noindent As a corollary to the above lemma, we see that the symplectic area of a punctured disc increases as the number of  punctures going to Reeb asymptotes increases. We will need a special case of this result - 

\begin{lemma}\label{lemm:single_area_small}
Let $u_\inn$ be a punctured disc with boundary punctures asymptotic to Reeb chords $\{c_i\}_{i\in I}$. Fix an element $\wt i \in I$. Let $v_\inn$ be a \textit{smooth}  single punctured disk that asymptotes to $c_{\wt i}$. Then $$0<E(v_\inn) < E(u_\inn).$$

\end{lemma}
\begin{proof}
    The proof follows from positivity of $A_g$ and Lemma \ref{lemm:area_comp_inside}.
\end{proof}

\begin{lemma}\label{lemm:indexlowering}
For any broken disk $\fu$ with an interior disk $u_\inn$ with more than a single puncture, there is a smooth broken disk $\fu'$ of lower positive symplectic area, i.e. 

$$0<E(\fu') < E(\fu).$$
\end{lemma}

\begin{proof}
     Assume that the inner piece contains a disk $u_\inn$ with more punctures than just one boundary puncture. Denote the outside disk component that contains the marked point mapping to $p$ as $u_o$. Let $\wt c$ be a Reeb chord of $u_o$ which connects to a multiply punctured inner disk $u_i$. Let $\fu'$ be the smooth broken disk obtained by swapping $u_\inn$ with a single punctured disc $v_\inn$ asymptotic to $\wt c$. From Lemma \ref{lemm:single_area_small} we have that $0< E(\fu') < E(\fu').$

\end{proof}

\noindent In particular, if we choose a perturbation scheme $\cP_0$ which is fixed in the interior and makes lifts of regular disks  as done in Lemma \ref{lemm:fixedininterior}, we have that any broken disk using the perturbation scheme $\cP_0$ will have virtual dimension higher than 0 if it contains interior disk pieces with more than one puncture.

\begin{proof}[Proof of Theorem \ref{classifthm}]
First, note that a stable, rigid, broken disk cannot have a neck piece. It follows from the stability of the broken disk that neck pieces have a non-trivial strip or cylinder, and thus the $\bR-$translations of the neck piece stops such a broken disk from being regular.

    We can choose a regularizing perturbation scheme $\cP$, close to $\cP_0$ as done in Theorem \ref{thm:regularity} and Lemma \ref{lemm:fixedininterior} where an interior disk of any rigid broken disk has one boundary puncture. This follows from a Gromov compactness argument : If we couldn't find such a perturbation, then we would get a sequence of rigid broken disk whose limit will be a broken disk in the $\cP_0$ perturbation scheme. We can then perform a swapping argument as done in Lemma \ref{lemm:indexlowering} to find a smooth broken disk with lower symplectic area, which contradicts the monotonicity of $L$. Indeed, the broken disk $\fu$ and $\fu'$ can be smoothly glued to obtain disks $u,u'$ in $\pi_2(M,L)$. From the monotonicity assumption on $L$, since Maslov number of $u$ is 2, we know that $u$ has the smallest possible positive symplectic area. But sine $E(\fu') < E(\fu)$, we would have $u'$ to have a smaller symplectic area than $u$, hence we arrive at a contradiction. Thus, any disk that appears in the inside piece has a single boundary puncture. 
    
    Now index computation from Lemma \ref{lemm:singlepuncindex} shows that if the Reeb chord at the puncture is not minimal, the evaluation map at the puncture will have a positive dimensional kernel which would contradict the rigidity of the broken disk. The index computation can be done by a Gromov compactness argument again since Lemma \ref{lemm:singlepuncindex} was for holomorphic disks\footnote{We believe that there should be a purely topological argument for the index, but did not pursue such a proof}. We can find a perturbation scheme $\cP$ close to $\cP_0$ such that the index of a single puncture holomorphic disk is obtained from the formula in Lemma \ref{lemm:singlepuncindex}. If not, then there should be a sequence of holomorphic disks with a single puncture going to a Reeb chord which has action $l$ times the minimal action but has an index different from $n+l$. The sequence has an upper bound on symplectic area as $\cL_\gamma$ is exact and the area can be written in terms of the asymptotic Reeb chord and the relative homology class of the boundary map. Thus, we have a Gromov convergent subsequence whose limit is a map withindex $n+l$ which contradicts that the sequence consists of maps which do not have index $n+l$.  Thus, the inside piece can have only single punctured disks where the puncture is asymptotic to a positive or negative minimal Reeb chord.

    The second part of the Theorem, which relates the projection of the inside disks with the asymptotic Reeb chord at the boundary puncture, follows from shifting the problem to the specific case by using the model bundle pair result of Theorem \ref{modelthm}. The classification result in Theorem 6.5 of \cite{Cha23} proves this result in the specific case of $Z = S^{2n+1}, Y=\bC P^n$, $L_Y = T_{Cliff}$, and using Theorem \ref{modelthm} the result follows for any $Z,Y,L_Y$. 
\end{proof}
\section{Applications}\label{app}

In this section, we explore applications of Bohr-Sommerfeld-Profile surgery. The main set of examples where we can apply BSP surgery is Biran's circle bundle lifts. Finally, we show that we can use BSP-surgeries to construct exotic monotone Lagrangian tori in $\bP^n$.

\subsection{Exotic Lifts and Vianna tori} \label{sub:exonvia}

We can  use Biran's circle bundle to provide a large set of examples where one can perform a BSP surgery. We recall the construction of Lagrangian circle bundles from  \cite{b06,bc09}. Let $\Sigma \hookrightarrow M$ be an inclusion of compact monotone symplectic manifolds where $(M,\omega)$ is a Kahler manifold and $\omega$ is an integral symplectic form. Assume that $\Sigma$ satisfies $PD[\Sigma] = \omega$ (if the Poincare dual was $k\omega$ we  rescale $M,\omega$ to $(M,k\omega)$ ), note that this is a specialization of Biran's construction. Thus, the circle bundle of the normal bundle of $\Sigma$ can be taken to be a pre-quantum bundle $Z$ over $\Sigma.$ From \cite{b06} we have an embedding of a disk bundle $E_\Sigma$ modelled on the normal bundle of $\Sigma$ onto $M\sm \Delta$ where $\Delta$ is an isotropic CW complex. Proposition 6.4.1 \cite{bc09} proves that if $dim M \geq 6,$ any monotone Lagrangian $L_\Sigma \hookrightarrow \Sigma$ has a monotone lift $L_M$ in $M$.

\begin{prop}\label{prop:biranbundlehandle}
    Let $L_M$ be a Biran circle bundle lift from $L_\Sigma \hookrightarrow  \Sigma$. Assume that $L_\Sigma$ satisfies Assumption \ref{as1}. After an exact isotopy, we can assume that $L_M$ has a BSP Lagrangian handle modelled over $L_\Sigma.$ 
\end{prop}

\begin{figure}
    \makebox[\textwidth][c]{
   
        \def\svgscale{0.8
}
\begingroup%
  \makeatletter%
  \providecommand\color[2][]{%
    \errmessage{(Inkscape) Color is used for the text in Inkscape, but the package 'color.sty' is not loaded}%
    \renewcommand\color[2][]{}%
  }%
  \providecommand\transparent[1]{%
    \errmessage{(Inkscape) Transparency is used (non-zero) for the text in Inkscape, but the package 'transparent.sty' is not loaded}%
    \renewcommand\transparent[1]{}%
  }%
  \providecommand\rotatebox[2]{#2}%
  \newcommand*\fsize{\dimexpr\f@size pt\relax}%
  \newcommand*\lineheight[1]{\fontsize{\fsize}{#1\fsize}\selectfont}%
  \ifx\svgwidth\undefined%
    \setlength{\unitlength}{450bp}%
    \ifx\svgscale\undefined%
      \relax%
    \else%
      \setlength{\unitlength}{\unitlength * \real{\svgscale}}%
    \fi%
  \else%
    \setlength{\unitlength}{\svgwidth}%
  \fi%
  \global\let\svgwidth\undefined%
  \global\let\svgscale\undefined%
  \makeatother%
  \begin{picture}(1,0.5)%
    \lineheight{1}%
    \setlength\tabcolsep{0pt}%
    \put(0,0){\includegraphics[width=\unitlength,page=1]{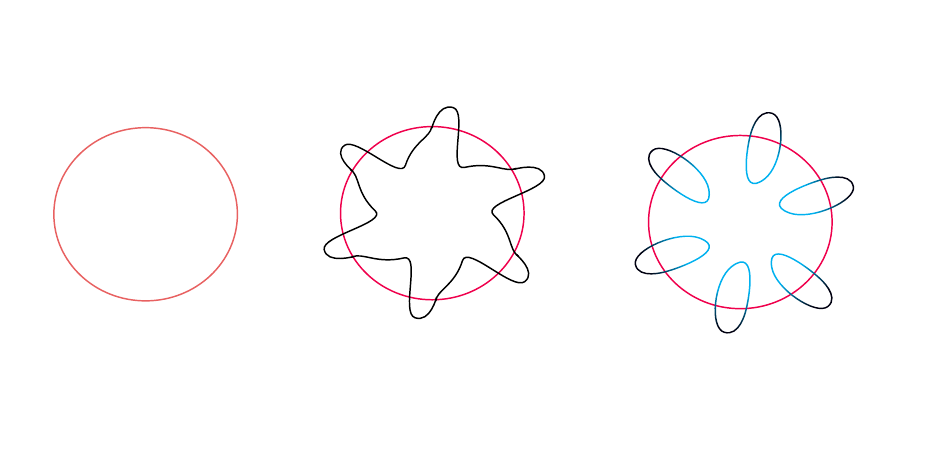}}%
    \put(0.11567757,0.12340249){\color[rgb]{0,0,0}\transparent{0.00784314}\makebox(0,0)[lt]{\lineheight{1.25}\smash{\begin{tabular}[t]{l}$L_M$\end{tabular}}}}%
    \put(0.41962883,0.12069975){\color[rgb]{0,0,0}\transparent{0.00784314}\makebox(0,0)[lt]{\lineheight{1.25}\smash{\begin{tabular}[t]{l}$L'$\end{tabular}}}}%
    \put(0.76210901,0.11326739){\color[rgb]{0,0,0}\transparent{0.00784314}\makebox(0,0)[lt]{\lineheight{1.25}\smash{\begin{tabular}[t]{l}$BSP_0(L')$\end{tabular}}}}%
  \end{picture}%
\endgroup%

 }
    \caption{A fiber slice of $L_M, L'$ and $BSP(L')$ respectively, i.e. intersection of them with $E_\Sigma|_p$ for $p\in L_\Sigma$.}
    \label{fig:fiberslice}
\end{figure}
\begin{remark}
    The fiber slice of $BSP(L')$ looks very similar to the monotone twist torus as defined in \cite{Chekanov_Schlenk_notes} which is actually $BSP(T_{Cliff})$. This similarity might suggest that a surgery with Bohr-Sommerfeld-profile is a fibered version of monotone twist construct as done in op. cit. - but this is merely conjectural and not pursued in this article.
\end{remark}

\begin{proof}
    Note that the restriction of the disk bundle $E_\Sigma $ to $L_\Sigma$ is a topologically trivial bundle with a flat connection induced from the connection on $\Z$. Recall that the holonomy around any loop on $L_\Sigma$ is multiplication by  $e^{\frac{i2l\pi}{k}}$ where $l<k$ and $k$ is the smallest integer satisfying Assumption \ref{as1}. Thus, we have that the holonomy can be realized by a $\bZ_k$ action where we view the $k$-th roots of unity as $\bZ_k$. If we choose any $\bZ_k$-invariant embedding of $S^1$ in a fiber $E_\Sigma|_p$ over any point $p$ on $L_\Sigma$, we can parallel transport it along $L_\Sigma$ to get a Lagrangian in $E_\Sigma.$ The Biran circle lift is obtained by taking parallel transport of a fixed radius circle centered at the origin over a fiber which is clearly invariant under the $\bZ_k$ action. We can perform an exact isotopy of the fixed radius circle through $\bZ_k$ equivariant embeddings, which induces an exact isotopy of $L_M$, see  figure \ref{fig:fiberslice}. By isotoping $L_M$ to $L'$ such that $L_M$ and $L'$ intersect transversely in each fiber of $L_\Sigma$, we see that $L'$ has a BSP handle modelled on $L_\Sigma.$

\end{proof}
   
    \begin{remark}
        
w    One should note that the disk bundle as constructed in \cite{bc09,b06} is a symplectic capping of the pre-quantum bundle $Z$.\footnote{We thank Dylan Cant for pointing out this sign issue.} The symplectic form on the bundle $E_\Sigma$ is $\omega_{can} = \pi^* \omega_\Sigma + d(r^2 \alpha)$ where $\alpha$ is a connection form such that $d\alpha= -\pi^* \omega_\Sigma$. By a change of co-ordinates, $r \mapsto \frac{1}{r}$, one sees that it is always possible to perform a Zero-area BSP surgery in this $E_\Sigma$.
    \end{remark}

\begin{defn}
    An \textit{exotic lift} of a Lagrangian $L_\Sigma \hookrightarrow \Sigma$ is defined as the BSP surgery, $BSP(L_M)$ , of the Biran lift $L_M$.
\end{defn}

Recall the notion of combinatorial mutation as introduced in \cite{minksum}. See \S \ref{subsec:geomrel} for the necessary definitions.

\begin{prop}\label{prop:newtstruct}
    $\newt(W_{BSP(\overline{T}_\abc)}) = \text{mut}_w({W_{\overline{T}_\abc}}$) for a width vector $w$ such that $h_{max}=n,h_{min}=-1.$ Moreover, $\newt(W_{\overline{T}_\abc} \cap \{ w = h_{max} \}$ is a point.
\end{prop}
\begin{proof}
    
From Theorem \ref{mainthm}, we see that $\newt(W_{BSP(\overline{T}_\abc)})$ is equal to mutation of $\newt(W_{\overline{T}_\abc})$ with a factor equal to a $n-1$ dimensional face of the $n-$simplex $\newt(W_{\overline{T}_\abc})$ which contains the triangle with affine edges of lengths $a,b,c.$ Since the factor for mutation is a codimension 1 face, the choice of $w$ is fixed up-to a sign. 

Now the rest of the proof follows from an induction on the Markov tree similar to the proof of Proposition 4.6 in \cite{chw}.   We first verify the statement for $\abc = (1,1,1)$ which is just the case of Clifford torus and higher mutation. Assume the statement is true for all $(a,b,c)$ in the Markov tree which are at distance $d$ from $(1,1,1)$. Thus, there is a width vector $w_\abc$ for which $h_{max}=n, h_{min}=-1.$ The induction step then follows from the fact that if $\abc$ and $(a',b',c')$ are Markov mutations, then $W_{\overline{T}_\abc}$ and $W_{\overline{T}_{(a',b',c')}}$ are related by mutating the triangular face corresponding to $T_\abc$. Note that even after such a mutation, the factor corresponding to the width vector $w_\abc$ is an $n-1$-simplex which has a triangular face with affine edges of lengths $a',b',c'.$ This finishes the proof.
\end{proof}

From Proposition \ref{prop:newtstruct} and Theorem \ref{mainthm}, we have the following structure of the Newton polytope of exotic lift of $\overline{T}_\abc$.

\begin{prop}\label{newpolytopes}
    Let $\wt T_\abc$ be an exotic lift of an $(n-1)$-dimensional lifted Vianna torus $\overline{T_\abc}$. Then $\newt(W_{\wt T_\abc})$ is an $n$ simplex with a vertex $v$ and $n-1$ dimensional face $F$ such that
    \begin{enumerate}
        \item All edges of $v$ have affine length $1$
        \item $F$ is equal to (up to affine equivalence) the Minkowski sum $n.\newt(W_{\Lambda^{n-1}_\abc})$. Here the notation $n.H$ refers to the iterated Minkowski sum $\underbrace{H + H + \dots}_{n \text{ times }}$
    \end{enumerate}
\end{prop}

\begin{proof}[Proof of Theorem \ref{morevianna}]
    From Proposition \ref{newpolytopes} we can compute the affine lengths of the edges in the Newton polytopes of $\wt T_\abc$. The lengths are listed below :
    \begin{itemize}
        \item a triangular face in $F$  of sides with affine lengths $na,nb$ and $nc$,
        \item all other edges in $F$ have affine length $n$,
        \item all edges from $v$ have affine length $1$.
    \end{itemize}
    From this description and the fact that the $Gl_n(\bZ)$ equivalence class of the Newton polyotope of disk-potential is a symplectomorphism invariant, we have that the set $$ \bigcup_{(a,b,c) \text{ is a  Markov triple}}\{ \overline{T}_\abc, BSP(\overline{T}_\abc) \}$$ consists of pairwise distinct (up to symplectomorphism) monotone Lagrangian tori in $\bP^n$. 
\end{proof}

\begin{remark}[Recipe to get $2^n$ lifts in $\bP^{2+n}$ starting from a single monotone Lagrangian in $\bP^2$.]
    Starting from a monotone Lagrangian in a symplectic divisor, there are two possible lifts - the Biran lift and its corresponding BSP surgery. Thus, there are $2^n$ choices of ways to (iterated) lift a monotone Lagrangian from $\bP^2$ to $\bP^{2+n}$. This potentially opens up $2^{n+2}$ distinct lifts of Vianna tori, but this would need a careful combinatorial analysis which we do not pursue in this article.
\end{remark}

\subsection{Conical Surgeries} \label{sub:conicalsurg}

We develop a generalization of Lagrangian disk surgery in this subsection. A conical Lagrangian in an $2n$-dimensional symplectic manifold is an embedding of a cone over an $n-1$-dimensional manifold which is a smooth Lagrangian embedding away from the conical point. We define a notion of a conical mutation configuration to be a  pair of Lagrangians $(L,C)$, where $L$ is a Lagrangian manifold and $C$ is a Lagrangian cone over a Bohr-Sommerfeld Legendrian $\Lambda$, which cleanly intersects $L$. We prove a Weinstein-neighborhood type theorem for such conical mutation configurations and show that we can perform a zero-area-BSP surgery.

For the rest of the subsection,   we will assume $\Lambda$ is a connected Bohr-Sommerfeld lift of a monotone Lagrangian $L_{\bP^n}$ in $\bC P^n$. Thus, $\Lambda$ is a Legendrian in the standard contact sphere $S^{2n+1}$. We define the \textit{standard cone over} $\Lambda$ to be the natural Lagrangian cone filling $C_s(\Lambda)$ of $\Lambda$ in the standard symplectic ball $B_r(0)$ (of any radius) viewed as a filling of $S^{2n+1}$. Recall that the topological cone of a manifold $K$ is the quotient space $$C(K) = K \times [0,1]/ \{ K \times 0\}.$$ The space $C(K)$ is a smooth manifold with boundary away from the \textit{conical point } $p:=[ K \times \{0\}]$.

\begin{defn}[Conical Lagrangian]
    A conical Lagrangian over the Bohr-Sommerfeld $\Lambda$ in the symplectic manifold $M$ is a topological embedding $\phi$ from the topological cone $C(\Lambda)$ to $M$ such that 
    \begin{enumerate}
        \item  $\phi: C(\Lambda) \to M$ is a smooth Lagrangian embedding away from the conical point $p$.
        \item  there is a Darboux embedding of the standard symplectic ball, $F: B_r(0) \to M $ such that $F(0) = p$ and $F\inv (\phi (C(\Lambda)) ) = C_s(\Lambda)$, i.e. near the conical point, the embedding $\phi$ is equivalent to an embedding of the standard cone over $\Lambda$ by a Darboux embedding.
    \end{enumerate}
\end{defn}

\begin{defn}[Conical Mutation Configuration]
    We define the pair $(L,C)$ to be a conical mutation configuration if \begin{enumerate}
        \item $L$ is a smooth Lagrangian 
        \item $C$ is a conical Lagrangian over $\Lambda$
        \item $L$ and $C$ only intersect at the boundary of the conical Lagrangian, $C$ and the intersection is clean.
    \end{enumerate}
\end{defn}

\begin{example}\label{modelcone}
    
A model example of a conical mutation configuration is $ (\wt L, C_s(\Lambda))$ in $\bC^{n+1}$ where $\wt L$ is the Biran lift of the Lagrangian $L_{\bP^n}$ and $C_s(\Lambda)$ is the standard cone over the Bohr-Sommerfeld lift of $L_{\bP^n}$.  See Figure \ref{conicalfig} for a diagram of intersection of $(L_{\bP^n},C_s(\Lambda))$ with a complex plane passing through the origin and intersecting $L_{\bP^n}\cap C_s(\Lambda).$
\end{example}

We will now prove a Weinstein-neighborhood type result, which proves that locally all mutation configurations depend solely on the Legendrian $\Lambda$. See \cite{karshon2025weinstein} for a very robust version of the Weinstein neighborhood theorem for stratified Lagrangians.

\begin{lemma}
    Given two conical mutation configurations $(L_i,C_i)$ where $C_i$'s are modelled over the same Bohr-Sommerfeld lift $\Lambda$ in two symplectic manifolds $M_i$ for $i= 1,2$, there is a pair of open neighborhoods $N_i$ of $C_i$ and a symplectomorphism $\Phi: N_1 \to N_2 $ such that $\Phi(C_1) = \Phi(C_2)$ and $\Phi(L_1 \cap N_1)  = L_2\cap N_2$.
\end{lemma}

\begin{proof}
    The proof is an usual application of Moser's trick after choosing diffeomorphisms which pullback the symplectic form $\omega_2$ to $\omega_1$ on a neighborhood of $C_1$. Note that since clean intersections are modelled by co-normal bundles, (see \cite[\S 3.2]{ciellat}) we can always extend the cones $C_i$ so that the intersection with $L_i$ happen in the interior of $C_i$. From the definition of conical Lagrangian, we can choose two symplectic embeddings $\phi_i$ of $B_r(0)$ (after potentially shrinking one of the balls) to $M_i$ which take the standard cone to the conical Lagrangian $C_i$. Thus, $\phi_2 \of \phi_1\inv$ is a symplectomorphism between neighborhoods  $U_i$ of the conical points of $C_1$ and $C_2$ which maps $C_1 \cap U_1 $ to $C_2 \cap U_2$. For clarity, let us identify $C_i \cap U_i$ with $\Lambda\times (0,\frac{1}{2})$. Note that there are two Weinstein neighborhoods $W_i$ of $C_i \cap U_i$. That is, there are symplectic embeddings $\cW_i: W_i \to T^*(\Lambda \times (0,\frac{1}{2}))$ which map $C_i\cap U_i$ to the zero section of the cotangent bundle. The map $\phi_2 \of \phi_1\inv$ induces a symplectomorphism of a neighborhood of the zero-section $\Lambda\times (0,\frac{1}{2})$, thus a symplectic automorphism $S$ of the tangent  bundle $T(T^*(\Lambda\times (0,\frac{1}{2})))|_{\Lambda\times (0,\frac{1}{2})}.$ Using an increasing  bijective smooth function $f:(0,\frac{1}{2})\to (0,1)$ which is equal to identity on $(0,1/2 - \varepsilon)$, we can extend the symplectomorphism $S$ to a symplectic automorphism $S_f$ of the tangent bundle $T(T^*(\Lambda\times (0,1)))$. Using the extension, $S_f$ we can construct a diffeomorphism $\varphi$ from a neighborhood of $C_1$ to $C_2$ such that the pullback $\varphi^*\omega_2$ agrees with $\omega_1$ over $C_1$ and agrees with the symplectomorphism induced from $\phi_2 \of \phi_1\inv $ over a shrunk ball of the conical point. We can now run a Moser isotopy to get a symplectomorphism. Note that the Moser's isotopy will be constant in a neighborhood of the conical point, as the pullback symplectic form $\varphi^*\omega_2$ agrees with $\omega_1$ near the conical point. Call this symplectomorphism $\wt\Phi$. Now, since $L_i$ and $C_i$ cleanly intersect, we can use a Hamiltonian isotopy argument similar to Lemma 3.1 in \cite{ciellat} to construct a symplectomorphism $\Phi$ satisfying the required properties.
\end{proof}

\begin{figure}
    \centering

        \def\svgscale{1}
\begingroup%
  \makeatletter%
  \providecommand\color[2][]{%
    \errmessage{(Inkscape) Color is used for the text in Inkscape, but the package 'color.sty' is not loaded}%
    \renewcommand\color[2][]{}%
  }%
  \providecommand\transparent[1]{%
    \errmessage{(Inkscape) Transparency is used (non-zero) for the text in Inkscape, but the package 'transparent.sty' is not loaded}%
    \renewcommand\transparent[1]{}%
  }%
  \providecommand\rotatebox[2]{#2}%
  \newcommand*\fsize{\dimexpr\f@size pt\relax}%
  \newcommand*\lineheight[1]{\fontsize{\fsize}{#1\fsize}\selectfont}%
  \ifx\svgwidth\undefined%
    \setlength{\unitlength}{450bp}%
    \ifx\svgscale\undefined%
      \relax%
    \else%
      \setlength{\unitlength}{\unitlength * \real{\svgscale}}%
    \fi%
  \else%
    \setlength{\unitlength}{\svgwidth}%
  \fi%
  \global\let\svgwidth\undefined%
  \global\let\svgscale\undefined%
  \makeatother%
  \begin{picture}(1,0.5)%
    \lineheight{1}%
    \setlength\tabcolsep{0pt}%
    \put(0,0){\includegraphics[width=\unitlength,page=1]{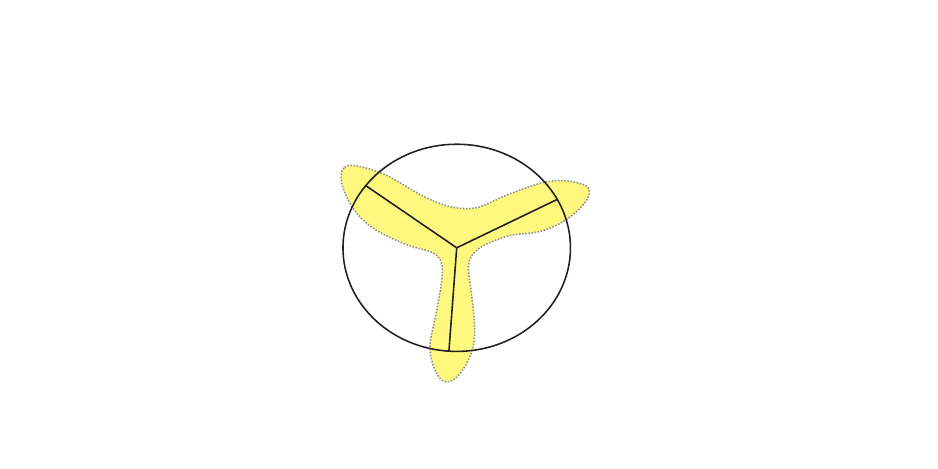}}%
    \put(0.61570633,0.26115239){\makebox(0,0)[lt]{\lineheight{1.25}\smash{\begin{tabular}[t]{l}$\cN$\end{tabular}}}}%
    \put(0.48791824,0.19261152){\makebox(0,0)[lt]{\lineheight{1.25}\smash{\begin{tabular}[t]{l}$C_s(\Lambda)$\end{tabular}}}}%
    \put(0.54623606,0.12081785){\makebox(0,0)[lt]{\lineheight{1.25}\smash{\begin{tabular}[t]{l}$\wt L$\end{tabular}}}}%
  \end{picture}%
\endgroup%

    \caption{Model Conical Mutation Configuration}
    \label{conicalfig}
\end{figure}

As an application of the previous lemma, we have that for any conical mutation configuration $(L,C)$ where $C$ is modelled over a Legendrian $\Lambda$, we have a symplectomorphism $\Phi$ from a neighborhood $\cN$ of the model cone in Example \ref{modelcone} such $\Phi(C_s(\Lambda)) = C$ and $\Phi(\wt L \cap \cN ) = \Phi(\cN) \cap L.$ Since $\cN$ contains a neighborhood of $0$, a direct check shows that after a Hamiltonian isotopy of $L$ we can ensure that $L$ has a BSP Lagrangian handle and allows a zero-area BSP surgery which we call a \textit{conical surgery.} 

\begin{defn}[Conical surgery]
      If $(L,C)$ is a conical mutation configuration,  the conical surgery $L_C$ of $L$ is obtained by performing a zero-area BSP surgery $BSP_\Lambda (L)$.
\end{defn}

\noindent Recall that one Bohr-Sommerfeld lift of the real projective space is the real Legendrian sphere, $S^{n-1}_\bR \subset S^{2n-1}.$ Thus, in this special case, the cone $C$ is a Lagrangian ball which cleanly intersects $L$ at its spherical boundary. The conical surgery of $L$ along $C$ is thus equivalent to the composition of an antisurgery as introduced by \cite{haug} and a Polterovich surgery. The following proposition is a special case of the main theorem \ref{mainthm}.

\begin{prop}
  The effect of conical surgery on disk potential is given by the relation:
  $$ W_{BSP^\Lambda(L)}(x_1,\dots,x_k,z,w_1,\dots,w_l)
 =W_L(x_1,\dots,x_k,zW_{\Lambda}(x_1,\dots,x_k),w_1,\dots,w_l). $$
\end{prop}

\bibliographystyle{amsplain}
\bibliography{bib}

\begin{bibdiv}
\begin{biblist}

\bib{Abbasbook}{book}{
      author={Abbas, Casim},
       title={An {I}ntroduction to {C}ompactness results in {S}ymplectic {F}ield {T}heory},
    language={en},
     edition={2014},
   publisher={Springer},
     address={Berlin, Germany},
        date={2014},
}

\bib{minksum}{article}{
      author={Akhtar, Mohammad},
      author={Coates, Tom},
      author={Galkin, Sergey},
      author={Kasprzyk, Alexander~M},
      author={others},
       title={Minkowski polynomials and mutations},
        date={2012},
     journal={SIGMA. Symmetry, Integrability and Geometry: Methods and Applications},
      volume={8},
       pages={094},
}

\bib{aurmohgay:symplchyp}{article}{
      author={Auroux, Denis},
      author={Gayet, Damien},
      author={Mohsen, Jean-Paul},
       title={Symplectic hypersurfaces in the complement of an isotropic submanifold},
        date={2001},
     journal={Mathematische Annalen},
      volume={321},
      number={4},
       pages={739\ndash 754},
}

\bib{Aur07}{article}{
      author={Auroux, Denis},
       title={Mirror symmetry and {$T$}-duality in the complement of an anticanonical divisor},
        date={2007},
        ISSN={1935-2565},
     journal={J. G\"{o}kova Geom. Topol. GGT},
      volume={1},
       pages={51\ndash 91},
      review={\MR{2386535}},
}

\bib{bc09}{article}{
      author={Biran, Paul},
      author={Cornea, Octav},
       title={Rigidity and uniruling for {L}agrangian submanifolds},
        date={2009},
     journal={Geometry \& Topology},
      volume={13},
      number={5},
       pages={2881\ndash 2989},
}

\bib{bcsw1}{article}{
      author={Blakey, Kenneth},
      author={Chanda, Soham},
      author={Sun, Yuhan},
      author={Woodward, Chris~T.},
       title={Augmentation varieties and disk potentials i},
        date={2024},
      eprint={2310.17821},
         url={https://arxiv.org/abs/2310.17821},
}

\bib{bcsw2}{misc}{
      author={Blakey, Kenneth},
      author={Chanda, Soham},
      author={Sun, Yuhan},
      author={Woodward, Chris~T.},
       title={Augmentation varieties and disk potentials {II}},
        date={2024},
         url={https://arxiv.org/abs/2401.13021},
}

\bib{bcsw3}{article}{
      author={Blakey, Kenneth},
      author={Chanda, Soham},
      author={Sun, Yuhan},
      author={Woodward, Chris~T},
       title={Augmentation varieties and disk potentials iii},
        date={2024},
     journal={arXiv preprint 2401.13024},
}

\bib{BEH03}{article}{
      author={Bourgeois, F.},
      author={Eliashberg, Y.},
      author={Hofer, H.},
      author={Wysocki, K.},
      author={Zehnder, E.},
       title={Compactness results in symplectic field theory},
        date={2003},
        ISSN={1465-3060},
     journal={Geom. Topol.},
      volume={7},
       pages={799\ndash 888},
         url={https://doi-org.ezp.lib.cam.ac.uk/10.2140/gt.2003.7.799},
      review={\MR{2026549}},
}

\bib{Biran_lagrangina_barrier}{article}{
      author={Biran, P.},
       title={Lagrangian barriers and symplectic embeddings},
        date={2001},
        ISSN={1016-443X,1420-8970},
     journal={Geom. Funct. Anal.},
      volume={11},
      number={3},
       pages={407\ndash 464},
         url={https://doi.org/10.1007/PL00001678},
      review={\MR{1844078}},
}

\bib{b06}{article}{
      author={Biran, P},
       title={Lagrangian non-intersections},
        date={2006},
     journal={Geometric \& Functional Analysis GAFA},
      volume={16},
      number={2},
       pages={279\ndash 326},
}

\bib{bk13}{article}{
      author={Biran, Paul},
      author={Khanevsky, Michael},
       title={A floer--gysin exact sequence for {L}agrangian submanifolds},
        date={2013},
     journal={Commentarii mathematici Helvetici},
      volume={88},
      number={4},
       pages={899\ndash 952},
}

\bib{cant2022dimension}{book}{
      author={Cant, Dylan~Jesse},
       title={A dimension formula for relative symplectic field theory},
   publisher={Stanford University},
        date={2022},
}

\bib{ciellat}{article}{
      author={Cieliebak, Kai},
      author={Ekholm, Tobias},
      author={Latschev, Janko},
       title={Compactness for holomorphic curves with switching {L}agrangian boundary conditions},
        date={2010},
     journal={The Journal of Symplectic Geometry},
      volume={8},
      number={3},
       pages={267\ndash 298},
}

\bib{cginfty}{article}{
      author={Casals, Roger},
      author={Gao, Honghao},
       title={Infinitely many {L}agrangian fillings},
        date={2022},
     journal={Annals of Mathematics},
      volume={195},
      number={1},
       pages={207\ndash 249},
}

\bib{Cha23}{misc}{
      author={Chanda, Soham},
       title={Floer {C}ohomology and {H}igher {M}utations},
        date={2023},
        note={arXiv:2301.08311},
}

\bib{chw}{article}{
      author={Chanda, Soham},
      author={Hirschi, Amanda},
      author={Wang, Luya},
       title={Infinitely many monotone {L}agrangian tori in higher projective spaces},
        date={2023},
     journal={arXiv preprint arXiv:2307.06934},
}

\bib{21coates}{article}{
      author={Coates, Tom},
      author={Kasprzyk, Alexander~M},
      author={Pitton, Giuseppe},
      author={Tveiten, Ketil},
       title={Maximally mutable laurent polynomials},
        date={2021},
     journal={Proceedings of the Royal Society A},
      volume={477},
      number={2254},
       pages={20210584},
}

\bib{CM05}{incollection}{
      author={Cieliebak, K.},
      author={Mohnke, K.},
       title={Compactness for punctured holomorphic curves},
        date={2005},
      volume={3},
       pages={589\ndash 654},
         url={http://projecteuclid.org.ezp.lib.cam.ac.uk/euclid.jsg/1154467631},
        note={Conference on Symplectic Topology},
      review={\MR{2235856}},
}

\bib{cm07}{article}{
      author={Cieliebak, Kai},
      author={Mohnke, Klaus},
       title={Symplectic hypersurfaces and transversality in {G}romov-{W}itten theory},
        date={2007},
     journal={Journal of Symplectic Geometry},
      volume={5},
      number={3},
       pages={281\ndash 356},
}

\bib{cruzgalkin}{article}{
      author={Cruz~Morales, John~Alexander},
      author={Galkin, Sergey},
      author={others},
       title={Upper bounds for mutations of potentials},
        date={2013},
     journal={SIGMA. Symmetry, Integrability and Geometry: Methods and Applications},
      volume={9},
       pages={005},
}

\bib{Chekanov_Schlenk_notes}{article}{
      author={Chekanov, Yuri},
      author={Schlenk, Felix},
       title={Notes on monotone {L}agrangian twist tori},
        date={2010},
        ISSN={1935-9179},
     journal={Electron. Res. Announc. Math. Sci.},
      volume={17},
       pages={104\ndash 121},
         url={https://doi.org/10.3934/era.2010.17.104},
      review={\MR{2735030}},
}

\bib{cwstab}{article}{
      author={Charest, Fran{\c{c}}ois},
      author={Woodward, Chris},
       title={Floer trajectories and stabilizing divisors},
        date={2017},
     journal={Journal of Fixed Point Theory and Applications},
      volume={19},
       pages={1165\ndash 1236},
}

\bib{floflip}{article}{
      author={Charest, Fran{\c c}ois},
      author={Woodward, {Chris T.}},
       title={Floer {C}ohomology and {F}lips},
    language={English (US)},
        date={2022},
        ISSN={0065-9266},
     journal={Memoirs of the American Mathematical Society},
      volume={279},
      number={1375},
}

\bib{cwmicro}{article}{
      author={Casals, Roger},
      author={Weng, Daping},
       title={Microlocal theory of legendrian links and cluster algebras},
        date={2024},
     journal={Geometry \& Topology},
      volume={28},
      number={2},
       pages={901\ndash 1000},
}

\bib{donaldson:symplsub}{article}{
      author={Donaldson, Simon~K.},
       title={Symplectic submanifolds and almost-complex geometry},
        date={1996},
     journal={Journal of Differential Geometry},
      volume={44},
      number={4},
       pages={666\ndash 705},
}

\bib{DTVW}{unpublished}{
      author={Diogo, Lu\'{i}s},
      author={Tonkonog, Dmitry},
      author={Vianna, Renato},
      author={Wu, Weiwei},
       title={Lifting {L}agrangians from {D}onaldson-type divisors},
        note={In preparation},
}

\bib{egh00}{article}{
      author={Eliashberg, Yakov},
      author={Glvental, A},
      author={Hofer, Helmut},
       title={Introduction to symplectic field theory},
        date={2010},
     journal={Visions in Mathematics: GAFA 2000 Special Volume, Part II},
       pages={560\ndash 673},
}

\bib{ch10}{misc}{
      author={Fukaya-Oh-Ohta-Ono},
       title={Lagrangian intersection {F}loer theory - anomaly and obstruction - chapter 10},
        date={2007},
}

\bib{haug}{article}{
      author={Haug, Luis},
       title={Lagrangian antisurgery},
        date={2020},
        ISSN={1945-001X},
     journal={Mathematical Research Letters},
      volume={27},
      number={5},
       pages={1423–1464},
         url={http://dx.doi.org/10.4310/MRL.2020.v27.n5.a7},
}

\bib{karshon2025weinstein}{article}{
      author={Karshon, Yael},
      author={Tukachinsky, Sara~B},
      author={Zimhony, Yoav},
       title={Weinstein neighbourhood theorems for stratified subspaces},
        date={2025},
     journal={arXiv preprint arXiv:2507.04897},
}

\bib{lm10}{article}{
      author={Lekili, Yank{\i}},
      author={Maydanskiy, Maksim},
       title={The symplectic topology of some rational homology balls},
    language={en},
        date={2014},
     journal={Comment. Math. Helv.},
      volume={89},
      number={3},
       pages={571\ndash 596},
}

\bib{makwu}{article}{
      author={Mak, Cheuk~Yu},
      author={Wu, Weiwei},
       title={Dehn twist exact sequences through lagrangian cobordism},
        date={2018},
     journal={Compositio Mathematica},
      volume={154},
      number={12},
       pages={2485–2533},
}

\bib{mw18}{article}{
      author={Mak, Cheuk~Yu},
      author={Wu, Weiwei},
       title={Spherical twists and {L}agrangian spherical manifolds},
        date={2018},
     journal={arXiv preprint arXiv:1810.06533},
}

\bib{mw19}{article}{
      author={Mak, Cheuk~Yu},
      author={Wu, Weiwei},
       title={Dehn twists and {L}agrangian spherical manifolds},
        date={2019},
     journal={Selecta Mathematica},
      volume={25},
       pages={1\ndash 85},
}

\bib{riemHilb}{article}{
      author={Oh, Yong-Geun},
       title={Riemann-{H}ilbert problem and application to the perturbation theory of analytic discs},
        date={1995},
     journal={Kyungpook Mathematical Journal},
      volume={35},
      number={1},
       pages={39\ndash 75},
}

\bib{pol01}{book}{
      author={Polterovich, Leonid},
       title={The geometry of the group of symplectic diffeomorphism},
    language={en},
     edition={2001},
      series={Lectures in Mathematics. ETH Z{\"u}rich},
   publisher={Springer},
     address={Basel, Switzerland},
        date={2001},
}

\bib{pol:surg}{article}{
      author={Polterovich, Leonid},
       title={The surgery of lagrange submanifolds},
        date={1991},
     journal={Geometric \& Functional Analysis GAFA},
      volume={1},
      number={2},
       pages={198\ndash 210},
}

\bib{PT20}{article}{
      author={Pascaleff, James},
      author={Tonkonog, Dmitry},
       title={The wall-crossing formula and {L}agrangian mutations},
        date={2020-02},
        ISSN={0001-8708},
     journal={Advances in Mathematics},
      volume={361},
}

\bib{PW19}{misc}{
      author={Palmer, Joseph},
      author={Woodward, Chris},
       title={Invariance of immersed {F}loer cohomology under {L}agrangian surgery},
        date={2019},
}

\bib{rg19}{article}{
      author={Rizell, Georgios~Dimitroglou},
      author={Golovko, Roman},
       title={Legendrian submanifolds from bohr-sommerfeld covers of monotone lagrangian tori},
        date={2019},
     journal={arXiv preprint arXiv:1901.08415},
}

\bib{seidelFuk}{book}{
      author={Seidel, Paul},
       title={Fukaya categories and {P}icard-{L}efschetz theory},
   publisher={European Mathematical Society},
        date={2008},
      volume={10},
}

\bib{stw2}{article}{
      author={Shende, Vivek},
      author={Treumann, David},
      author={Williams, Harold},
       title={On the combinatorics of exact {L}agrangian surfaces},
        date={2016},
     journal={arXiv preprint arXiv:1603.07449},
}

\bib{stw1}{article}{
      author={Shende, Vivek},
      author={Treumann, David},
      author={Williams, Harold},
      author={Zaslow, Eric},
       title={Cluster varieties from {L}egendrian knots},
        date={2019},
     journal={Duke Mathematical Journal},
      volume={168},
      number={15},
}

\bib{Via16}{article}{
      author={Vianna, Renato Ferreira de~Velloso},
       title={Infinitely many exotic monotone {L}agrangian tori in {$\Bbb{CP}^2$}},
        date={2016},
        ISSN={1753-8416},
     journal={J. Topol.},
      volume={9},
      number={2},
       pages={535\ndash 551},
         url={https://doi-org.ezp.lib.cam.ac.uk/10.1112/jtopol/jtw002},
      review={\MR{3509972}},
}

\bib{Via17}{article}{
      author={Vianna, Renato},
       title={Infinitely many monotone {L}agrangian tori in del {P}ezzo surfaces},
        date={2017},
     journal={Selecta Mathematica},
      volume={23},
       pages={1955\ndash 1996},
}

\bib{JholoWendl}{article}{
      author={Wendl, Chris},
       title={Lectures on holomorphic curves in symplectic and contact geometry},
        date={2010},
     journal={arXiv preprint arXiv:1011.1690},
}

\bib{wendl:sft}{misc}{
      author={Wendl, Chris},
       title={Lectures on symplectic field theory},
        note={\url{https://www.mathematik.hu-berlin.de/~wendl/pub/SFTlectures_book3.pdf}},
}

\bib{yau17}{misc}{
      author={Yau, Mei-Lin},
       title={Surgery and isotopy of {L}agrangian surfaces},
        date={2017},
}

\bib{Y22}{article}{
      author={Yuan, Hang},
       title={Disk counting and wall-crossing phenomenon via family {F}loer theory},
        date={2022},
     journal={Journal of Fixed Point Theory and Applications},
      volume={24},
      number={4},
       pages={77},
}

\end{biblist}
\end{bibdiv}

\Addresses

\end{document}